\documentclass[oneside]{amsart}
\usepackage{amsmath}
\usepackage{amsfonts}
\usepackage{amsthm}
\usepackage{amssymb}
\usepackage{hyperref}
\usepackage{amscd}
\usepackage{color}
\usepackage{enumerate}
\usepackage{xypic}
\usepackage{tikz}
\usepackage{tikz-cd}
\usetikzlibrary{arrows,automata}
\usepackage{colortbl}
\usepackage[normalem]{ulem}
\usepackage{multirow}
\usepackage{makecell}
\usepackage{enumitem}

\usepackage{booktabs}

\usepackage{dsfont}

\makeatletter
\def\l@subsection{\@tocline{2}{0pt}{2.5pc}{5pc}{}}
\makeatother

\newcommand{\RTT}{\mathfrak{T}}
\newcommand{\TT}{\mathbb{T}}

\newcommand{\ZZ}{\mathbb{Z}}
\newcommand{\KK}{\mathbb{K}}

\newcommand{\RR}{\mathbb{R}}
\newcommand{\PP}{\mathbb{P}}

\newcommand{\QQ}{\mathbb{Q}}

\newcommand{\C}{\mathcal{C}}
\newcommand{\K}{\mathcal{K}}
\newcommand{\T}{\mathcal{T}}
\newcommand{\V}{\mathcal{V}}
\newcommand{\W}{\mathcal{W}}
\newcommand{\cY}{\mathcal{Y}}
\newcommand{\PS}{\mathfrak{P}}
\newcommand{\ba}{\mathbf{a}}
\newcommand{\bb}{\mathbf{b}}
\newcommand{\bc}{\mathbf{c}}

\newcommand{\bx}{\mathbf{x}}
\newcommand{\bt}{\mathbf{t}}
\newcommand{\bs}{\mathbf{s}}
\newcommand{\CTC}{\mathbb{L}_{\bullet}^{S/\KK}}
\newcommand{\Aff}{\mathbb{A}}
\newcommand{\mfm}{\mathfrak{m}}
\newcommand{\SG}{\mathfrak{S}}
\newcommand{\AOP}{\A_\circ^\univ}

\newcommand{\Hilb}{\mathcal{H}}
\newcommand{\U}{\mathcal{U}}
\newcommand{\A}{\mathcal{A}}
\newcommand{\F}{\mathcal{F}}
\newcommand{\tB}{\widetilde{B}}

\newcommand{\hy}{\hat{y}}
\newcommand{\tx}{\widetilde{\bf x}}

\newcommand{\bg}{{\mathbf g}}
\newcommand{\bz}{{\mathbf{z}}}

\newcommand{\uCG}{\underline{\mathcal{C} G}}

\newcommand{\univ}{\mathrm{univ}}
\newcommand{\exr}{\mathrm{ex}}
\newcommand{\prin}{\mathrm{prin}}
\newcommand{\tr}{\mathrm{tr}}

\newcommand{\cA}{\mathfrak{A}}

\DeclareMathOperator{\link}{lk}

\DeclareMathOperator{\spec}{Spec}

\DeclareMathOperator{\Gr}{Gr}

\DeclareMathOperator{\Hom}{Hom}
\DeclareMathOperator{\supp}{supp}
\DeclareMathOperator{\Mat}{Mat}
\DeclareMathOperator{\Ext}{Ext}
\DeclareMathOperator{\Der}{Der}
\DeclareMathOperator{\proj}{Proj}
\DeclareMathOperator{\Aut}{Aut}

\newtheorem*{thm*}{Theorem}
\newtheorem{thm}[equation]{Theorem}
\newtheorem{prop}[equation]{Proposition}
\newtheorem{lemma}[equation]{Lemma}
\newtheorem{cor}[equation]{Corollary}

\newtheorem{conj}[equation]{Conjecture}

\theoremstyle{definition}

\newtheorem{defn}[equation]{Definition}

\newtheorem{rem}[equation]{Remark}
\newtheorem{warning}[equation]{Warning}
\newtheorem{oldex}[equation]{Example}

\newenvironment{ex}
  {\pushQED{\qed}
  \oldex}
  {\popQED\endoldex}

\numberwithin{equation}{subsection}

\newcolumntype{g}{>{\columncolor{yellow}}c}
\title{Deformation Theory for Finite Cluster Complexes}

\author{Nathan Ilten}
\address{Department of Mathematics, Simon Fraser University,
8888 University Drive, Burnaby BC V5A1S6, Canada}
\email{\href{mailto:nilten@sfu.ca}{nilten@sfu.ca}}

\author{Alfredo N\'ajera Ch\'avez}
\address{Instituto de Matem\'aticas, Universidad Nacional Aut\'onoma de M\'exico,
Le\'on 2, altos, Oaxaca de Ju\'arez, Centro Hist\'orico, 68000 Oaxaca, Mexico}
\email{\href{mailto:najera@im.unam.mx}{najera@im.unam.mx}}

\author{Hipolito Treffinger}
\address{Departamento de Matemática, Facultad de Ciencias Exactas y Naturales,
Universidad de Buenos Aires, and IMAS-CONICET, Pabellon I – Ciudad
Universitaria, Buenos Aires, 1428, Argentina}
\email{\href{mailto:htreffinger@dm.uba.ar}{htreffinger@dm.uba.ar}}

\begin{document}

\maketitle

\begin{abstract}
	We study the deformation theory of the Stanley-Reisner rings associated to cluster complexes for skew-symmetrizable cluster algebras of geometric and finite cluster type. In particular, we show that in the skew-symmetric case, these cluster complexes are unobstructed, generalizing a result of Ilten and Christophersen in the $A_n$ case. 

	We also study the connection between cluster algebras with universal coefficients and cluster complexes.
We show that for a full rank positively graded cluster algebra $\A$ of geometric and finite cluster type, the cluster algebra $\A^\univ$ with universal coefficients may be recovered as the universal family over a partial closure of a torus orbit in a multigraded Hilbert scheme.
Likewise, we show that under suitable hypotheses, the cluster algebra $\A^\univ$ may be recovered as the coordinate ring for a certain torus-invariant semiuniversal deformation of the Stanley-Reisner ring of the cluster complex.

We apply these results to show that for any cluster algebra $\A$ of geometric and finite cluster type, $\A$ is Gorenstein, and $\A$ is unobstructed if it is skew-symmetric. Moreover, if $\A$ has  enough frozen variables then it has no non-trivial torus-invariant deformations. We also study the Gr\"obner theory of the ideal of relations among cluster and frozen variables of $\A$. As a byproduct we generalize previous results in this setting obtained by Bossinger, Mohammadi and Nájera Chávez for Grassmannians of planes and $\Gr(3,6)$.
\end{abstract}
\tableofcontents
\section{Introduction}
\subsection{Unobstructed Simplicial Complexes}
The \emph{dual $n$-associahedron} is the simplicial complex whose vertices are the diagonals in a regular $(n+1)$-gon, and whose faces are given by collections of pairwise non-crossing diagonals. This simplicial complex $\K$ is dual to the so-called Stasheff polytope.
Ilten and Christophersen prove in \cite{srdegen} that the dual associahedron $\K$ has the remarkable property of being \emph{unobstructed}, that is, the second cotangent cohomology $T^2(S_\K)$ is zero, where $S_\K$ is the Stanley-Reisner ring of $\K$. See \S \ref{sec:SR} and \S \ref{sec:cotangent} for definitions.

In general,  when a simplicial complex $\K$ is unobstructed, the projective scheme $\proj S_\K$ corresponds to a smooth point of the Hilbert scheme parametrizing projective schemes with the same Hilbert polynomial as $S_\K$. This can be a very useful tool for studying Hilbert schemes, see for example \cite{crelle}.
The desire to better understand \emph{why} the dual associahedron is unobstructed, and to find other natural examples of unobstructed simplicial complexes, is the starting point for this paper.

\subsection{Unobstructed Cluster Complexes}
In \cite{FZ_assoc}, Fomin and Zelevinsky associate to any root system $\Phi$ a simplicial complex $\Delta(\Phi)$. This generalizes the dual $n$-associahedron introduced above, as it corresponds exactly to the root system of type $A_{n-2}$. It is thus natural to ask if $\Delta(\Phi)$ is unobstructed for other choices of root systems $\Phi$.

In fact, the simplicial complexes $\Delta(\Phi)$ occur naturally in the theory of cluster algebras. Also introduced by Fomin and Zelevinksy, a \emph{cluster algebra} is a special kind of commutative algebra constructed via an iterative combinatorial process. See \S \ref{sec:cluster} for definitions, notation, and conventions for cluster algebras. In addition to its algebraic structure, every cluster algebra $\A$ is endowed with the combinatorial information of a collection of finite subsets called \emph{clusters}; these can be encoded in a simplicial complex $\K$ called the \emph{cluster complex} of $\A$, see  \S \ref{sec:complex}. Of special interest here is the case when $\A$ only has finitely many clusters, in which case it is said to be of finite cluster type. It turns out that, modulo choice of coefficients, cluster algebras of finite cluster type are classified exactly by root systems (or equivalently Dynkin diagrams) \cite{FZ_clustersII}, and that in these cases, the cluster complex $\K$ is exactly the simplicial complex $\Delta(\Phi)$ of the associated root system \cite{FZ_assoc}. See \S \ref{sec:cluster} and \S \ref{sec:complex} for details.

The perspective we take in this paper is thus that the dual $n$-associahedron is the \emph{cluster complex} for any cluster algebra of type $A_{n-2}$.
Our first result is the following:
\begin{thm}[Theorem \ref{thm:unobstructed} and Remark \ref{rem:obstructed}]
Let $\A$ be a cluster algebra of finite cluster type. Then its cluster complex $\K$ is unobstructed if and only if $\A$ is skew-symmetric.
\end{thm}
\noindent Skew-symmetric cluster algebras of finite cluster type are exactly those corresponding to simply laced Dynkin diagrams.
 In particular, we recover that the dual associahedron is unobstructed, making use of ``cluster theoretic'' arguments.
 We believe the most natural proof is the one we present here, which makes use of categorification. However, it is also possible to prove the theorem using purely combinatorial arguments. The proof of \cite{srdegen} for the type $A_n$ case may be adapted to the $D_n$ case using the combinatorial model of \cite{FZ_clustersII}, leaving a finite number of exceptional cases to check.

In addition to the above unobstructedness result, we are also able to give a concrete description of the space of first order deformations $T^1(S_\K)$ for $\K$ the cluster complex of any skew-symmetrizable cluster algebra $\A$ of finite cluster type.
See Theorem \ref{thm:t1}. 

\subsection{Universal Coefficients}
The remaining results of the paper revolve around a fundamental geometric relationship that we will establish between a cluster algebra $\A$ of finite cluster type and the Stanley-Reisner ring $S_\K$ of its cluster complex $\K$.
We will always be working over an algebraically closed field $\KK$ of characteristic zero.
Recall that to any cluster algebra of finite cluster type $\A$, there is an associated cluster algebra $\A^\univ$ with \emph{universal coefficients}, see Definition \ref{def:univ frozen}. The algebra $\A^\univ$ is an algebra over a polynomial ring $\KK[t_1,\ldots,t_p]$, where $p$ is the number of cluster variables of $\A$. Geometrically, this corresponds to a family \[\spec (\A^\univ) \to \Aff^p.\]

This leads to our second result:
\begin{thm}[Theorem \ref{thm:hilb}\ref{part:h2}]
Let $\A$ be a cluster algebra of finite cluster type. The fiber over zero of $\spec \A^\univ\to \Aff^p$ is the scheme cut out by the Stanley-Reisner ideal of the cluster complex $\K$.
\end{thm}
\noindent A key ingredient for this theorem is the folklore result that for cluster algebras of finite type (with non-invertible frozen variables), the extended cluster monomials form a basis (see Theorem \ref{thm:basis}). This statement is a consequence of the machinery of \cite{GHKK}.

Under some additional assumptions on $\A$, we are able to give new interpretations of the cluster algebra $\A^\univ$. The first shows that under appropriate hypotheses, $\A^\univ$ may be constructed from $\A$ by essentially acting on $\A$ by a torus:
\begin{thm}[Theorem \ref{thm:hilb}\ref{part:h4}]
Let $\A$ be a positively graded cluster algebra of finite cluster type and full rank. Then the family $\spec \A^\univ\to \Aff^p$ arises as the restriction of the universal family for a multigraded Hilbert scheme to a partial closure of the torus orbit of the point corresponding to the algebra $\A$.
\end{thm}
In the special cases when $\A$ is the coordinate ring of the Grassmannian $\Gr(2,n)$  $(n\geq 4)$ or $\Gr(3,6)$, Bossinger, Mohammadi, and N\'ajera Ch\'avez \cite{BMN} show that $A^\univ$ may be obtained via a particular construction from Gr\"obner theory. Our theorem is a generalization of this, as explained in Remark \ref{rem:BMN}.

In our second interpretation of $\A^\univ$, we show that under appropriate hypotheses, $\A^\univ$ can be reconstructed from the cluster complex $\K$ (and a little extra data) via deformation theory. We say that a cluster algebra $\A$ has \emph{enough} frozen variables if property T1 is satisfied, see Definition \ref{defn:T1}.
\begin{thm}[Theorem \ref{thm:versal}]\label{thm:iversal}
Let $\A$ be a cluster algebra of finite cluster type with $p$ cluster and $q$ frozen variables, and assume it has enough frozen variables. 
Let $\TT$ be the torus of the ambient space for the cluster embedding $\spec \A\subseteq \Aff^{p+q}$, and $H\subseteq \TT$ the maximal subgroup fixing $\spec \A$.
Then there is a canonical $\TT$-equivariant $H$-invariant semiuniversal deformation $\cY\to \Aff^p$ of $(\spec S_{\K})\times \Aff^q$, and it may be canonically identified with the family $\spec \A^\univ\to \Aff^p$.
\end{thm}
\noindent
The slightly more general statement of Theorem \ref{thm:versal} may be applied even when $\A$ does not have ``enough'' frozen variables. See \S\ref{ex:g2} for an example illustrating the reconstruction of $\A^{\univ}$ from $\K$.

\subsection{Applications}
We may apply the theorems of the previous subsection to obtain new information about some algebraic properties of a cluster algebra of finite cluster type.

\begin{cor}[Corollary \ref{cor:gorenstein}]
Let $\A$ be a cluster algebra of finite cluster type. Then $\A$ is Gorenstein.
\end{cor}
\noindent
This was previously known in the case that $\A$ has no frozen variables, see \cite[Corollary 1.21]{BFZ_clustersIII} and the discussion following Corollary \ref{cor:gorenstein}.
We note here that we are following the convention of \emph{not} inverting frozen variables, see Remark \ref{rem:invert}.

\begin{cor}[Corollaries \ref{cor:t2} and \ref{cor:t1}]
	Let $\A$ be a cluster algebra of finite cluster type. 
	\begin{enumerate}
		\item If $\A$ is skew-symmetric, then $T^2(\A)=0$. 
		\item If $\A$ has full rank and satisfies the T1 property, then $T^1(\A)^H=0$, with $H$ as in Theorem \ref{thm:iversal}.
	\end{enumerate}
\end{cor}

We also obtain information about the Gr\"obner fan for the ideal $I_\A$ of relations among the cluster and frozen variables of $\A$ (see \S\ref{sec:clusterembedding}).
\begin{cor}[Corollary \ref{cor:grob}]
The Gr\"obner fan of $I_\A$ has a maximal dimensional cone $C$ corresponding to the Stanley-Reisner ideal $I_\K$. This cone may be described explicitly. If $\A$ has full rank (respectively full $\ZZ$-rank), the cone $C$ is simplicial (respectively smooth) modulo lineality space.
\end{cor}

\subsection{Organization}
We begin in \S\ref{sec:prelim} by recalling preliminaries and fixing notation. In particular, we introduce simplicial complexes, relevant notions from deformation theory, cluster algebras, and cluster complexes. None of this material is new.

In \S\ref{sec:cotan} we prove our first new results, showing that cluster complexes $\K$ for skew symmetric cluster algebras of finite cluster type are unobstructed, and giving a description of $T^1(S_\K)$. Both these results rely heavily on Altmann and Christophersen's work on cotangent cohomology for Stanley-Reisner schemes \cite{ac1,ac2}. The other essential ingredient is Proposition \ref{prop:Lb}, which describes the intersection behaviour of certain links of the cluster complex $\K$. 

In \S\ref{sec:coeffs} we recall the notions of principal and universal coefficients for cluster algebras, as well as Grabowski's work on gradings of cluster algebras \cite{Gra}. 
We relate the degrees of cluster variables to $\bg$-vectors (Corollary \ref{cor:cv_degrees}), characterize positively graded cluster algebras in terms of their $\bg$-fan (Remark \ref{rem:g-fan}), and show that one can always add frozen variables to any cluster algebra of finite cluster type to obtain a positive grading (Lemma \ref{lemm:add_frozens}). We also use the results of \cite{GHKK} to prove the folklore result that any cluster algebra of finite type (with non-invertible frozen variables) has a basis given by extended cluster monomials (Theorem \ref{thm:basis}).

In \S\ref{sec:hilb} we dive into the connection between $\A^\univ$ and the multigraded Hilbert scheme parametrizing multigraded ideals with the same multigraded Hilbert function as the ideal of relations among the cluster and frozen variables of $\A$. After fixing notation in \S\ref{sec:clusterembedding}, we state and prove the main result in this direction in \S\ref{sec:hilbthm}. The main ingredients are previous results on universal cluster algebras \cite[\S 12]{FZ_clustersIV}, the use of torus actions (or equivalently multigradings), and the basis of extended cluster monomials.
In \S\ref{sec:grob}, we apply these results to study the Gr\"obner theory of $I_\A$.

Finally, in \S\ref{sec:def} we study the relationship between $\A^\univ$ and the semiuniversal deformation of the Stanley-Reisner scheme $(\spec S_\K)\times \Aff^q$. For this, we first need to analyze the differential for the characteristic map of the family $\spec \A^\univ\to\Aff^p$, and understand the behaviour of derivations on the Stanley-Reisner ideal $I_\K$.

Throughout the paper, we follow a running example of a cluster algebra of type $A_2$, see Examples \ref{ex:r1}, \ref{ex:r2}, \ref{ex:r3}, \ref{ex:r4}, \ref{ex:r5}, \ref{ex:r6}, \ref{ex:r7}, \ref{ex:r8}, \ref{ex:r9},  and \ref{ex:r11}. We also consider an example of type $G_2$ immediately following in \S\ref{ex:g2}. \emph{Macaulay2} code for both examples is found in Appendix \ref{app:code}.
\subsection{An Example: $G_2$}\label{ex:g2}
To conclude the introduction, we consider the example of the cluster algebra $\A$ of type $G_2$ with trivial coefficients. There are exactly eight cluster variables $v_1,\ldots,v_8$, and they are related by the exchange relations
\begin{align*}
	v_{i-1}v_{i+1}=v_i+1 \qquad i\ \textrm{even};\\
	v_{i-1}v_{i+1}=v_i^3+1 \qquad i\ \textrm{odd}.
\end{align*}
Here, indices are taken modulo $8$. The cluster complex $\K$ is the boundary of an octagon. Its vertices (in cyclical order) are exactly the $8$ cluster variables $v_1,\ldots,v_8$.
Letting $z_i\in\KK[z_1,\ldots,z_k]$  be the free variable corresponding to the vertex $v_i$ of the cluster complex $\K$, the Stanley-Reisner ideal $I_\K$ is generated by
\begin{align*}
	z_iz_{i+2},\qquad z_iz_{i+3} \qquad i=1,\ldots,8\\
	z_jz_{j+4} \qquad j=1,\ldots,4.
\end{align*}

In this example, $\A$ does not have ``enough'' frozen variables (as it does not have any at all). We may nonetheless proceed as follows.
The space $T^1(S_\K)$ of first order deformations is infinite dimensional, but by Lemma \ref{lemma:injective}, the image of the characteristic map of $\spec \A^{\univ}\to\Aff^8=\spec\KK[t_1,\ldots,t_8]$ in $T^1(S_\K)$ consists of exactly the first order deformations given by the perturbation of $I_\K$ generated by
\begin{align*}
	z_{i-1}z_{i+1}-t_{i}z_i \qquad i\ \textrm{even};\\
	z_{i-1}z_{i+1}-t_{i}z_i^3 \qquad i\ \textrm{odd};\\
	z_iz_{i+3},\qquad z_jz_{j+4}
\end{align*}
modulo $\langle t_1,\ldots,t_8\rangle^2$.

One may compute a lifting of this family over $\Aff^8$ (using e.g.~the \emph{Macaulay2} package {\tt{VersalDeformations}} \cite{versal}) to obtain
\begin{align*}
	\begin{array}{l l l l l}
	z_{i-1}z_{i+1} &-t_{i}z_i &-t_{i+2}t_{i+3}t_{i+4}^2t_{i+5}t_{i+6}\qquad  &i\ \textrm{even};\\
	z_{i-1}z_{i+1} &-t_{i}z_i^3 &- t_{i+2}t_{i+3}^3t_{i+4}^2t_{i+5}^3t_{i+6}\qquad &i\ \textrm{odd};\\
	z_iz_{i+3} &-t_{i+1}t_{i+2}z_{i+2}^2&-t_{i+4}t_{i+5}^2t_{i+6}t_{i+7}z_{i+4}\qquad  &i\ \textrm{odd};\\
	z_iz_{i-3} &-t_{i-1}t_{i-2}z_{i-2}^2&-t_{i-4}t_{i-5}^2t_{i-6}t_{i-7}z_{i-4}\qquad  &i\ \textrm{odd};\\
	z_iz_{i+4} &-t_{i-1}t_{i-2}t_{i-3}z_{i-2}&-t_{i+1}t_{i+2}t_{i+3}z_{i+2}\qquad  &i\ \textrm{odd};\\
	z_iz_{i+4} &-t_{i-1}t_{i-2}^3t_{i-3}z_{i-2}&-t_{i+1}t_{i+2}^3t_{i+3}z_{i+2} -3t_{i-2}t_{i+2}(t_1\cdots t_8) \qquad  &
	i\ \textrm{even}.\\
\end{array}
\end{align*}
See \S\ref{code:g2} for \emph{Macaulay2} code computing this family.
In this instance, the hypotheses of Theorem \ref{thm:versal}\ref{part:one} are fulfilled. Furthermore, one may check that the above family is an \emph{exchange minimal} family (see Definition \ref{defn:exmin}).
By Theorem \ref{thm:versal}\ref{part:one} we may thus conclude that the above family is the universal cluster algebra $\A^\univ$ for $\A$. 

Specializing all deformation parameters to $1$, we obtain a presentation of the cluster algebra $\A$:
\begin{align*}
	\begin{array}{l l l l l}
	z_{i-1}z_{i+1} &-z_i &-1\qquad  &i\ \textrm{even};\\
	z_{i-1}z_{i+1} &-z_i^3 &- 1\qquad &i\ \textrm{odd};\\
	z_iz_{i+3} &-z_{i+2}^2&-z_{i+4}\qquad  &i\ \textrm{odd};\\
	z_iz_{i-3} &-z_{i-2}^2&-z_{i-4}\qquad  &i\ \textrm{odd};\\
	z_iz_{i+4} &-z_{i-2}&-z_{i+2}\qquad  &i\ \textrm{odd};\\
	z_iz_{i+4} &-z_{i-2}&-z_{i+2}-3
	\qquad & i\ \textrm{even}.\\
\end{array}
\end{align*}
We recover the exchange relations, but also additional relations between non-exchangeable cluster variables. In fact, the polynomials we have written down form a Gr\"obner basis (with initial ideal $I_\K)$.

\subsection*{Acknowledgments}
We would like to express our gratitude to Lara Bossinger and Jake Levinson for inspiring discussions related to the subject of this paper.
N.~Ilten is supported by NSERC.
 A. N\'ajera Ch\'avez is supported by the Secretar\'ia de Ciencias, Humanidades, Tecnolog\'ia e Inovaci\'on (Secihti) grant CF-2023-G-106. 
H.~Treffinger was partially funded by the Deutsche Forschungsgemeinschaft (DFG, German Research Foundation) under Germany's Excellence Strategy Programme -- EXC-2047/1 -- 390685813.
H.~Treffinger is also supported by the European Union’s Horizon 2020 research and innovation programme under the Marie Sklodowska-Curie grant agreement No 893654.
Part of the work by H.~Treffinger was carried out at the Isaac Newton Institute in Cambridge, in a visit supported by a grant from the Simons Foundation.
Portions of this paper were written on the unceded traditional territories of the Squamish, Tsleil-Waututh, Musqueam and Kwikwetlem Peoples. 

We thank the anonymous referee for helpful questions and comments.

\section{Preliminaries}\label{sec:prelim}
Throughout, $\KK$ is an algebraically closed field of characteristic zero.  Unless otherwise stated, all tensor products are over $\KK$.
\subsection{Simplicial Complexes and Stanley-Reisner Rings}\label{sec:SR}
A \emph{simplicial complex} is a collection $\K$ of subsets of some fixed finite set, such that if $f\in \K$, then every subset of $f$ is also in $\K$. Elements of $\K$ are called \emph{faces} of $\K$. Faces of $\K$ consisting of a single element are called \emph{vertices}; faces with two elements are \emph{edges}. We denote by $\V(\K)$ the union of all vertices of $\K$ (and will write $\V$ if $\K$ is clear from context).
We note that $\K$ is a subset of $\PS(\V(\K))$, the power set of $\V(\K)$.

To any simplicial complex $\K$, one may construct its \emph{Stanley-Reisner ideal}
\[
	I_\K=\left\langle \prod_{v\in f} z_v\ |\ f\in \PS(\V)\setminus \K\right\rangle \subseteq \KK[z_v\ |\ v\in \V]
\]
and its \emph{Stanley-Reisner ring}
\[
S_\K=\KK[z_v\ |\ v\in \V]/I_\K.
\]
Since $I_\K$ is a monomial ideal, the ring $S_\K$ has a natural grading by $\ZZ^{\V}$. Here, $\ZZ^{\V}$ is the group of tuples of elements of $\ZZ$ indexed by elements of $\V$.
See \cite{stanley} for more details on Stanley-Reisner rings.

A simplicial complex $\K$ is a \emph{flag complex} if all minimal non-faces have two elements, or equivalently, the ideal $I_\K$ is generated by quadrics.
We call the minimal non-faces of a flag complex \emph{non-edges}.

Given a face $f$ of $\K$, the \emph{link} of $f$ in $\K$ is the simplicial complex 
\[
	\link(f,\K)=\{g\in \K\ |\ f\cap g=\emptyset\ \textrm{and}\ f\cup g\in\K\}. 
\]
Given two simplicial complexes $\K,\K'$ with distinct vertex sets, their \emph{join} is
\[
	\K*\K'=\{f\cup f'\ |\ f\in \K,f'\in\K'\}.
\]
On the side of Stanley-Reisner rings, this corresponds to taking the tensor product:
\[
	S_{\K*\K'}\cong S_\K\otimes S_{\K'}.
\]

To any simplicial complex $\K$ we may associate a topological space known as its \emph{geometric realization}:
\[
|\K|=	\left\{\alpha:\V\to \RR\ |\ \supp(\alpha)\in \K\ \textrm{and}\ \sum_{v\in \V} \alpha(v)=1\right\}.
\]
Here $\supp(\alpha)=\{v\in \V\ |\ \alpha(v)\neq 0\}$.
We say that $\K$ is an \emph{$n$-sphere} if $|\K|$ is piecewise-linear homeomorphic to the boundary of $|\PS(\{0,\ldots,n+1\})|$.

\subsection{Cotangent Cohomology for Flag Complexes}\label{sec:cotangent}
For any commutative $\KK$-algebra $S$,
there is a complex $\CTC$ of free $S$ modules called the \emph{cotangent complex}, see e.g.~\cite[pp 34]{cotangent2}. The $i$th \emph{cotangent cohomology} of $S$ is 
\[
T^i(S)=H^i(\Hom_S(\CTC,S)).
\]
We will be particularly interested in $T^i$ for $i\leq 2$, which may be understood in more concrete terms. For example, the module $T^0(S)$ is just the module $\Der_\KK(S,S)$ of $\KK$-linear derivations of $S$ to itself.

The module $T^1(S)$ may also be understood quite concretely. Consider a presentation 
\[
0 \to I \to R \to S \to 0
\]
of $S$ with $R$ a polynomial ring (with possibly infinitely many variables). We may identify the module $T^1(S)$ with the quotient of $\Hom_R(I,S)$ by the image of $\Der_\KK(R,S)$ under the natural map \cite[\S 1.3]{deftheory}.
Elements of $T^1(S)$ are in bijection with isomorphism classes of \emph{first order deformations} of $S$: flat $\KK[\epsilon]/\epsilon^2$-algebras $\widetilde S$ whose reduction modulo $\epsilon$ is $S$.

When $S$ is graded by an abelian group $M$, the $S$-modules $T^i(S)$ inherit a natural $M$-grading. In particular,
taking a homogeneous presentation $R\to S$, we may represent a degree $u$ element of  $T^1(S)$ by an $R$-homomorphism $I\to S$ sending elements of degree $v$ to degree $v+u$.

The module $T^2(S)$ also has a concrete description, for which we refer the reader to \cite[\S 1.3]{deftheory}. Non-zero elements of $T^2(S)$ give \emph{obstructions} to lifting deformations of $S$ to higher order. In particular, if $T^2(S)=0$, any deformation of $S$ can be extended to arbitrary order.
See  \cite{deftheory} for a general introduction to deformation theory.

For $\K$ a simplicial complex, we will write $T^i(\K):=T^i(S_\K)$. 
\begin{defn}
	A simplicial complex is \emph{unobstructed} if $T^2(\K)=0$.
\end{defn}

The $S_\K$-modules $T^i(\K)$ inherit the natural $\ZZ^{\V}$-grading of $S_\K$.
Altmann and Christophersen \cite{ac1} give a combinatorial description for the graded pieces  $T^i(\K)_\bc$ for $\bc\in\ZZ^\V$. We will not give a full account of this description here, but instead provide only the details necessary for our purposes.

We will always write $\bc\in\ZZ^\V$ uniquely as $\bc=\ba-\bb$ with $\ba,\bb\in\ZZ_{\geq 0}^\V$ having disjoint supports $a,b\subseteq \V$. 
\begin{thm}[{\cite[Theorem 9, Proposition 11]{ac1}}]\label{thm:localization}
	For any simplicial complex $\K$ and $i=1,2$, $T^i(\K)_\bc=0$ unless $a\in \K$, $\bb\in\{0,1\}^\V$, $b\neq \emptyset$, and $b$ is contained in the set of vertices of $\link(a,\K)$. Furthermore, if $a\neq \emptyset$, there is a natural isomorphism
	\[
		T^i(\K)_\bc\cong T^i(\link(a,\K))_{-\bb}.
	\]
\end{thm}
In this paper, we will be particularly interested in cotangent cohomology for flag simplicial complexes that are spheres.
\begin{prop}\label{prop:t1}
	Let $\K$ be a flag complex that is a sphere. Assume that
	for any non-edge $\{u,v\}$, $\link(u,\K)\cap \link(v,\K)$ is either empty or contractible.
	Then for any $\bb\in\{0,1\}^\V$, 
	\[
		T^1(\K)_{-\bb}\cong\begin{cases}
			\KK & \textrm{if }\K\ \textrm{is a $0$-sphere and}\ b=\V\\
			0 & \textrm{otherwise}
		\end{cases}
	\]

\end{prop}
\begin{proof}
	We begin by noting that $T^1(\K)$ is a quotient of $\Hom(I_\K,S_\K)$. Since $\K$ is a flag complex, $I_\K$ is generated by quadrics.
	It follows that if $T^1(\K)_{-\bb}\neq 0$, the support $b$ satisfies $|b|\leq 2$ with $b\notin \K$ if $|b|=2$.
	
	Since $\K$ is a sphere, \cite[Lemma 4.3]{ac2} implies that $T^1(\K)_{-\bb}$ vanishes if $b$ is a vertex, so we may reduce to the case that 
	$b=\{u,v\}$ is a non-edge. Furthermore, by loc.~cit. $\dim T^1(\K)_{-\bb}\leq 1$, so it only remains to characterize when $T^1(\K)_{-\bb}$ vanishes.

	By \cite[Theorem 4.6]{ac2}, $T^1(\K)_{-\bb}\neq 0$ if and only if $\K$ is the join of $\partial b:=\PS(b)\setminus \{b\}$ with a sphere $L$. 
The complex $L$ is necessarily $\link(u,\K)\cap \link(v,\K)$. If this complex is contractible, then $L$ is not a sphere, and $T^1(\K)_{-\bb}\neq 0$. If 
$\link(u,\K)\cap \link(v,\K)$ is not contractible, then by assumption it is empty. Thus, $T^1(\K)_{-\bb}\neq 0$ if and only if $\K=\partial b$. The statement of the proposition follows.
\end{proof}

\begin{prop}\label{prop:t2}
Let $\K$ be a flag complex that is a sphere of dimension at least two. Then $\K$ is unobstructed if the following conditions are met:
\begin{enumerate}
	\item For every non-empty face $f$ of $\K$, $\link(f,\K)$ is unobstructed; and \label{item:links}
	\item For any non-edge $\{u,v\}$, $\link(u,\K)\cap \link(v,\K)$ is either empty or contractible. \label{item:Lb}
\end{enumerate}
\end{prop}
\begin{proof}
	This follows from \cite[Proposition 4.3]{srdegen}.
\end{proof}

\begin{rem}
\label{rem:1-spheres}
	Proposition \ref{prop:t2} is \emph{not} true in general if $\K$ is a $1$-sphere. In fact, if $\K$ is the boundary of an $m$-gon, then $\K$ is unobstructed if and only if $m\leq 5$ \cite[Example 17]{ac1}.
\end{rem}

We will make use of the following lemma:
\begin{lemma}[{\cite[Proposition 1.3]{srdegen}}]\label{lemma:tensor}
Let $S$ and $S'$ be commutative $\KK$-algebras. Then for all $i\geq 0$,
\[
T^i(S\otimes S')\cong T^i(S)\otimes S'\oplus T^i(S')\otimes S.
\]
\end{lemma}

\subsection{Semiuniversal Deformations}
We again refer the reader to \cite{deftheory} for details on deformation theory. For invariant and equivariant deformations, see \cite{rim}.

Suppose that $S$ is a quotient of a polynomial ring $\KK[z_1,\ldots,z_n]$ by an ideal $I$. 
Let $Y=\spec S\subseteq \Aff^n$.
Let $\TT$ be the maximal subtorus of $(\KK^*)^n$ which sends $Y$ to itself, and $H$ any subgroup of $\TT$.
Recall that an \emph{$H$-invariant deformation} of $Y$ over a base $Z$ is an $H$-invariant cartesian diagram
\[
\begin{tikzcd}
	Y \arrow[r] \arrow[d] & \cY \arrow[d,"\pi"] \\
	\spec \KK \arrow [r]& Z
\end{tikzcd}
\]
where $\pi$ is flat, and $H$ acts trivially on $Z$. 
Such a deformation is \emph{embedded} if $\cY\subseteq \Aff^n\times Z$, $\pi$ is just the projection, and the map $Y\to \cY\to \Aff^n\times Z$ is the natural inclusion induced by $\spec \KK\to Z$.

A \emph{morphism} of $H$-invariant deformations from $\pi':\cY'\to Z$ to $\pi:\cY\to Z'$ is a pair $(\phi,\widetilde \phi)$ of $H$-invariant morphisms making 
\[
	\begin{tikzcd}
		\cY' \arrow[ddd,"\pi'"] \arrow [rr,"\widetilde{\phi}"] && \cY \arrow[ddd,"\pi"] \\
		&Y\arrow[ur]\arrow[ul]\arrow[d]\\
	 &\spec \KK\arrow[dl]\arrow[dr]&\\
	 Z' \arrow [rr,"\phi"] && Z
\end{tikzcd}
\]
into a commutative diagram whose outside square is Cartesian. An isomorphism is a morphism such that $\phi,\widetilde{\phi}$ are both isomorphisms.
For embedded deformations $\pi',\pi$, a \emph{morphism of embedded deformations} is a morphism $(\phi,\widetilde{\phi})$ such that $\widetilde{\phi}$ is compatible with the inclusions $\cY'\subseteq \Aff^n\times Z'$, $\cY\subseteq \Aff^n\times Z$, and the map $\Aff^n\times Z'\to \Aff^n\times Z$ induced by $\phi$. The map $\widetilde{\phi}$ is uniquely determined by $\phi$, so we will often just refer to $\phi:Z'\to Z$ as a morphism of embedded deformations.

For any $H$-invariant deformation $\pi:\cY\to Z$ of $Y$ with $P\in Y$ the image of $\spec \KK$, there is a natural map of $\KK$-vector spaces
\[
	T_{P}Z\to T^1(S)^H
\]
sending any tangent vector $\eta\in T_PZ$ to the isomorphism class of the first order deformation obtained by pulling back $\pi$ along the morphism $\spec \KK[\epsilon]/\epsilon^2 \to Z$ corresponding to $\eta$. We call this map the \emph{characteristic map} for the deformation $\pi$.

\begin{defn}
An $H$-invariant deformation $\pi:\cY\to Z$  of $Y$ is \emph{semiuniversal} if 
	for any other $H$-invariant deformation $\pi':\cY'\to Z'$ of $Y$ with $Z'$ the spectrum of a local Artinian ring, there is a morphism $(\phi,\widetilde \phi)$ from $\pi'$ to $\pi$, and the differential of $\phi$ at the distinguished point is uniquely determined.\footnote{Technically, we have defined what it means to be \emph{formally} semiuniversal. However, since this is the only notion of semiuniversality we are considering, we omit the word ``formally''.}
\end{defn}

Under the condition that the $H$-invariant part $T^1(S)^H$ of $T^1(S)$ is finite dimensional, there is a semiuniversal $H$-invariant deformation $\pi:\cY\to Z$ of $Y=\spec S$. 
We will call $\cY$ and $Z$ respectively the \emph{semiuniversal family} and the \emph{semiuniversal base space}. There is a canonical identification of $T^1(S)^H$ with the tangent space of $Z$ at $P$.
By \cite{rim}, a semiuniversal deformation can be constructed so that it is also $\TT$-equivariant.
Furthermore, in this setting it can always be constructed as an embedded deformation.
Such an embedded $\TT$-equivariant deformation can be constructed algorithmically by lifting first-order perturbations of the elements of $I$ to higher and higher order \cite{stevens,versal} in a homogeneous fashion.

\begin{warning}
In general, there are two important warnings for the reader new to deformation theory:
\begin{enumerate}
	\item $\cY$ and $Z$ are in general formal schemes;
	\item Although any two semiuniversal deformations are isomorphic after passing to a formal completion, the isomorphism is not canonical.
		Likewise, the embedding of the semiuniversal family is not in general canonical.
\end{enumerate}
We emphasize these warnings now because one of the main points of our Theorem \ref{thm:versal} is that we will be able to ignore them in our special setting by making use of $\TT$-equivariance and our exchange-minimal deformations \ref{defn:exmin}. The semiuniversal deformation we will construct in Theorem \ref{thm:versal} comes with a canonical embedding, and the base and semiuniversal family are algebraic schemes.
\end{warning}

\begin{rem}[Standard Action]\label{rem:standardaction}
Assume that $T^2(S)^H=0$. Then a semiuniversal base space $Z$ for $Y$ is (the formal completion of) $\Aff^p$, where $p$ is the dimension of $T^1(S)^H$. Indeed, because of the vanishing of $T^2$, there are no obstructions to lifting deformations to higher and higher order. In general, this identification is not canonical. 
	
In the $\TT$-equivariant setting above, the $\TT$-action on $Y$ induces a $\TT$ action on this affine space $\Aff^p$. For a coordinate $t$ of $\Aff^p$ corresponding to an element of $T^1(S)^H$ of degree $\ba$, $\TT$ acts on $t$ with weight $-\ba$. In particular, the action of $\TT$  on $\Aff^p$ has a dense orbit if and only if $T^1(S)^H$ has a basis of semiinvariant functions whose degrees are linearly independent.
Indeed, the dimension of a general orbit of $\Aff^p$ under a torus action is equal to the dimension of the convex cone generated by the weights for this action; in this instance the generators are the opposites of the degrees of $T^1(S)^H$. Since $T^1(S)^H$ is $p$-dimensional, the claim follows.

In the case that $\Aff^p$ has a dense $\TT$-orbit, we will say that $\TT$ induces a \emph{standard action} on $Z$. 
\end{rem}

\subsection{Cluster Algebras}\label{sec:cluster}

A \emph{cluster algebra} is an integral domain equipped with a particular combinatorial structure. 
In the literature cluster algebras are usually defined as algebras over $\QQ$, however, in this work we consider cluster algebras defined over $\KK$. Since we are assuming that the characteristic of $\KK$ is zero, a cluster algebra over $\KK$ can be obtained from a cluster algebra over $\QQ$ by taking the tensor product (over $\QQ$) with $\KK$.
The reader is referred to \cite{FZ_clustersI,FZ_clustersII,FZ_clustersIV,FWZ_chapter1,FWZ_chapter4,FWZ_chapter6} for more details on the general theory. 

For positive integers $n$ and $m$ let $ \Mat_{m\times n} (\ZZ)$ denote the set of $m\times n$ matrices with integer entires.  
A square matrix $B=(b_{ij})\in \Mat_{n\times n} (\ZZ) $ is \emph{skew-symmetrizable} if there exists $d_1,\ldots,d_n\in \ZZ_{>0}$ such that $d_ib_{ij}=-d_jb_{ji}$ for all $i,j$. If we can take $d_1, \ldots , d_n$ to be $1$ we say that $B$ is \emph{skew-symmetric}.

From now on, we fix positive integers $n $ and $m$ such that $n\leq m $. We further fix a field $\F$ isomorphic to the field of rational functions over $ \KK$ in $m$ variables.

\begin{defn}
    A \emph{labeled seed of geometric type} in $\F$ is a pair $(\tx, \tB)$, where
\begin{itemize}
\item $\tx =(x_1,\dots, x_m)$ is an $m$-tuple of algebraically independent elements of $\F$ such that $\F = \KK(x_1, \dots , x_m)$;
\item $\widetilde{B} \in \Mat_{m\times n}(\ZZ)$ is such that its top $n\times n$ submatrix $B$ is skew-symmetrizable.
\end{itemize}
We adopt the following nomenclature: $n$ is the \emph{rank} of $(\tx,\tB)$\footnote{Observe that this notion of rank has nothing to do with the rank of $B$ or $\tB$.}, $\tB$ is an \emph{extended exchange matrix}, $B$ is an \emph{exchange matrix}, $\tx$ is an \emph{extended cluster}, the $n$-tuple ${\bf x}:=(x_1, \dots, x_n)$ is a \emph{cluster} and $x_{n+1}, \dots , x_m$ are the \emph{frozen variables}.
\end{defn}

In this paper we will only consider labeled seeds of geometric type. Hence, we will simply refer to these as seeds. 

\begin{rem}
To any exchange matrix $B$ we can associate a weighted directed graph $\Gamma(B)$ as follows.
We take the graph $\Gamma(B)$ to be the graph whose vertices are $ \{1, \dots , n\}$, with an edge from $i$ to $j$ if and only if ${b_{ij}}>0$. The \emph{weight} of such an edge is $c_{ij}=|b_{ij}|\cdot |b_{ji}|$.
In case $ \tB$ is skew-symmetric $\Gamma(\tB)$ can be thought of (and will be referred to) as a quiver (the weights correspond to the square of the number of arrows between the vertices).
\end{rem}

\begin{defn}
	Consider any exchange matrix $B$. The graph $\Gamma(B)$ decomposes into connected components $\Gamma_1,\ldots,\Gamma_k$. We let $B_i$ be the submatrix of $B$ whose row and column indices are among the vertices of $\Gamma_i$. We call $B_1,\ldots,B_k$ the \emph{connected components} of $B$ and say that $B$ is indecomposable if $\Gamma(B)$ is connected.
\end{defn}

For $k\in \{1, \dots , n \}$, the \emph{mutation in direction $k$} of a seed $(\tx,\tB)$ is the new seed $\mu_k(\tx,\tB)=(\tx',\tB')$, defined as follows: 
\begin{itemize}
	\item The new extended cluster is $\tx'=(x_1,\ldots,x_{k-1},x_k',x_{k+1},\ldots,x_m)$ where
    \begin{equation}
    \label{eq:exchange_rel}
    x_k x'_k= \prod_{i: b_{ik}>0}x^{b_{ik}}_i + \prod_{i:b_{ik}<0}x^{-b_{ik}}_i.
    \end{equation}
   An expression of the form \eqref{eq:exchange_rel} is called an \emph{exchange relation} and the monomial $x_k x_k'$ is the associated \emph{exchange monomial}.
    \item The new extended exchange matrix $\tB'=(b'_{ij})$ is defined by the following rule
    \begin{equation}
    \label{eq:matrix mut}
        b'_{ij}:= \left\{\begin{matrix} -b_{ij} &\text{ if } i=k\ \text{or}\ j=k,\\
        b_{ij} +\text{sgn}(b_{ik})\text{max}(b_{ik}b_{kj},0) &\text{ else;}\end{matrix}\right.
          \end{equation}
    where, for $x\in \RR$ we let
$
\text{sgn}(x):= \left\{\begin{matrix}-1 &\text{ if } x<0,\\
0 &\text{ if } x=0,\\
1 &\text{ if } x>0.\end{matrix}\right.
$
\end{itemize}
One can show that $(\tx',\tB')$ as constructed above is again a seed.

Two seeds are \emph{mutation equivalent} if they can be linked to each other by a finite number of mutations. This is an equivalence relation in the set of seeds in $ \F$ since $\mu_k$ is an involution, that is, we have
\[\mu_k(\mu_k(\tx,\tB))=(\tx,\tB).\] Moreover, if $ (\tx', \tB')$ is a seed mutation equivalent to $ (\tx,\tB)$ then  $\tx'$ is of the form $(x_1', \dots, x_n', x_{n+1}, \dots , x_m)$ for some $x_i'\in \F$.

In order to parametrize the cluster variables and the (extended) exchange matrices associated to $ \A(\tx,\tB)$ we label the seeds mutation equivalent to $(\tx,\tB)$ as follows. We consider the $n$-regular tree $\RTT^n$ whose edges are labeled by the mutable directions $1,\dots , n$, so that the $n$ edges incident to each vertex carry different labels.
Since $\mu $ is an involution we can assign (in different ways) a seed $(\tx_v,\tB_v)$ to every vertex $v$ of $\RTT^n$ in such a way that if $v$ and $v'$ are vertices of $ \RTT^n$ joined by an edge labeled with $k$ then $(\tx_{v'},\tB_{v'})=\mu_k(\tx_{v},\tB_{v})$.
Any such labeling is completely determined by choosing an initial vertex $v_0$ of $\RTT^n$ and setting $(\tx_{v_0},\tB_{v_0})=(\tx,\tB)$. We fix once and for all the choice of $v_0$. Moreover, we let
\[
\tx_v:=(x_{1;v}, \dots , x_{m;v}).
\]
Observe that $x_{i;v}=x_i$ for all $n< i \leq m$ and every vertex $v$ of $\RTT^n$.

\begin{defn}\label{def:cluster algebra}
Let $(\tx,\tB) $ be a seed in $\F$.  The \emph{cluster algebra} $\A(\tx,\tB)$ (of geometric type, over $\KK$) is the $\KK$-subalgebra of $\F$ generated by the elements of \emph{all} the extended clusters $ \tx_v$ in the seeds mutation equivalent to $(\tx,\tB)$. 
The elements $x_{i:v}$ for $i\leq n$ of the clusters constructed in this way are the \emph{cluster variables} of $\A(\tx,\tB)$. We call $(\tx,\tB)$ the \emph{initial seed} and the elements of the cluster ${\bf x}$ the \emph{initial cluster variables}. 
If $m>n$ we say that the cluster algebra $\A(\tx,\tB)$ has \emph{coefficients} and $x_{n+1},\ldots , x_{m} $ will be called the \emph{frozen variables}.
\end{defn}

\begin{ex}[Running Example]\label{ex:r1}
We consider the cluster algebra $\A$ associated to the seed 
	\begin{align*}	
		\tB=\left(\begin{array}{c c}
			0&1\\
			-1&0\\
			1&1\\
			-1&-1\\
			1&-1
		\end{array}\right)
\qquad
\tx=(x_{13},x_{14},s_1,s_2,s_3).
	\end{align*}
	We have called the elements of the initial cluster $x_{13},x_{14}$ instead of $x_1,x_2$ because of the connection to Pl\"ucker coordinates we will describe below.
	The frozen variables are $s_1,s_2,s_3$. The cluster variables of $\A$ are
\begin{align*}
	x_{13},\qquad x_{14}, \qquad
	x_{24}:=\frac{s_2 x_{14}+s_1s_3}{x_{13}}\\
	x_{35}:=\frac{s_1 x_{13}+s_2s_3}{x_{14}}\qquad
	x_{25}:=\frac{x_{13}s_1^2+x_{14}s_2^2+s_1s_2s_3}{x_{13}x_{14}}
\end{align*}
For example, $x_{24}$ is the cluster variable replacing $x_{13}$ in the initial seed after mutating at the position corresponding to $x_{13}$, that is, the first position.
We have again chosen names for the other cluster variables that are suggestive of Pl\"ucker coordinates. 

It is straightforward to check that the cluster variables satisfy exactly the relations
\begin{align*}
	x_{13}x_{24}&=s_2x_{14}+s_1s_3\\
	x_{13}x_{25}&=s_2^2+s_1x_{35}\\
	x_{14}x_{25}&=s_2x_{24}+s_1^2\\
	x_{14}x_{35}&=s_2s_3+s_1x_{13}\\
	x_{24}x_{35}&=x_{25}s_3+s_1s_2.
\end{align*}
These are exactly the Pl\"ucker relations for the Grassmannian $\Gr(2,5)$ after setting $x_{12}$ and $x_{45}$ to be $s_1$, $x_{23}$ and $x_{15}$ to be $s_2$, and $x_{34}=s_3$.

The weighted directed graph $\Gamma(B)$ associated to the exchange matrix $B$ is pictured in Figure \ref{fig:ex1}.
We will follow this example throughout the entire paper.
\end{ex}
\begin{figure}
\begin{tikzpicture}[
            > = stealth, 
            shorten > = 1pt, 
            auto,
            node distance = 3cm, 
            semithick 
        ]

        \tikzstyle{every state}=[
            draw = black,
            thick,
            fill = white,
            minimum size = 4mm
        ]

	\node[state] (x13) {$x_{13}$};
	\node[state] (x14) [right of=x13] {$x_{14}$};

        \path[->] (x13) edge node {1} (x14);

    \end{tikzpicture}

\caption{The weighted graph $\Gamma(B)$ for Example \ref{ex:r1}}\label{fig:ex1}
\end{figure}
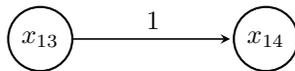

\begin{rem}
	In the literature, cluster algebras as introduced in Definition \ref{def:cluster algebra} are called skew-symmetrizable cluster algebras of geometric type over $\KK $. We call a cluster algebra \emph{skew-symmetric} if one (or equivalently all) of its seeds $(\tx,\tB)$ has skew-symmetric exchange matrix. 
\end{rem}
\begin{rem}\label{rem:invert}
	In the literature, cluster algebras with frozen variables $x_{n+1},\ldots,x_{m}$ are often defined as algebras over the ring
	$\KK[x_{n+1}^{\pm 1},\ldots,x_m^{\pm 1}]$, that is, the cluster algebra also contains the inverse of each of the frozen variables. This is the convention taken for instance in the seminal papers \cite{BFZ_clustersIII,FZ_clustersIV}. However, in more recent works such as \cite{ADS,GHKK} it was noticed that working with non-invertible frozen variables can lead to interesting constructions. This is the convention followed in \cite[Definition 3.1.6]{FWZ_chapter1} and is necessary when trying to endow homogeneous coordinate rings of projective varieties with the structure of a cluster algebra, as the only units of a positively graded $\KK$-algebra are necessarily constants.

	As stated in Definition \ref{def:cluster algebra}, we follow the convention of \cite{FWZ_chapter1} that a cluster algebra does not include the inverses of the frozen variables.
	This can make the situation more complicated than in \cite{BFZ_clustersIII}. For instance, in the convention of loc.~cit., by \cite[Corollary 1.2.1]{BFZ_clustersIII}, any cluster algebra of finite cluster type (see Definition \ref{defn:ft}) is a complete intersection ring. However, in the convention we adopt, this is not the case: Example \ref{ex:r1} is not a complete intersection. Indeed, this cluster algebra $\A$ is a $5$-dimensional quotient of the polynomial ring $\KK[x_{ij},s_k]$ in $8$ variables, but the kernel of the surjection $\KK[x_{ij},s_k]\to \A$ is minimally generated by $5$ elements and is thus not generated by a regular sequence.
\end{rem}

\begin{rem}\label{rem:ice_quiver}
Let $\tB$ be an extended exchange matrix (of size $m\times n$). It is straightforward to show that $\tB$ can be extended (non-uniquely) to an $m\times m$ skew-symmetrizable matrix $\tB'$. There is a canonical inclusion of $\A(\tB)$ into $\A(\tB')$. Thus, any cluster algebra with frozen variables may be thought of as a subalgebra of a cluster algebra with no frozen variables, with the subalgebra generated by a subset of the cluster variables.
\end{rem}

Up to a canonical isomorphism, $\A(\tx,\tB)$ is independent of $\tx$. Therefore, we will frequently write $\A(\tB)$ instead of $\A(\tx, \tB)$. The set of cluster variables associated to $\A(\tx,\tB)$ will be denoted by $ \V(\tx,\tB) $.
We will write $\V(\tB)$ or simply $\V$ to denote this set if the initial seed is clear from the context.

\begin{thm}[The strong Laurent phenomenon {\cite[Theorem 3.3.6]{FWZ_chapter1}}]
\label{thm:SLaurent}
Let $(\tx',\tB') $ be a seed mutation equivalent to $(\tx,\tB)$ and let $\bx'=(x_1', \dots , x_n' )$ be the corresponding cluster. Then the cluster algebra $\A(\tx, \tB)$ is a $\KK [x_{n+1},\dots, x_m]$-subalgebra of $\KK [x_{n+1},\dots, x_m][x_1'^{\pm 1},\dots x_{n}'^{\pm 1}]$.
\end{thm}

\begin{defn}\label{defn:ft}
A cluster algebra is of \emph{finite cluster type} if its set of cluster variables is a finite set. Otherwise it is of \emph{infinite cluster type}.
\end{defn}

\begin{rem}
\label{rem:coeff}
Observe that the strong Laurent phenomenon has the following geometric interpretation. If $\A(\tB)$ is of finite cluster type then we can consider $\spec (\A(\tB))$. By definition this an irreducible scheme over $\spec(\KK)$. 
By Theorem \ref{thm:SLaurent} we can consider $\spec (\A(\tB))$ as a scheme over $ \spec (\KK[{\bf t}])$, for any ${\bf t}\subseteq \{x_{n+1}, \dots x_{m}\}$.
\end{rem}

The \emph{Cartan counterpart} of an exchange matrix $B$ is the square matrix $A(B)=(a_{ij})$ of the same size  given by $a_{ij}=2$ if $i=j$ and $a_{ij}=-|b_{ij}|$ if $i\neq j $. 
A Cartan matrix $A(B)$ is  \emph{indecomposable} if $\Gamma(B)$ is connected; an indecomposable Cartan matrix is of \emph{finite type} if all the principal minors of $A(B) $ are positive. 
It is a classic fact in the theory of Lie algebras that indecomposable Cartan matrices of finite type are classified by Dynkin diagrams.

\begin{thm}
[Finite type classification \cite{FZ_clustersII}]
\label{thm:finite}
A cluster algebra $\A$ is of finite cluster type if and only if $\A$ has an extended exchange matrix $\tB$ such that for each connected component $B_i$ of $B$ the Cartan matrix $A(B_i)$ is  of finite type.
\end{thm}

\subsection{The Cluster Complex}\label{sec:complex}

To any cluster algebra $ \A=\A(\tB)$ we can associate a simplicial complex $\K(\A)=\K(\tB)$ called the \emph{cluster complex}.
\begin{itemize}
\item The set vertices of $\K(\tB)$ is the set of cluster variables $\V$ associated to $ \A(\tB)$;
\item A subset of cluster variables $\{u_1, \dots , u_i\}$ is a face of $\K(\tB)$ if and only if there exists a cluster ${\bf u}$ containing $u_1, \dots , u_i$ as entries.
\end{itemize}
Cluster variables belonging to a common face of $ \K(\tB)$ are called \emph{compatible}. 

It follows directly from the above definition that this complex is a flag complex.
Likewise, it is straightforward to verify that if $\tB$ and $\tB'$ are extended exchange matrices that can be related by iterated mutation, then $\K(\tB)$ and $\K(\tB')$ are canonically isomorphic.
A much deeper fact is that the cluster complex only depends on the mutable part of an extended exchange matrix. In other words, we have that $\K(\tB) $ is canonically isomorphic to $ \K(B)$.
This is an instance of the synchronicity phenomenon in cluster patterns studied in \cite{Nak} and follows directly from \cite[Proposition 6.1]{denom}.

\begin{thm}
\label{thm:bijection}
\cite{denom}
Let $\V(B)=\{ x_{i;v} \mid 1\leq i \leq n, v \in (\RTT^n)_0\}$ and $\V(\tB)=\{ \tilde{x}_{i;v} \mid 1\leq i \leq n, v \in (\RTT^n)_0\}$ be the sets of cluster variables of $\A(B)$ and $\A(\tB)$, respectively. Then the assignment $x_{i;v} \mapsto \tilde{x}_{i;v}$ establishes a bijection from $\V(B)$ to $\V(\tB)$. In particular, two cluster variables $x_{i;v}$ and $x_{j;v'}$ are compatible if and only if $\tilde{x}_{i;v}$ and $\tilde{x}_{j;v'}$ are compatible.
\end{thm}

In \cite{FZ_assoc} the authors study cluster complexes associated to cluster algebras of finite cluster type. In particular, they introduced a beautiful approach to understand the compatibility of cluster variables in terms of the associated root system. For cluster algebras of type $A$ the associated cluster complex is the simplicial complex dual to the boundary of the classical associahedron, also known as the Stasheff Polytope. 
More generally, we have the following:
\begin{thm}[{\cite[Corollary 1.11]{FZ_assoc}}]
Let $\A(\tB)$ be a cluster algebra of rank $n$ and finite cluster type. Then the cluster complex $\K(\tB)$ is an $(n-1)$-sphere.
\end{thm}
\begin{ex}[Continuation of Running Example \ref{ex:r1}]\label{ex:r2}
	We consider the cluster algebra $\A$ from Example \ref{ex:r1}. This is a cluster algebra of type $A_2$. The cluster complex $\K$ for $\A$ is the boundary of a pentagon, with vertices labeled as pictured in Figure \ref{fig:ex2}. 
\end{ex}
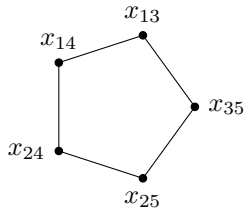
\begin{figure}
	\begin{tikzpicture}
   \newdimen\R
   \R=1cm
   \draw (0:\R) \foreach \x in {72,144,216,288,360 } {  -- (\x:\R) };
   \foreach \x/\l/\p in
   { 72/{$x_{13}$}/above,
   144/{$x_{14}$}/above,
   216/{$x_{24}$}/left,
   288/{$x_{25}$}/below,
   360/{$x_{35}$}/right
     }
     \node[inner sep=1pt,circle,draw,fill,label={\p:\l}] at (\x:\R) {};
\end{tikzpicture}
	\caption{The cluster complex $\K$ from Example \ref{ex:r1}}\label{fig:ex2}
\end{figure}

\section{Cotangent Cohomology for Cluster Complexes}\label{sec:cotan}
\subsection{Intersection of Links}
The following proposition will be  key for our results on the cotangent cohomology for cluster complexes:
\begin{prop}\label{prop:Lb}
Let $\A$ be a cluster algebra of finite cluster type such that some (or equivalently all) seeds $(\tx,\tB)$ satisfy that $\Gamma(B)$ is connected. For any two vertices $u,v$ of $\K(\A)$ that do not form an edge,
\[
\link(u,\K(\A))\cap \link(v,\K(\A))
\]
is either contractible or empty.
\end{prop}
In the $A_n$ case, this proposition is proven combinatorially in \cite[Lemma 5.2]{srdegen}. For the $B_n$, $C_n$, and $D_n$ cases, one may prove the proposition similarly using a combinatorial model for the type $B$, $C$, and $D$ associahedra \cite[Propositions 3.15 and 3.16]{FZ_assoc}. It remains to check the $E_6$, $E_7$, $E_8$, $F_4$ and $G_2$ cases. At the end of this subsection, we will instead provide a uniform proof of Proposition \ref{prop:Lb} using categorification. In particular, we will use standard concepts related to the categorification of cluster algebras such as 
\begin{itemize}
\item a \emph{cluster category}, see \cite[\S 1]{BMRRT};
\item a \emph{Frobenius 2-Calabi-Yau realization of a cluster algebra with coefficients} (also called a Frobenius categorification of a cluster algebra with coefficients), see \cite[Definition 5.1]{Fu_Kel};  
\item a \emph{2-Calabi-Yau action} on a Frobenius 2-Calabi-Yau realization of a cluster algebra, see \cite[Definition 2.46]{Demonet11}; 
\item a \emph{cluster-tilting object} and its associated quiver, see \cite[\S 2.4]{Fu_Kel} (the quiver we consider is the quiver of the additive subcategory of the cluster-tilting object); 
\item a \emph{cluster-tilting subcategory}, see \cite[\S II.1]{BIRS09};
\item a \emph{section} in a translation quiver, see \cite[Definition 2]{slices}.
\end{itemize}
The reader is referred to \cite{BMRRT,Kel_survay} for additional background and notation related to the categorification of cluster algebras. 
We denote by $[1]$ the suspension functor of a triangulated category; for every $a\in \ZZ$ we let $[a]:=[1]^a$.

We briefly recall the construction of the cluster category $\C_Q$ associated to a Dynkin quiver $Q$.
Let $\KK Q$ be the path algebra over $\KK$ associated to $Q$ and denote by $\mathcal{D}^b_Q$ the derived category associated to the category of finite-dimensional right $\KK Q$-modules. Since $\KK Q $ is hereditary, every indecomposable object of $\mathcal{D}^b_Q$ is isomorphic to a stalk complex whose only non zero entry is an indecomposble $\KK Q$-module.
By definition, $\C_Q$ is the orbit category associated to the autoequivalence $F:=[1]\circ \tau^{-1}$ of $\mathcal{D}^b_Q$, where $[1]$ and $\tau$ are respectively the suspension functor and Auslander-Reiten translation of $\mathcal{D}^b_Q$.
By definition $\C_Q$ and $\mathcal{D}^b_Q$ have the same objects and the suspension functor and Auslander-Reiten translation of $\C_Q$ are induced by those of $\mathcal{D}^b_Q$, see \cite{Kel_triang}.

We also briefly recall some properties of the Auslander-Reiten quiver of $\mathcal{D}^b_Q$, which we denote by $\Gamma (\mathcal{D}^b_Q) $.
Every vertex $[X]$ of $\Gamma (\mathcal{D}^b_Q) $ corresponds to the isomorphism class of an indecomposable object $X$ in $\mathcal{D}^b_Q$. The number of arrows from a vertex $[X]$ to a vertex $[Y]$ coincides with the dimension of the space of irreducible morphisms from $X$ to $Y$.
When $Q$ is an orientation of a simply laced Dynkin diagram then $\Gamma (\mathcal{D}^b_Q) $ has a simple combinatorial description. It is shown in \cite[Chapter I.5.6]{Hap} that $\Gamma (\mathcal{D}^b_Q) $ is isomorphic to the quiver $\ZZ Q $ defined as follows. The vertices of $\ZZ Q$ are pairs $(v,j)$ such that $v$ is a vertex of $Q$ and $j\in \ZZ$. For every arrow $\alpha$ in $Q$ from a vertex $v$ to vertex $v'$ and for every $j\in \ZZ$ there is an arrow from $(v,j)$ to $(v',j)$ and an arrow from $(v',j)$ to $(v,j+1)$, and all the arrows of $\ZZ Q$ arise in this way. 
From now on we systematically identify $ \Gamma (\mathcal{D}^b_Q)$ and $ \ZZ Q$.

Fix a section $\Sigma$ of $  \ZZ Q$, for instance any copy $Q\times {j}$ inside $\ZZ Q $. 
By definiton, we have that $\Sigma$ partitions the vertices of $\ZZ Q $ into three disjoint sets: the vertices that have a path of positive length towards $\Sigma$, which we call the predecessors of $\Sigma$, the vertices that lie in $\Sigma$, and the vertices that have a path of positive length from $\Sigma$, which we call the successors of $\Sigma$. 
Under the identification of $ \Gamma (\mathcal{D}^b_Q)$ with $ \ZZ Q$, the Auslander-Reiten translation of $\mathcal D^b_Q$ corresponds to the map $(v,j)\mapsto (v,j-1)$. In fact, if $[X]$ corresponds to a vertex of $\Sigma $ then $[\tau^{-1} X]$ corresponds to a successor of $\Sigma$ (\cite[Chapter I.5.6]{Hap}). Likewise, by the existence of an Auslander-Reiten triangle of the form $X\to Y\to \tau^{-1} X\to X[1]$, $[X[1]]$ also corresponds to a successor of $\Sigma$.

\begin{lemma}\label{lem:A1}
Let $Q$ be a bipartite orientation of a connected simply laced Dynkin diagram and let $\C_Q$ be the cluster category associated to $Q$. 
Let $W$ and $W'$ be (not-necessarily distinct) indecomposable summands of a basic cluster-tilting object $\underline{T}$ of $\C_Q$ whose quiver is isomorphic to $Q$ and such that both $W$ and $W'$ correspond either to a sink or to a source of the quiver of $\underline{T}$. 
If $W[2]$ is isomorphic to $W'$ then $Q$ is of type $A_1$. In particular, $\underline{T}$ is indecomposable and $\underline{T}=W=W'$.
\end{lemma}

\begin{proof}
	We only treat the case where both $W$ and $W'$ correspond to sinks of the quiver of $\underline{T}$; the source case is analogous.
	By the definition of $\C_Q$, there are indecomposable objects $\widetilde{W}$ and $\widetilde{W}'$ of $\mathcal{D}^b_Q$ such that $\widetilde{W}$ (resp. $\widetilde{W}'$) is isomorphic to $ W$ (resp. $W'$) in $ \C_Q$.
	After shifting by $\tau$, we can replace $\underline{T}$, $W$ and $\widetilde{W}$ by $\tau^k\underline{T}$, $\tau^k W$ and $\tau^k \widetilde{W}$, respectively for any $k\in\ZZ$. All bipartite cluster-tilting objects whose quiver is isomorphic to $Q$ are in the same $\tau$-orbit. This orbit includes the cluster-tilting object induced by $\KK Q$; its quiver is $Q$ and the vertices correspond exactly to the indecomposable projective $\KK Q$-modules.
	We can thus assume without loss of generality that $\widetilde{W}$ and $\widetilde{W}'$ are sinks of the bipartite section $Q\times 0\subset \ZZ Q$, and they are indecomposable projective $\KK Q$-modules.
 Since $W[2]\cong W'$, we must have that $\widetilde{W}[2]=F^a(\widetilde{W}')=([1]\circ \tau^{-1})^a(\widetilde{W}')$ for some power $a\in \ZZ$.

We claim that this is only possible if $a=1$. 
 From the identification of $\Gamma(\mathcal{D}^b_Q)$ with $\ZZ Q$ it is clear that $a>0$. Indeed, for $a< 0$ we have that $F^a(\widetilde{W}')$ is a predecessor of $\Sigma$ and $\widetilde{W}[2]$ is a successor of $\Sigma$, which contradicts the fact that $\widetilde{W}[2]=F^a(\widetilde{W}')$.
The case $a=0$ is discarded by a similar argument.
Next, if $a>1$ we would have that $\widetilde{W}= [a-2]\circ \tau^{-a} \widetilde{W'}$ is at the same time in $\Sigma$ and a successor of $\Sigma$, which is impossible.
Hence we must have that $a=1$. 

Since $\widetilde{W}\cong P[0]$ for an indecomposable projective $P$, then we have that  $\widetilde{W}'[-1]\cong I[-1]$ for some indecomposable injective representation $I$ of $Q$.
Therefore $\widetilde{W}' \cong P'[0]$, where $P'$ is an indecomposable projective-injective representation of $Q$. 
Now since $P'$ corresponds to a sink, then it is a simple representation of $Q$ and the only simple injective representations of $Q$ correspond to sources of $Q$. 
We can then conclude that $Q$ is a single vertex by the hypothesis that $Q$ is connected. The claims of the lemma follow.
\end{proof}

\begin{proof}[Proof of Proposition \ref{prop:Lb}]
    In view of \cite[Theorem 4.47]{Demonet11} and \cite[Theorem 2.1]{GLS11} there is a triple $(\C, G, \T)$ that categorifies a cluster algebra of geometric type $\A(\C, G, \T) $ whose cluster type is the same as $\A $ (see \cite[\S3.5 and \S4.3]{Demonet11}). 
    Here the category $\C$ is a Frobenius categorification of a skew-symmetric cluster algebra $\A(\C)$ of geometric and finite cluster type, $G$ is a finite group acting on $\C$, the action is $2$-Calabi-Yau and $\T$ is a $G$-stable cluster-tilting subcategory of $\C$. In the skew-symmetric case $G$ is the trivial group and $\A(\C)=\A(\C, G, \T)$. 
    Since the cluster complex is independent of the choice of coefficients it is enough to show the statement for $\A(\C, G, \T)$. 
    We now consider the equivariant category $\C G$ in the sense of \cite[\S2]{Demonet11}.
    The properties of $\C G$ that we will use are the following.
    \begin{itemize}
    \item The category $\C G$ is Hom-finite, Krull-Schmidt, Frobenius and stably 2-Calabi-Yau. In particular, $\Ext^1_{\C G}(X,Y)\cong  \Hom_\KK(\Ext^1_{\C G}(Y,X), \KK)$ for all objects $X$ and $Y$ of $\C G$;
    \item there is a bijection $v \mapsto T_v$ between the set of cluster and frozen variables associated to $\A(\C,G,\T)$ and the isomorphism classes of indecomposable objects of $\C G$. Under this bijection frozen variables correspond to the classes of the indecomposable projective-injective objects;
    \item two cluster variables $u$ and $v$ of $\A(\C,G,\T)$ are compatible if and only if $\Ext^1_{\C G}(T_u,T_v)= 0 $. Moreover, if $\Hom_{\uCG}(T_u,T_v)\neq 0$ then $u$ and $v $ are compatible, where $\uCG$ is the triangulated category obtained from $\C G$ by quotienting out the morphisms that factor through projective objects.
    \end{itemize}
The first item summarizes Lemma 2.8 (i), Corollary 2.44 and Proposition 2.47 (ii) of \cite{Demonet11}, while the second item follows from \cite[Theroem A]{Demonet11} and the fact that for any cluster category categorifying a skew-symmetric cluster algebra of finite cluster type every indecomposable object is rigid. The last item is well-known in the skew-symmetric case and can be extended to the skew-symmetrizable one using the fact that $ \Hom_{\uCG }(T_u,T_v)\subseteq \Hom_{\underline{\C}}(\mathfrak F(T_u),\mathfrak F (T_v))$ and $ \Ext^1_{\C G}(T_u,T_v)\subseteq \Ext^1_{\C}(\mathfrak F (T_u), \mathfrak F (T_v))$, where $\mathfrak F:\C G\to \C$ is the forgetful functor.

Consider non-compatible cluster variables $ u$ and $v$ of $\A(\C,G ,\T)$ and set $\K=\K(\A(\C,G,\T))$. Then 
\[
	\Ext_{\C G}^1(T_v,T_u)\neq0 \qquad\textrm{and}\qquad\Ext_{\C G}^1(T_u,T_v)\neq 0.
\]
Let $0\to T_v \to E \to T_u \to 0$ and $0\to T_u \to E' \to T_v \to 0$ be non-zero elements of  $\Ext_{\C G}^1(T_u,T_v)$ and  $\Ext_{\C G}^1(T_v,T_u)$, respectively. 

Assume first that both $ E$ and $E'$ are projective. Then $T_u\cong T_v[1]$ and $T_v\cong T_u[1]$ in $\uCG $. This implies that $T_u$ is isomorphic to $T_u[2]$.  We will use Lemma~\ref{lem:A1} to show that this can only happen if $ \Gamma(B)$ consists of a single vertex.
Indeed, we can identify $\uCG$ with a cluster category $\C Q$ associated to a connected Dynkin quiver $Q$.
Since $\C Q$ is independent of the mutation class of $Q$ we may assume further that $Q$ is bipartite.
Moreover, in this case we have that $\mathfrak F(T_u)$ is a sum of indecomposable objects which are summands of a bipartite cluster-tilting object $\underline{T}$ of $\uCG$ whose quiver is isomorphic to $Q$, and all the indecomposable summands of $\mathfrak F(T_u)$ are either sinks or sources of the quiver of $\underline{T}$. In particular, if $\mathfrak F(T_u)$ is invariant under $[2]$ then every indecomposable summand $W$ of $\mathfrak F(T_u)$ must satisfy that $W[2]$ is also an indecomposable summand, say $W'$, of $\mathfrak F (T_u)$.
By Lemma~\ref{lem:A1} we must have that $ \Gamma(B)$  consists of a single vertex and in this case it is easily verified that  $\link(u,\K(\A))\cap \link(v,\K(\A))=\emptyset$.

We may thus assume that at least one of $E$ and $E'$ has non-projective indecomposable summands. 
So without loss of generality let $E= P \oplus Z $, where $P$ is projective and $Z$ is non-zero and has no projective direct summands. 
Let $T_{z_1}, \dots , T_{z_s}$ be the indecomposable direct summands of $Z$ satisfying that $\Hom_{\uCG}(T_v,T_{z_i})$ and $\Hom_{\uCG}(T_{z_i},T_u)$ are non-zero. 
Observe that there is at least one such summand since $ T_v\to Z \to T_u \to T_v[1]$ is a non-zero element of $\Ext^1_{\uCG}(T_u,T_v)\cong \Ext^1_{\C G} (T_u,T_v)$ (recall that projectives vanish in $\uCG$).
Therefore,  $\Ext^1_{\C G}(T_v,T_{z_i})$ and $\Ext^1_{\C G}(T_{z_i},T_u)$ vanish. 
In particular, every $z_i$ is compatible with both $u$ and $v$, so $z_i \in \link(u,\K)\cap \link(v,\K)$. We now claim that
\[
	\link(u,\K)\cap \link(v,\K)= \PS(\{z_1,\dots , z_s\}) * \K',
\]
where $\K' $ is a sub-complex of $ \K$ and $\PS(\{z_1,\dots , z_s\})$ is the standard simplex with vertices $z_1,\dots , z_s$. 
For this, it is sufficient to show that every $z_i$ is compatible with every element $y\in \link(u,\K)\cap \link(v,\K)$. 
We apply the functor $\Hom_{\C G}(T_y,-)$ to  $0 \to T_v \to E \to T_u \to 0$ to obtain the exact sequence
\[
\Ext^1_{\C G}(T_y,T_v) \to
\Ext^1_{\C G}(T_y,E) \to
\Ext^1_{\C G}(T_y,T_u). 
\]
Since $y $ is compatible with both $v$ and $u$ we have that 
\[
\Ext^1_{\C G}(T_y,T_v) = 0 = \Ext^1_{\C G}(T_y,T_u).
\]
Therefore, $\Ext^1_{\C G}(T_y,E)=0$ which implies that $\Ext^1_{\C G}(T_y,T_{z_i})=0$ for all $i$. The result follows.

\end{proof}

\subsection{Unobstructedness}

We now come to our first main result:
\begin{thm}\label{thm:unobstructed}
Let $\K$ be the cluster complex of a skew-symmetric cluster algebra $\A$ of finite cluster type. Then $\K$ is unobstructed, that is, $T^2(\K)=0$.
\end{thm}
\noindent
We will complete the proof of this theorem at the end of this section after building up the necessary lemmas and propositions. The general strategy for the proof will be to induct on the rank of the cluster algebra $\A$. If an exchange matrix of $B$ of $\A$ is disconnected, we will see that it suffices to separately consider the connected components of $B$, allowing us to reduce the rank. For the connected case, we will instead be able to employ the criterion of Proposition \ref{prop:t2}; the key Proposition \ref{prop:Lb} will guarantee that the necessary hypotheses are met. 
\begin{rem}
\label{rem:obstructed}
	The skew-symmetric cluster algebras of finite cluster type are exactly those with a seed $(\tx,\tB)$ such that the Cartan counterpart of each connected component of $B$ corresponds to a simply laced Dynkin diagram, that is, of type $ADE$. However, for cluster algebras having a seed whose exchange matrix has a connected component whose Cartan counterpart corresponds to a non-simply laced Dynkin diagram, one can see that $T^2(\K)\neq 0$. 
	In the rank 2 case this follows from Remark \ref{rem:1-spheres} and the fact that a cluster complex of type $B_2$ or $C_2$ is a hexagon and of type $G_2$ is an octagon (see Example \ref{ex:g2}).
	Cluster complexes of rank 2 arise as links of certain faces of cluster complexes of higher rank (see Lemma \ref{lemma:clinks} below).
	Combining this with Theorem \ref{thm:localization} shows that skew-symmetrizable but non-skew-symmetric cluster complexes of arbitrary rank are obstructed.
\end{rem}

Although we are proving the theorem for skew-symmetric cluster algebras, we will later use several of the following lemmas without this restriction.

\begin{lemma}\label{lemma:joins}
Let $B_1,\ldots,B_k$ be the connected components of a skew-symmetrizable matrix $B$.  Then
\[
\K(B)=\K(B_1)*\cdots*\K(B_k).
\]
\end{lemma}
\begin{proof}
	This follows from the definition of the cluster complex. Any sequence of mutations at indices in $B_i$ does not change the $B_j$ for $j\neq i$. Likewise, none of the associated cluster variables are changed. Hence, any cluster variable created by iterated mutation in one component is compatible with any other created by iterated mutation in some other component.
\end{proof}

\begin{lemma}\label{lemma:clinks}
	Let $(\tx,\tB)$ be a seed and $f\in \K(B)$ a subset of the elements of $\bx$. Let $B'$ be the matrix obtained from $Q(B)$ by removing the rows and columns indexed by $\{i\ |\ x_i\in f\}$. Then 
	\[
\link(f,\K(B))\cong \K(B').
	\]
\end{lemma}
\begin{proof}
	It is clear from the definition of the cluster complex that $\K(B')$ can be identified with a subcomplex $\K'$ of $\link(f,\K(B))$.
By \cite[Theorem 6.2]{denom}, any two clusters containing a given fixed subset of cluster variables can be linked to each other via mutations that do not involve any exchange of the fixed cluster variables. It follows that $\K'$ is the entire link.
\end{proof}

\begin{proof}[Proof of Theorem \ref{thm:unobstructed}]
We will proceed by induction on the rank of the cluster algebra $\A$. In the rank one case, $\K$ is a $0$-sphere. One checks directly that $T^2(\K)=0$.
In the rank two case, $\K$ is the boundary of a square (type $A_1\times A_1$) or of a pentagon (type $A_2$). Here, we are using that $\A$ is skew-symmetric.
In both cases, one again checks directly that 
$T^2(\K)=0$.

For the induction step, suppose first that $\A$ has a seed $(\tx,\tB)$ with $\Gamma(B)$ disconnected. By the induction hypothesis along with Lemmas \ref{lemma:tensor} and \ref{lemma:joins}, we may conclude that $\K$ is unobstructed.

Suppose instead that every (or equivalently, some) seed $(\tx,\tB)$ has $\Gamma(B)$ connected. It follows by Theorem \ref{thm:finite} that $\Gamma(B)$ is a finite Dynkin diagram for some exchange matrix $B$.  In this case, we will apply Proposition \ref{prop:t2} to show the induction step. To check hypothesis (\ref{item:links}) of the proposition, we note that Lemma \ref{lemma:clinks}
shows that any link of $\K$ is a cluster complex for a cluster algebra $\A'$ of smaller rank. Hypothesis (\ref{item:links}) is then fulfilled by the induction hypothesis.
To check hypothesis (\ref{item:Lb}) of Proposition \ref{prop:t2}, we simply apply Proposition \ref{prop:Lb}.

The theorem follows by induction.
\end{proof}

\subsection{First Order Deformations}
Let $\K$ be the cluster complex of a cluster algebra $\A$ of finite cluster type.
We will now give a description of $T^1(\K)$ using the language of seeds. 
Consider any seed $(\tx,\tB)$ of $\A$. Recall that the elements of the cluster $\bx$ all appear as vertices of some face of the complex $\K$.

\begin{defn}
	Let $(\tx,\tB)$ be a seed of $\A$, and $\Omega$ a subset of the elements of $\bx=(x_1,\ldots,x_n)$. A cluster variable $x_k$ in the cluster $\bx$ is \emph{$\Omega$-isolated} if
	\begin{enumerate}
		\item $x_k$ is not in $\Omega$; and
		\item after removing the vertices of $\Gamma(B)$ corresponding to elements of $\Omega$, the vertex $k$ of $\Gamma(B)$ has no neighbors.
	\end{enumerate}
\end{defn}

\begin{thm}\label{thm:t1}
	Every graded piece of $T^1(\K)$ is either $0$- or $1$-dimensional. The degrees $\bc=\ba-\bb$ with non-zero $T^1$ may be constructed exactly as follows: given a seed $(\tx,\tB)$, a subset $\Omega$ of the elements of $\bx$, and an $\Omega$-isolated cluster variable $x_k$, let $\ba$ be any non-negative vector with support $\Omega$, and let $\bb\in\{0,1\}^\V$ have support $\{x_k,x_k'\}$. Here $x_k'$ is defined as in \eqref{eq:exchange_rel}.
\end{thm}
\begin{proof}
	By Theorem \ref{thm:localization}, $T^1(\K)_{\ba-\bb}\neq 0$ implies that $a\in \K$ and $b$ is contained in the set of vertices of $\link(a,\K)$. In such situations, 
	\[
		T^1(\K)_{\ba-\bb}=T^1(\link(a,\K))_{-\bb}.
	\]
	Since $a\in \K$, there is a seed $(\tx,\tB)$ with $\Omega:=\{a\}$ a subset of the elements of $\bx$.

By Lemma \ref{lemma:clinks}, we may identify
$\link(\Omega,\K)$ with $\K(B')$ where $B'$ is the matrix obtained from $B$ by removing the rows and columns corresponding to elements of $\Omega$. Decomposing $B'$ into connected components $B_1',\ldots,B_k'$, we further obtain
\[
\link(\Omega,\K)\cong \K(B_1')*\cdots*\K(B_k')
\]
by Lemma \ref{lemma:joins}.
By Lemma \ref{lemma:tensor}, we conclude that 
\[
	T^1(\K)_{\ba-\bb}\cong\bigoplus_i T^1(\K(B_i'))_{-\bb}.
\]
Note that at most one term $T^1(\K(B_i'))_{-\bb}$ is non-zero, as $b$ is contained in the vertices of at most one of the $\K(B_i')$.

Suppose that $b$ is contained in the vertices of $\K(B_i')$. Since $\Gamma(B_i')$ is connected and $\A(B_i')$ is of finite cluster type, we may apply Proposition \ref{prop:Lb}. Hence, for any non-edge $\{u,v\}$ of $\K'=\K(B_i')$, the intersections of $\link(u,\K')$ and $\link(v,\K')$ is either empty or contractible. Since $\K'$ is also a spherical flag complex, we may use Proposition \ref{prop:t1}
to determine 
$T^1(\K')_{-\bb}$.

To have $T^1(\K')_{-\bb}\cong \KK$, we must have that 
$\K'$ is a zero-sphere, and $b$ consists of the two vertices of this sphere.
The condition that $\K'$ is a zero-sphere is fulfilled exactly when $B_i'$ has rank one, that is consists of a single row and column $k$. This is the same thing as saying that the cluster variable $x_k$ is $\Omega$-isolated. The condition that $b$ consists of the two vertices of the sphere is the same as saying that $b=\{x_k,x_k'\}$.

In all other cases, 
$T^1(\K')_{-\bb}=0 $. The claims in the statement of the theorem follow.
\end{proof}

\begin{rem}\label{rem:t1el}
	Suppose that $T^1(\K)_{\ba-\bb}$ is non-zero. Then the generator for this graded piece is the image of the homomorphism $\psi\in\Hom(I_\K,S_\K)$ that sends $z^{\bb}$ to $z^{\ba}$, and every other minimal generator of $I_\K$ to zero. Indeed, the degree forces the condition that every other generator of $I_\K$ be sent to zero.
\end{rem}

\begin{ex}[Continuation of Running Example \ref{ex:r2}]\label{ex:r3}
	Let $\A$ be the cluster algebra from Example \ref{ex:r1}. Its cluster complex $\K$ is described in Example \ref{ex:r2}. By Theorem \ref{thm:unobstructed}, $\K$ is unobstructed. Furthermore, we may apply $T^1(\K)$ to compute the non-zero degrees of $\T^1(\K)$. Each cluster of $\A$ just consists of two variables, so for any cluster $\{x_1,x_2\}$, $x_1$ is $\Omega=\{x_2\}$-isolated (and vice versa). Making use of $\K$, we see that $T^1(\K)_{\ba-\bb}\neq 0$ if and only if $\ba,\bb$ have support as in Table \ref{table:ex3}, and $\bb\in\{0,1\}^\V$.
\end{ex}
\begin{table}
\begin{align*}
		\begin{array}{c c}
\supp(\ba) & \supp(\bb)\\
\hline
x_{14} & x_{13},x_{24}\\
x_{24} & x_{14},x_{25}\\
x_{25} & x_{24},x_{35}\\
x_{35} & x_{25},x_{13}\\
x_{13} & x_{35},x_{14}
	\end{array}
	\end{align*}
	\caption{Support of $T^1(\K)$ for Example \ref{ex:r3}}\label{table:ex3}
\end{table}

\section{Coefficients, Gradings, and Bases of Cluster Monomials}\label{sec:coeffs}
\subsection{Principal and Universal Coefficients}
Let $\tB\in \Mat_{m\times n}(\ZZ)$ be an extended exchange matrix. Its \emph{principal extension} is the extended exchange matrix $\tB^\prin$ obtained from $\tB$ by attaching a copy of the $n\times n$ identity matrix $I_n$ at its bottom. This can be represented as follows:
\[
\tB^\prin= \left(\begin{array}{c}
		\tB\\
		I_n
\end{array}\right)\in \Mat_{(m+n)\times n}(\ZZ).
\]
\begin{defn}\label{defn:prin}
The \emph{cluster algebra with principal coefficients} associated to $\tB$ is $\A(\tB^\prin)$. 
\end{defn}
In order to distinguish this case from others, we denote the new frozen variables of $ \A(\tB^{\prin})$ by $ y_i=x_{m+i}$ for $ 1\leq i \leq n$.
By the strong Laurent phenomenon \ref{thm:SLaurent} $\A(\tB^\prin)$ is a $\KK[y_1,\dots,y_n]$-subalgebra of $\KK[y_1,\dots,y_n][x_{n+1},\ldots,x_m][x_1^{\pm 1},\dots, x_n^{\pm 1}]$. 
Any cluster variable of $\A(\tB^\prin)$ can be written as $ x_{i;v}$ for some $1\leq i \leq n$ and a vertex $v$ of $\RTT^n$.

\begin{rem}\label{rem:tacit}
Cluster algebras with principal coefficients were initially defined \cite[Definition 3.1]{FZ_clustersIV} in the case of $m=n$, that is, no frozen variables. However, as mentioned in \ref{rem:ice_quiver}, we may view $\A(\tB)$ as a subalgebra of a rank $m$ cluster algebra $\A(\tB')$ with no frozen variables. Then
$\A(\tB^\prin)$ is a subalgebra of $\A((\tB')^\prin)$, and we may apply structural results on $\A((\tB')^\prin)$ (and the related concepts of $F$-polynomials and $\bg$-vectors) to deduce similar results for
$\A(\tB^\prin)$. We tacitly do this in the following.
\end{rem}

We now recall some of the fundamental results about cluster algebras with principal coefficients. It follows from \cite[Corollary 6.3, Proposition 5.1]{FZ_clustersIV} and \cite[Theorem 5.11]{GHKK} that a cluster variable (or a frozen variable) $x_{i;v}^\prin$ of $ \A(\tB^\prin)$ has the following shape
\begin{equation}
\label{eq:separation}
x_{i;v}^\prin= x_1^{g_1}\cdots x_m^{g_m} F_{i;v}(\hy_1,\dots , \hy_n),
\end{equation}
where,
\begin{itemize}
    \item $g_1,\dots , g_m\in \ZZ$ and $g_{j}\geq 0$ for all $n<j \leq m$;
    \item $F_{i;v}$ is a polynomial in $ \ZZ[y_1,\dots , y_n]$ with constant term equal to $1$;
    \item  for $1\leq j \leq n$ we take $\hy_j=y_j \prod_{i=1}^m x_i^{b_{ij}}$.
\end{itemize}
\begin{defn}
\label{def:g-vecs}
Let $x_{i;v}$ be a cluster or frozen variable of $\A(\tB)$. The $\bg$-\emph{vector} of $ x_{i;v}$ (with respect to the initial seed of $\A(\tB^\prin)$) is the vector 
\[
\bg (x_{i;v}):=(g_1,\dots, g_m)\in \ZZ^n\times \ZZ_{\geq 0}^{m-n},
\]
where $g_j$ is as in \eqref{eq:separation} for $x_{i;v}^\prin$.For each vertex $v$ of $\RTT^n$ the $\bg$-matrix $ G_v$ is the $m\times m$ square matrix whose $i$th column is $\bg(x_{i;v}) $. Polynomials $F_{i;v} $ as in \eqref{eq:separation} are called $F$-\emph{polynomials}.
\end{defn}
Since by Theorem \ref{thm:bijection} there is a canonical bijection $x_{i;v} \mapsto x^\prin_{i;v}$ between the sets of cluster variables of $\A(\tB)$ and $\A(\tB^\prin)$, $\bg (x_{i;v})$ only depends on the cluster or frozen variable $x_{i;v}$ and not directly on $i,v$.

\begin{rem}
	In the literature, $\bg(x_{i;v})$ as we have defined it is often called the (extended) $\bg$-\emph{vector} of $ x_{i;v}^\prin$ (as opposed to $x_{i;v}$). On the other hand, with the conventions adopted in Definition \ref{def:g-vecs}, the $\bg$-\emph{vector} of $ x_{i;v}^\prin$ is $\bg(x_{i;v}^\prin)$, which lives in $\ZZ^{m+n}$ (as opposed to $\ZZ^m$). However, notice that if $ x_{i;v}^\prin$ is mutable then the last $n$ entries of $\bg(x_{i;v}^\prin)$ are zero.
\end{rem}

\begin{rem}
We observe that the $F$-polynomials and the $\bg$-vector of a cluster variable depends on the choice of an initial seed and the vertex $ v_0$ of $\RTT^n$ labeling the initial seed of  $\A(\tB^\prin)$.
In this paper, whenever we talk about $F$-polynomials or {\bf g}-vectors we assume that the aforementioned choices have been taken once and for all. 
\end{rem}

\begin{rem}
\label{rem:gF}
	The fact that the last $m-n$ entries of a $\bg$-vector are non-negative follows from the sign-coherence of $\bg$-vectors proved in the skew-symmetric case in \cite{DWZ} and in full generality in \cite{GHKK}.  Indeed, any $\bg$-matrix of $\tB$ obtained by iterated mutation in the indices $1,\dots , n$ must contain the $\bg$-vectors of the frozen variables $x_{n+1}, \dots, x_{m}$. We have that $\bg(x_k)=(\delta_{k,j})_{j=1}^m$. In particular, given $n < j\leq m $ we have that the $j $th entry of  $  \bg(x_{j})$ is positive and by sign-coherence the $j$th entry of all the $\bg$-vectors of $\A(\tB^\prin)$ are non-negative.

	Corollary 6.3 of \cite{FZ_clustersIV} asserts that $F_{i;v}\in \ZZ[y_1, \dots , y_m]$ (working in $\A( (\tB')^\prin)$, see Remark \ref{rem:tacit}. However, using the mutation rule of $F$-polynomials appearing in \cite[Proposition 5.1]{FZ_clustersIV}  and the fact that to reach $v$ we never mutated in the directions $n+1,\dots ,m $, we can conclude that the variables $y_{n+1},\dots , y_m$ do not appear in $ F_{i;v}$.
\end{rem}

In many cases, the $F$-polynomial $F_{i;v}$ can also be used to compute $\bg(x_{i;v})$.

\begin{conj}
\cite[Conjecture 6.11]{FZ_clustersIV}
\label{conj:g_via_F}
Suppose that $ F_{i;v}$ is not identically equal to $1$.  Then the $\bg$-vector $\bg(x_{i;v})=(g_1,\dots , g_m)$ is given by
\[
    \prod_{i=1}^m u_i^{\tilde{g}_i}= \dfrac{F_{i;v}|_{\text{Trop}(u_1, \dots , u_m)}(u_1^{-1},\dots , u_n^{-1})}{F_{i;v}|_{\text{Trop}(u_1, \dots , u_m)}(\prod_{s=1}^m u_s^{b_{s1}},\dots ,\prod_{s=1}^m u_s^{b_{sn}})}.
\]
Here the specialization to $\text{Trop}(u_1, \dots , u_m)$ means we are working in the \emph{tropical semifield}, see e.g.~\cite[Definition 2.2]{FZ_clustersIV}.
\end{conj}
\noindent Conjecture \ref{conj:g_via_F} is true whenever $\A(\tB)$ is of finite cluster type \cite[Corollary 10.10]{FZ_clustersIV}. 

Although $\A(\tB^\prin)$ depends on our choice of seed, we may also add coefficients in a way that is \emph{independent} of the seed choice.
For any matrix $B$, denote its transpose by $B^\tr$.
\begin{defn}[{\cite[\S 12]{FZ_clustersIV}, \cite{reading}}]\label{def:univ frozen}
Let $\tB$ be an extended exchange matrix of finite cluster type. The \emph {cluster algebra with universal coefficients} $\A^\univ(\tB)$ is the cluster algebra associated to $\tB^\univ$, where
\begin{displaymath}
\tB^\univ=
\left(
\begin{array}{c}
     \tB \\ 
     U_{B}
\end{array}
\right)
\end{displaymath}
and $U_{B}$ is the matrix whose rows are the $\bg$-vectors of the cluster variables of $\A(B^{\tr})$.
\end{defn}
Up to canonical isomorphism, $\A^\univ(\tB)$ is independent of the choice of seed, and we will simply write $\A^\univ$. This definition agrees with the original definition of a universal cluster algebra \cite[\S 12]{FZ_clustersIV}, see \cite{reading}. We see from the above that $\A^\univ$ is obtained from $\A$ by adding in $p$ frozen variables $\bt=t_1,\ldots,t_p$, where $p$ is the number of cluster variables of $\A$. 

\begin{ex}[Continuation of Running Example \ref{ex:r3}]\label{ex:r4}
	Let $\A$ be the cluster algebra from Example \ref{ex:r1}.
One computes its $\bg$-vectors to be
\begin{align*}
	&\bg(x_{25})=(-1,0,0,2,0)\qquad && \bg(x_{13})=(1,0,0,0,0)&\\
	&\bg(x_{24})=(-1,1,0,1,0)\qquad&&\bg(x_{35})=(0,-1,0,1,1) &\\
	&\bg(x_{14})=(0,1,0,0,0) \qquad &&\bg(s_{1})=(0,0,1,0,0) &\\
	&\bg(s_{2})=(0,0,0,1,0)\qquad &&\bg(s_{3})=(0,0,0,0,1). &
\end{align*}
The projection of these $\bg$-vectors to the first two coordinates is pictured in Figure \ref{fig:ex4}.

One computes that, up to permutation of the rows, the matrix $U_B$ used in the extended exchange matrix for the cluster algebra with universal coefficients is
\[
\left(	\begin{array}{c c}
0&-1\\
0&1\\
-1&1\\
1&0\\
-1&0
	\end{array}\right)
\]
Taking frozen variables $t_1,\ldots,t_5$, we obtain the following exchange relations for $\A^\univ$:
\begin{align*}
	x_{24}x_{35}&=t_1s_3x_{25}+t_2t_5s_1s_2\\
	x_{14}x_{35}&=t_2s_1x_{13}+t_1t_3s_2s_3\\
	x_{14}x_{25}&=t_3s_2x_{24}+t_2t_4s_1^2\\
	x_{13}x_{25}&=t_4s_1x_{35}+t_3t_5s_2^2\\
	x_{13}x_{24}&=t_5s_2x_{14}+t_1t_4s_1s_3.
\end{align*}
\end{ex}
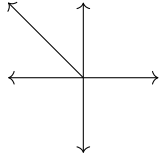
\begin{figure}
	\begin{tikzpicture}
		\draw[->] (0,0) -- (1,0);
		\draw[->] (0,0) -- (0,1);
		\draw[->] (0,0) -- (-1,0);
		\draw[->] (0,0) -- (0,-1);
		\draw[->] (0,0) -- (-1,1);
	\end{tikzpicture}

	\caption{Projection of $\bg$-vectors for Example \ref{ex:r4}}\label{fig:ex4}
\end{figure}

\subsection{Positively Graded Cluster Algebras}\label{sec:grading}

We will consider $\ZZ^d$-gradings on $\A(\tB)$ such that every cluster variable is homogeneous, where $d$ is some natural number. Observe that this is equivalent to the condition that all the exchange relations are homogeneous. 
A systematic study of these gradings was initiated in \cite{Gra} and we follow this approach. Although we allow frozen variables, our discussion essentially follows from that of \cite[\S 3]{Gra}.

\begin{defn}
A \emph{graded seed} is a triple $(\tx,\tB,D)$, where $(\tx,\tB)$ is a seed in $\F$ and $D\in \Mat_{m \times d}(\ZZ)$ such that $ \tB^\tr D=0$. 
\end{defn}

Assume that $(\tx,\tB, D)$ is a graded seed and let $D_i\in \ZZ^d$ be the $i$th row of $D$. For $1\leq i \leq m $ we can declare $\deg_D(x_i):=D_i$ to be the degree of the variable $x_i$ of $\tx$. 
Since $(\tx,\tB, D)$ is a graded seed the right hand side of the exchange relation \eqref{eq:exchange_rel} is homogeneous with respect to this grading. Therefore, we can use the exchange relation to determine the degree of the cluster variable $x'_k$ and obtain a graded seed $(\tx',\tB',D')$, where $(\tx',\tB')=\mu_k(\tx,\tB)$. We can iterate this process to obtain a well defined\footnote{It is a nontrivial fact that this is indeed well defined since we can reach the same cluster variable by different mutation sequences. The claim follows from \cite[Proposition 3.2]{Gra}.} degree $\deg_D(x)\in \ZZ^d $ for every cluster variable $x $ of $\A(\tx,\tB)$. The way the grading changes under mutation is encoded by the following class of matrices. 

\begin{defn}
\label{def:E_matrix}
Let $ \tB$ be an extended exchange matrix. Consider a sign $ \varepsilon \in \{+,-\}$ and $k \in \{ 1, \dots , n\}$. Let $ E_{k,\varepsilon}=(e_{ij})$ be the $m \times m$ matrix  given by 
\[
	e_{ij}:= \left\{\begin{matrix}\delta_{ij} &\text{ if } k\neq j\\
-1 &\text{ if } i=j=k,\\
\max(0,-\varepsilon b_{ik}) &\text{ if } i\neq j=k.\end{matrix}\right.
\]
\end{defn}
\noindent Observe that $e_{ij}$ depends on $\tB$ although we are not incorporating this dependence into the notation.

For a mutable index $k \in \{1, \dots, n \}$ the graded seed $\mu_k(\tx,\tB, D)$ is $(\tx',\tB',D')$, where $(\tx',\tB')=\mu_k(\tx,\tB)$ and 
\[
D'=(E_{k,+})^T D=(E_{k,-})^TD.
\]
Observe that it is indeed the case that $(E_{k,+})^T D=E_{k,-}D$ since the $k$th column of $\tB^T$ coincides with the $k$th column of $E_{k,+}-E_{k,-}$.

As before, given an initial assignment  $(\tx_{v_0},\tB_{v_0},D_{v_0})$ of a graded seed to a fixed vertex $v_0$ of $\RTT^n$, by applying iterated mutations we have a graded seed $(\tx_{v},\tB_{v},D_{v})$ for every vertex $v$ of $\RTT^n$. We assume that an initial assignment is fixed once and for all.
The following result follows from \cite[Theorem 5.11]{GHKK}, see also \cite{NZ} and \cite[\S 5.6]{Kel12}.

\begin{lemma}
\label{lem:G_and_E}
Suppose that $v$ is a vertex of $\RTT^n$ connected to $v_0$ by edges carrying the labels $k_0,\dots , k_{r}$, where the edge labeled by $k_i$ connects the vertices $v_i$ with $v_{i+1}$ and  $v_{r+1}=v$. 
Recall that for a vertex $v\in\RTT^n$, $G_v$ is the matrix whose columns are $\bg$-vectors for the extended cluster at $v$.
Then there exists a choice of signs $ \varepsilon_0,\dots ,\varepsilon_r $ such that
\[
G_{v} =E_{k_0,\varepsilon_0}\cdots E_{k_{r},\varepsilon_r}.
\]
\end{lemma}

\begin{cor}
\label{cor:cv_degrees}
Let $(\tx,\tB,D) $ be the initial graded seed associated to a vertex $v_0$ of $\RTT^n $ and $(\tx_v,\tB_v,D_v) $ the graded seed associated to a vertex $v$ of $\RTT^n$.
Then
\[
\deg(x_{i;v})=\bg_{i;v}\cdot D.
\]
\end{cor}
\begin{proof}
By definition, the vectors  $\deg(x_{1;v}),\dots,\deg(x_{m;v})\in \ZZ^d$ are the rows of $D_v$. 
By Lemma \ref{lem:G_and_E} we have that $G_v=E_{k_0,\varepsilon_0}\cdots E_{k_{r},\varepsilon_r}$ for a choice of signs $\varepsilon_0, \dots, \varepsilon_r$.
In particular, we have that 
\begin{align*}
D_v&= E^T_{k_{r},\varepsilon_r} \cdots E^T_{k_{0},\varepsilon_0}D \\
&=(E_{k_0,\varepsilon_0}\cdots E_{k_{r},\varepsilon_r})^T D \\
&= G_v^T D.   
\end{align*}
The result follows.
\end{proof}

\begin{defn}
A cluster algebra $\A(\tB)$ is \emph{positively graded} if it admits a $\ZZ$-grading such that every cluster variable is homogeneous of positive degree.
\end{defn}

\begin{rem}
\label{rem:g-fan}
By \cite[Theorem 2.13, Corollary 5.9]{GHKK} the cluster complex $ \K(\tB)$ can be realized as a simplicial fan $\Delta^+(\tB)$ in $\ZZ^m$ (every choice of initial seed of $\A=\A(\tB)$ gives rise to a realization).
We call $\Delta^+(\tB)$ the {\bf g}-fan since the maximal cones of $\Delta^+(\tB)$ are precisely the cones spanned by the columns of the $\bg$-matrices $G_v$, where $v$ ranges on all the vertices of $v\in \RTT^n$.
From this perspective, Corollary \ref{cor:cv_degrees} implies the following:
A cluster algebra $\A(\tB)$ is positively graded if and only if there exists a vector $ D\in \ZZ^m $ such that $\tB^TD=0$ and the {\bf g}-fan $\Delta^+(\tB)$ is contained in the open half space $ \{ v\in \RR^m \mid v\cdot D >0 \}$.
\end{rem}

Observe that if a cluster algebra of finite cluster type $\A(\tB)$ is positively graded, then $n<m$ \cite[\S4]{Gra}.

\begin{lemma}
    \label{lemm:extended_g-vecs}
    Let $w$ be a cluster variable of $ \A(\bx, B)$ and $\tilde{w}$ be the corresponding cluster variable of $\A (\tx,\tB)$ under the bijection of Theorem \ref{thm:bijection}. Denote by $\bg(w)=(g_1(w),\dots, g_n(w))$ the ${\bf g}$-vector of $w$ with respect to $(\bx, B)$ and by $\bg(\tilde{w})=(g_1(\tilde{w}),\dots, g_m(\tilde{w}))$ the ${\bf g}$-vector of $\tilde{w}$ with respect to $(\tx,\tB)$. 
    \begin{enumerate}
	    \item\label{item:i} Then $g_j(w)=g_j(\tilde{w})$ for all $1\leq j \leq n$.
	    \item\label{item:ii} Assume Conjecture \ref{conj:g_via_F} holds and $ \tilde{w}$ does not belong to the initial cluster. 
		    Then there exists a mutable index $k\leq n$ such that for any frozen index $n<j\leq m$, $-b_{j k}\leq g_{j}(\tilde{w})$.
    \end{enumerate} 
    \end{lemma}
    \begin{proof}
	    Item \ref{item:i} follows from the mutation formula for $\bg$-vectors contained in  \cite[Proposition 6.6]{FZ_clustersIV}.
    
    We proceed to show item \ref{item:ii}. Assume that $w=x_{i;v}$ and $\tilde{w}=\tilde{x}_{i;v}$ for some mutable index $i$ and some vertex $v$ of $\RTT^n$. Since we are assuming that $x_{i;v}$ is non-initial the $F$-polynomial $ F_{i;v}$ is not identically equal to $1$. This follows from the combination of \cite[Theorem 1.12, Theorem 1.3]{GHKK} and \cite[Theorem 6.16]{Nak_scat}. Since we are assuming that Conjecture \ref{conj:g_via_F} holds then we can compute $\bg(\tilde{x}_{i;v}) $ as the exponent of the monomial
    \begin{equation}
    \label{eq:g-vectors}
    \prod_{i=1}^m u_i^{g_i(\tilde{w})}= \dfrac{F_{i;v}|_{\text{Trop}(u_1, \dots , u_m)}(u_1^{-1},\dots , u_n^{-1})}{F_{i;v}|_{\text{Trop}(u_1, \dots , u_m)}(\prod_{s=1}^m u_s^{b_{s1}},\dots ,\prod_{s=1}^m u_s^{b_{sn}})}.
   \end{equation}
   
   Recall that $F_{i;v}\in \ZZ[y_1, \dots , y_n] $. Since $F_{i;v}$ is not identically equal to $1$, the variable $y_k$ appears in $F_{i;v} $ for some $1 \leq k \leq n$. Consider any $n<j\leq m$. If $b_{jk}\geq 0$, we are done, since we know that $g_{j}(\tilde{w})\geq 0$ for all $j>n$.

   Suppose instead that $b_{j k}<0$. By our assumption that $y_k$ appears in $F_{i;v}$, a power of the monomial $\prod_{s=1}^m u_s^{b_{sk}}= u_j^{b_{jk}} \prod_{s\neq j} u_s^{b_{sk}}$ appears in $F_{i;v}$. 
   Since $b_{jk}<0$, the exponent of $u_j$ in
   \[
{F_{i;v}|_{\text{Trop}(u_1, \dots , u_m)}(\prod_{s=1}^m u_s^{b_{s1}},\dots ,\prod_{s=1}^m u_s^{b_{sn}})}
\]
is less than or equal to $b_{jk}$.

   On the other hand, as noticed in the second paragraph of Remark~\ref{rem:gF}, $u_j$ does not appear in $F_{i;v} $. Hence $u_j$ does not divide $F_{i;v}|_{\text{Trop}(u_1, \dots , u_m)}(u_1^{-1},\dots , u_m^{-1})$. From these facts we conclude that the exponent of $u_j$ in  $\prod_{i=1}^m u_i^{g_i(\tilde{w})}$ is
   at least $-b_{jk}$, and thus $-b_{jk}\leq g_j(\tilde{w})$.
    \end{proof}
    
    \begin{rem}
	    Conjecture 6.11 of \cite{FZ_clustersIV} holds for any cluster algebra of finite cluster type, see \cite[Corollary 10.10]{FZ_clustersIV}. Hence we can apply item (\ref{item:ii}) of Lemma \ref{lemm:extended_g-vecs} to any cluster algebra of finite cluster type.
    \end{rem}
    
\begin{lemma}
    \label{lemm:add_frozens}
Let $\A(\tB)$ be a cluster algebra of finite cluster type. Then we can add rows to $\tB$ to obtain a larger exchange matrix $\tB^+$ such that $\A(\tB^+)$ admits a positive grading. Moreover, we can do this in a way such that all the frozen variables and the initial cluster variables have degree $1$.
\end{lemma}

\begin{proof}
Consider the $\bg$-vectors associated to $\A(\tB)$. All these are vectors in $\ZZ^m$. Let $ {\bg}(w)=(g_1(w),\dots , g_m(w))$ denote the $\bg $-vector of a cluster variable $w$ with respect to a fixed choice of initial seed $(\tx,\tB)$. Let $c=\min_{w\in {\V}}\{ \sum_{i=1}^m g_i(w) \}-1$. 
We proceed to construct $\tB'$ by adding $n$ rows (corresponding to new frozen indices) to $\tB$. For every mutable index $k\in \{ 1, \dots , n\}$ we add one new frozen index $k'=m+k$ and a new row $ (b_{k' 1},\dots, b_{k'n})$ given by 
\[
b_{k'j}=\left\{\begin{matrix}0 &\text{ if } j\neq k \\
-|c| &\text{ if } j=k.\end{matrix}\right.
\]
Observe that the triple $(\tx',\tB',(1,\dots ,1)) $ might not be a graded seed since the product  $(1,\dots , 1)\cdot\tB' $ might not be zero.
However, we can always consider a graded seed $(\tx^+,\tB^+,(1,\dots ,1)) $ obtained from $(\tx',\tB',(1,\dots ,1)) $ by adding yet another row to $\tB'$, corresponding to a new frozen index $m+n+1$.
Denote by ${\bf 1}:=(1,\dots ,1)\in \ZZ^{m+n+1}$ the grading vector of this graded seed.
By item (\ref{item:i}) of Lemma \ref{lemm:extended_g-vecs}, we know that for every cluster variable $w^+$ of $ \A(\tB^+)$, its $\bg$-vector (computed with respect to $(\tx^+,\tB^+)$) is of the form $(g_1(w^+),\dots, g_{m+n+1}(w^+))$, where the first $m$ entries coincide with the $\bg $-vector of a cluster variable of $\A(\tB)$, and the last entries are non-negative. Moreover, if $w^+$ is non-initial then by item (\ref{item:ii}) of Lemma \ref{lemm:extended_g-vecs} we know that  $g_{k'}(w^+)\geq -b_{k' k}= |c|$, where $k$ is a mutable index and $k'=m+k$ as above.
Now recall that $g_j(w^+)\geq 0$ for all frozen indices $j$. Putting these facts together we obtain the following inequalities for non-initial $w^+$:
\[ 
\bg(w^+)\cdot {\bf 1}\geq g_{k'}(w^+)+\sum_{i=1}^m g_i(w^+) \geq |c| + \sum_{i=1}^m g_i(w^+)\geq 1.
\]
The result now follows from Corollary \ref{cor:cv_degrees}.
\end{proof}

\begin{ex}[Continuation of Running Example \ref{ex:r4}]\label{ex:r5}
	We continue our discussion of the cluster algebra $\A$ from Example \ref{ex:r1}. Let $D$ be the matrix whose transpose is $(1,1,1,1,1)$. Then clearly $\tB^\tr\cdot D=0$, and we thus obtain a $\ZZ$-grading on $\A$. Multiplying $D$ by the $\bg$-vectors of the cluster variables of $\A$ (see Example \ref{ex:r4}), we obtain that all cluster variables have degree $1\in\ZZ$. In particular, the cluster algebra $\A$ has a positive grading.
\end{ex}

\begin{rem}\label{rem:torsion}
	The above discussion of gradings on cluster algebras via graded seeds can be extended to include gradings by any abelian group, not just $\ZZ^d$, see \cite[\S 2.1]{GP18}. When the grading group is finitely generated, this corresponds to considering actions on $\spec \A$ of general diagonalizable groups as opposed to just algebraic tori. For this, it is helpful to slightly rephrase things. We may interpret the grading matrices $D$ from above as encoding a homomorphism
\[D^\tr:\ZZ^m \to \ZZ^d.\]
More generally, we may consider homomorphisms
\[
\gamma:\ZZ^m \to M
\]
where $M$ is any abelian group. 
The condition that $\gamma$ defines a grading on $\A$ now becomes that $\gamma$ vanishes on the $\ZZ$-linear rowspan of $\tB^\tr$, or equivalently, that $\gamma \circ \tB(\ZZ^n)=0$.

The map $\gamma$ tells us the degrees of any cluster or frozen variable in the initial extended cluster $\tx$. After mutating at some mutable index $k\in \{1,\ldots,n\}$, the mutated version of $\gamma$ is given by
\[
	\gamma'=\gamma \circ E_{k,+}=\gamma \circ E_{k,-}.
\]
\end{rem}
\begin{ex}
Consider a cluster algebra $\A$ of type $B_2$, with exchange matrix 
\[
\tB=B=\left(\begin{array}{c c}
		0&1\\
		-2&0
	\end{array}\right).
\]
A possible grading for $\A$ is given by the homomorphism $\gamma:\ZZ^2\to \ZZ/2\ZZ$ with $\gamma(e_1)=\overline 1$ and $\gamma(e_2)=\overline{0}$.

The cluster algebra $\A$ has six cluster variables $x_1,\ldots,x_6$ which are related via the exchange relations
\begin{align*}
	x_{i-1}x_{i+1}=x_i^2+1\qquad i\ \textrm{even}\\
	x_{i-1}x_{i+1}=x_i+1\qquad i\ \textrm{odd}.
\end{align*}
Here indices are taken modulo $6$.
The variables $x_1,x_2$ are in the cluster with exchange matrix $\tB$. The grading constructed via $\gamma$ assigns degree $\overline 1\in \ZZ/2\ZZ$ to all $x_i$ with $i$ odd, and $0$ to all $x_i$ to $i$ even. It is clear that the exchange relations are homogeneous with respect to this grading.
\end{ex}

\subsection{Cluster Monomials}

Let $\A$ be a cluster algebra. A \emph{cluster monomial} is any monomial in the cluster variables of $\A$ such that all cluster variables appearing as a factor belong to a common cluster. An \emph{extended cluster monomial} is the product of a cluster monomial with any monomial in the frozen variables. A \emph{Laurent-extended cluster monomial} 
is the product of a cluster monomial with any Laurent monomial in the frozen variables.

We can extend the definition of the {\bf g}-vector of a cluster variable for  Laurent-extended cluster monomials as follows.
Let $\V$ and $\W$ be the cluster and frozen variables of $\A$. 
For a Laurent-extended cluster monomial $\prod_{x\in \V\cup\W} x^{\alpha_x}$ we set
\[
	\bg \left(\prod_{x\in \V\cup\W} x^{\alpha_x}\right):= \sum_x \alpha_x \bg(x) \in \ZZ^m.
\]

In this subsection, we prove the following folklore result:
\begin{thm}\label{thm:basis}
Let $\A$ be a cluster algebra of finite cluster type. 
The cluster monomials of $\A$ form a basis over the polynomial ring in the frozen variables. Equivalently, the extended cluster monomials form a basis over $\KK$. 
\end{thm}
It was stated without proof that for cluster algebras of types $A,B,C$ and $D$,  the Laurent-extended cluster monomials form a $\KK$-basis after tensoring with the ring of Laurent polynomials in the frozen variables \cite[Theorem 4.27]{CDM}. Theorem \ref{thm:basis} was proven for cluster algebras of finite cluster type with trivial coefficients in \cite[Corollary 3]{CK}.
Theorem \ref{thm:basis} implies that any cluster algebra of finite cluster type may be obtained from the universal one via base change, see \cite[Remark 12.11]{FZ_clustersIV}

The main ingredient is the following lemma, which follows from \cite{GHKK}:
\begin{lemma}\label{lemma:unibasis}
Let $\A(B)$ be a cluster algebra of finite cluster type with trivial coefficients. 
Let $\A'$ be the algebra obtained from $\A(B^\univ)$ by adding principal coefficients $s_1,\ldots,s_{n}$, and adjoining $s_1^{-1},\ldots,s_{n}^{-1}$.
Then the cluster monomials of $\A'$ form a  $\KK[\bt][s_1^{\pm 1},\ldots,s_{n}^{\pm 1}]$-basis of $ \A'$.
\end{lemma}
\begin{proof}
    To prove this statement we rely on the results of \cite[\S8 and \S9]{GHKK}.
    The reader is referred to \cite{GHK,GHKK} for a thorough treatment of the objects we consider in this proof.
    Throughout we use the following notation. If $\tB$ is an extended exchange matrix we let $\cA(\tB)$ be the cluster variety associated to (the fixed data corresponding to) $\tB$ in the sense of \cite[\S 2]{GHK}, and $\cA(\tB)_{\prin}$ be the corresponding cluster variety with principal coefficients. 
    We note that in loc.~cit. the authors add principal coefficients for not only cluster variables in an initial seed (as we do in Definition \ref{defn:prin}) but also the frozen variables. These principal coefficients corresponding to frozen variables are irrelevant for our discussion, as they do not appear in any exchange relations. We will thus stay consistent with Definition \ref{defn:prin} and only add principal coefficients for non-frozen variables.
       We also note that in \cite{GHKK}, the $\bg$-fan (as introduced in our Remark \ref{rem:g-fan}) also contains directions coming from the $\bg$-vectors of the multiplicative inverses of frozen variables. More precisely, each chamber of the $\bg$-fan is the cone generated by a maximal collection of compatible Laurent-extended cluster monomials.

    We let $ \overline{\cA}(\tB)$ be the partial compactification of $\overline{\cA}(\tB)$ given by setting the frozen variables to zero and denote by $\cA^{\vee}(\tB)$ the Langlands dual cluster variety.
    Using the above notation, $\A'$ is the \emph{ordinary} cluster algebra of $\overline{\cA}(B^\univ)_\prin$.
We will show using \cite[Corollary 9.18]{GHKK} that the \emph{upper} cluster algebra $\A''$ of $\overline{\cA}(B^\univ)_\prin$ has a $\KK[\bt][s_1^{\pm 1},\ldots,s_n^{\pm 1}]$-basis given by the cluster monomials. Since the upper cluster algebra contains $\A'$, and all cluster monomials are contained in $\A'$, the claim of the lemma will follow.

First we show that the two assumptions of the corollary are satisfied: the variety $\cA^{\vee}(B^\univ)_{\prin}$  has enough global monomials, and every frozen variable of $\A(B^\univ)$ has an optimized seed.
    The first condition follows at once from Propositions 8.16 and 8.25 of \cite{GHKK}. Indeed, since $\A(B)$ has finite cluster type, the {\bf g}-fan for $\cA(B^\univ)_\prin$  is the product of a complete fan (the {\bf g}-fan for $\A(B)$) with $\RR^{p+n}$, where $p$ is the number of cluster variables of $\A(B)$. In particular, the {\bf g}-fan for $\cA(B^\univ)_\prin$ covers the whole ambient space that contains it. In particular it is large (that is, it is not contained in a half space) and hence the hypothesis of Propositions 8.25 of \cite{GHKK} is fulfilled. 

    The property that every frozen variable of $\A(B^\univ)$ has an optimized seed is equivalent to the property that for every frozen index $j$ there is a matrix in the mutation class of $B^\univ$ such that its $j$th row has non-negative entries (see \cite[Definition 9.1]{GHKK}).
    This existence is clear since the $j$th row of the matrices in this mutation class can be identified with the set of the different {\bf g}-vectors of a fixed cluster variable of $\A(B^\tr)$ with respect to all the different (labeled) seeds in $\A(B^\tr)$. See Definition \ref{def:univ frozen}. 
    In particular, this set contains a canonical basis vector of $\ZZ^n$.
 
In order to apply \cite[Corollary 9.18]{GHKK}, we must also show that the cone $\Xi$ of \cite[Corollary 9.17]{GHKK} is contained in the convex hull of the set $\Theta$ of loc.~cit.  But the set $\Theta$ contains all integral points of the corresponding {\bf g}-fan, so as noted above, this will contain all of $\RR^{n+p}\times \RR^{n}$. In fact, since this {\bf g}-fan covers the entire ambient space, $\Theta$ consists of exactly the integral points of the {\bf g}-fan.
 
  We now apply \cite[Corollary 9.18]{GHKK} to $\overline{\cA}(B^\univ)_\prin$. Since $\Theta$ consists of the integral points of the {\bf g}-fan, we can conclude that $\A''$ has a $\KK$-basis given by the Laurent-extended cluster monomials whose {\bf g}-vectors lie in a rational polyhedral cone $\Xi \subset \RR^n\times \RR^p\times \RR^{n}$. Here,  the factor $\RR^p$ corresponds to the universal coefficients, and the latter $\RR^{n} $ comes from the principal coefficients.
    Moreover, from the description of $\Xi$ in \cite[Corollary 9.17]{GHKK} and our choice of partial compactification, it follows that $\Xi \subset \RR^{n}\times \RR_{\geq 0}^p\times \RR^{n}$, so the basis of $\A''$ only contains products of cluster monomials with monomials in $\KK[\bt][s_1^{\pm 1},\ldots,s_{n}^{\pm 1}]$. Since Laurent-extended cluster monomials are linearly independent (\cite[Theorem 6.8.10]{FWZ_chapter6} in the finite cluster type case), and $\A''$ contains every product of a cluster monomial with $\KK[\bt][s_1^{\pm 1},\ldots,s_{n}^{\pm 1}]$, we obtain that $\A''$ has the desired basis.
\end{proof}

We may now complete the proof of the theorem:

\begin{proof}[Proof of Theorem \ref{thm:basis}]
	The linear independence of the cluster monomials is \cite[Theorem 11.2]{FZ_clustersIV}. Thus, it remains to show that they span the cluster algebra. To that end, let $\A'$ be as in Lemma \ref{lemma:unibasis} for $\A(B)$, where $B$ is an exchange matrix for $\A$. By Lemma \ref{lemma:unibasis}, the cluster monomials of $\A'$ form a basis over $\KK[\bt][s_1^{\pm 1},\ldots,s_{n}^{\pm 1}]=\KK[\bt][\bs^{\pm 1}]$.

Let $\V$ be the set of cluster variables for $\A$ and  $\W$ the set of frozen variables.
Let $I_{\A}^\exr$ and $I_{\A'}^\exr$ denote the ideals of respectively $\KK[z_v\ |\ v\in \V\cup \W]$ and $\KK[z_v\ |\ v\in \V ][\bt][\bs^{\pm 1}]$ generated by the expressions obtained from the exchange relations \eqref{eq:exchange_rel} by substituting $z_v$ for the cluster or frozen variable $v$. Here we are using the bijection between cluster variables of $\A$ and $\A'$ of Theorem \ref{thm:bijection}.

By e.g.~\cite[Theorem 6.8.10]{FWZ_chapter6}, we may obtain $\A$ (or $\A'$) by taking the quotient of $\KK[z_v\ |\ v\in \V\cup \W]$ (or $\KK[z_v\ |\ v\in \V ][\bt][\bs^{\pm 1}]$) by the saturation of $I_{\A}^\exr$ (respectively $I_{\A'}^\exr$) with respect to the product $\prod_{v\in\V} z_v$. Furthermore, by the universal property of $\A(B^\univ)$ (cf.~\cite[Definition 12.3]{FZ_clustersIV}) there is a ring homomorphism $\KK[\bt]\to \KK[z_w\ |\ w\in\W]$ such that for the induced map
\[\rho:\KK[z_v\ |\ v\in \V][\bt][\bs^{\pm 1}]\to\KK[z_v\ |\ v\in \V\cup \W]\]
sending $s_i$ to $1$
we have
\[
I_{\A}^\exr=\rho(I_{\A'}^\exr).
\]

By an abuse of notation, we will call a monomial of $\KK[z_v\ |\ v\in \V]$ a cluster monomial if the cluster variables corresponding to its factors form a cluster monomial in $\A$ (or equivalently $\A'$).

To prove the claim, we must show that any monomial $\bz^\alpha\in \KK[z_v\ |\ v\in \V]$ can be written as a $\KK[z_w\ |\ w\in\W]$-linear combination of cluster monomials, modulo the saturation $I_\A$ of $I_{\A}^\exr$.
By Lemma \ref{lemma:unibasis}, we can write $\bz^\alpha$ as a $\KK[\bt][\bs^{\pm 1}]$-linear combination of cluster monomials, modulo the saturation $I_\A'$ of $I_{\A'}^\exr$. That is,
\[
	\left(\prod_{v\in\V} z_v^k\right)(\bz^\alpha-\sum_if_i\bz^{\beta_i})\in I_{\A'}^\exr
\]
for some $k\geq 0$, where $\bz^{\beta_i}$ are cluster monomials and $f_i\in\KK[\bt][\bs^{\pm 1}]$. Applying $\rho$, we obtain
\[
	\left(\prod_{v\in\V} z_v^k\right)(\bz^\alpha-\sum_i\rho(f_i)\bz^{\beta_i})\in I_{\A}^\exr
\]
and we see that $\bz^\alpha$ is indeed a $\KK[z_w\ |\ w\in\W]$-linear combination of cluster monomials, modulo $I_\A$.

\end{proof}

\section{Universal Coefficients and Hilbert Schemes}\label{sec:hilb}
\subsection{The Cluster Embedding}\label{sec:clusterembedding}
Let $\A$ be a cluster algebra of finite cluster type, with $\V$ the set of cluster variables and $\W$ the set of frozen variables. Since $\A$ is generated by the cluster and frozen variables, we obtain a presentation 
\[
	0\to I_\A\to \KK[z_v\ |\ v\in\V\cup\W]\to \A\to 0
\]
by sending the $z_v$ to the corresponding cluster and frozen variables.
The ideal $I_\A$ includes the exchange relations \eqref{eq:exchange_rel}, but in general has additional generators.

Geometrically, the above presentation gives rise to
the so-called \emph{cluster embedding}
\[
	\spec \A\subseteq \Aff^{p+q}
\]
where $p=\#\V$ and $q=m-n=\#\W$. Set $\TT=(\KK^*)^{p+q}$ and let $H\subseteq \TT$ be the maximal subgroup of $\TT$ whose elements send $\spec \A$ to itself in this embedding.
This inclusion of tori induces a surjection $\deg_H:\ZZ^{p+q}\to M$ to the character group $M$ of $H$. Since $H$ preserves $\spec \A$, the map $\deg_H$ gives an $M$-grading of the cluster algebra $\A$. In general, $H$ could be disconnected; equivalently, $M$ could contain torsion.

\begin{prop}\label{prop:M}
Let $(\tx,\tB)$ be a seed of $\A$. This induces an isomorphism of the cokernel of $\tB$ with $M$, the character group of $H$.
\end{prop}
\begin{proof}
Let $M'$ be the cokernel of $\tB$.
We thus obtain a map $\ZZ^m\to M'$.
This gives a grading of $\KK[\bz]$ as follows.
The degree of the variable $z_{x_i}$ corresponding to the initial cluster element $x_i$ is the image of the $i$th basis vector. Using the procedure of \S\ref{sec:grading} (in particular Remark \ref{rem:torsion}), this determines the degrees of all the variables $z_v$, and by construction, $\A$ is homogeneous with respect to this grading. Interpreting any other grading as coming from a map $\gamma$ as in Remark \ref{rem:torsion}, the universal property of the cokernel guarantees that any other grading must factor through the grading corresponding to the map $\ZZ^m\to M'$. By maximality of $H$ we conclude that $M'$ is indeed isomorphic to $M$, the character group of $H$.
\end{proof}
\begin{defn}[cf. {\cite[\S 4.3.7]{FWZ_chapter4}}]
	A cluster algebra $\A$ has \emph{full $\ZZ$-rank} if for some (or equivalently for every) seed $(\tx,\tB)$, the rows of $\tB$ span $\ZZ^m$.
	We say it has \emph{full rank} if for some (or equivalently for every) seed $(\tx,\tB)$, $\tB$ has full rank.
\end{defn}
\begin{cor}\label{cor:fullZ}
Suppose that $\A$ has full rank. Then $H$ is a torus if and only if $\A$ has full $\ZZ$-rank.
\end{cor}
\begin{proof}
	The property of $H$ being a torus is equivalent to $M$ being torsion-free. By Proposition \ref{prop:M}, this is equivalent to $\tB$ having torsion-free cokernel. Under the assumption that $\A$ has full rank, this is clearly equivalent to $\A$ having full $\ZZ$-rank.
\end{proof}

The presentation of $\A$ above lifts canonically to a presentation of $\A^\univ$:
\begin{equation}\label{eqn:pres}
		0\to  I_{\A^\univ}\to	\KK[z_v,\bt\ |\ v\in\V\cup\W]\to \A^\univ\to 0.
	\end{equation}
	It follows from the discussion of \S\ref{sec:grading} that the map $\deg_H:\ZZ^{p+q}\to M$ also makes $\A^\univ$ $M$-graded. In this grading, each coefficient $t_i$ has degree $0$, and the degree of a cluster or frozen variable of $\A^\univ$ is the same as the degree of the variable of $\A$ to which it specializes.

Geometrically, we have 
\[
\begin{tikzcd}
	\spec \A^\univ \arrow[r,hook] \arrow[rd]& \Aff^{p+q}\times\Aff^p \arrow[d]\\
& \Aff^p
\end{tikzcd}
\]
with the vertical arrow a projection, and $\spec \A^\univ$ a $H$-invariant subvariety of $\Aff^{p+q}\times \Aff^p$.

\begin{ex}[Continuation of Running Example \ref{ex:r5}]\label{ex:r6}
	Let $\A$ be the cluster algebra from Example \ref{ex:r1}. We have $p=5$ and $q=3$, and thus obtain an embedding of $\spec \A$ in $\Aff^8$. The exchange matrix $\tB$ has full $\ZZ$-rank. Identifying $M$ with the cokernel of $\tB$ and choosing as a basis the images of the third, fourth, and fifth standard basis vectors, we obtain an isomorphism of $M$ with $\ZZ^3$.
Under this identification, the grading matrix $D$ is given by
\[
	D^\tr=\left(\begin{array}{c c c c c}
-1&1&1&0&0\\
1&-1&0&1&0\\
1&1&0&0&1
	\end{array}\right)
\]
and the 
degrees of the variables $z_v$ are as follows:
\begin{align*}
	&\deg_H(x_{25})=(1,1,-1)\qquad &&	\deg_H(x_{13})=(-1,1,1)\\
	&\deg_H(x_{24})=(2,-1,0)\qquad &&	\deg_H(x_{35})=(-1,2,0)\\
        &\deg_H(x_{14})=(1,-1,1)\qquad &&	\deg_H(s_1)=(1,0,0)\\
	&\deg_H(s_2)=(0,1,0)\qquad && \deg_H(s_3)=(0,0,1).
\end{align*}
We observe that the map $\ZZ^3\to \ZZ$ obtained by summing the three coordinates recovers the positive $\ZZ$-grading of Example \ref{ex:r5}.
\end{ex}

\subsection{The Hilbert Scheme}\label{sec:hilbthm}
Assume that $\A$ is positively graded and consider the cluster embedding $\spec \A\subseteq \Aff^{p+q}$ introduced in \S\ref{sec:clusterembedding}, where $p$ (resp. $q$) is the number of cluster variables (resp. frozen variables) of $\A$. As before we let $\TT=(\KK^*)^{p+q}$ be the maximal torus of $ \Aff^{p+q}$, let $H\subseteq \TT$ be the maximal subgroup of $\TT$ whose elements send $\spec \A$ to itself in the cluster embedding and denote by $M$ the character group of $M$.
Let $\Hilb_{\A}^H$ denote the multigraded Hilbert scheme whose closed points parametrize $H$-invariant subschemes of $\Aff^{p+q}$ with the same $M$-graded Hilbert function as $\spec \A$. Equivalently, $\Hilb_{\A}^H$ parametrizes $M$-graded ideals $J$ of $\KK[\bz]$ such that $\KK[\bz]/J$ and $\KK[\bz]/I_\A$ have the same $M$-graded Hilbert function.
By \cite[Corollary 1.2]{multigraded} this is a projective scheme.
Clearly the cluster algebra $\A$ corresponds to a point of $\Hilb_{\A}^H$, which we denote by $[\A]$.

The multigraded Hilbert scheme comes equipped with a \emph{universal family} $\U\to\Hilb_{\A}^H$, where the fiber of $\U$ over a closed point $P\in \Hilb_{\A}^H$ is exactly the subscheme of $\Aff^{p+q}$ corresponding to $P$.
We note that the action of $\TT=(\KK^*)^{p+q}$ on $\Aff^{p+q}$ induces an action of $\TT$ on $\Hilb_{\A}^H$. 
We denote the $\TT$-orbit of $[\A]$ inside $\Hilb_\A^H$ by $\TT.[\A]$ and its closure by $\overline{\TT.[\A]}$.

We now relate $\A^\univ$ to the Hilbert scheme $\Hilb_{\A}^H$:
\begin{thm}\label{thm:hilb}
	Let $\A$ be a cluster algebra of finite cluster type, and let $\K$ be the cluster complex of $\A$. 
	\begin{enumerate}
		\item There is an action of $\TT$ on $\Aff^p=\spec \KK[\bt]$ so that $\spec \A^{\univ}$ is a $\TT$-invariant subvariety of $\Aff^{p+q}\times \Aff^p$.
		\item\label{part:h2}	The specialization of $\A^\univ$ at $\bt=0$ is the Stanley-Reisner ring $S_{\K*\PS(\W)}$.
		\item Assume that $\A$ is positively graded. Then there is a $\TT$-equivariant morphism \[\phi:\spec \KK[\bt] \to \Hilb_{\A}^H\] such that $\spec \A^{\univ}$ is the pullback of the universal family $\U$ along $\phi$. Moreover, $\phi$ is injective with injective differential.
		\item\label{part:h4} Suppose that $\A$ is positively graded and of full rank. Let $W$ be the unique torus invariant chart of $\overline{\TT.[\A]}\subseteq \Hilb_\A^H$ containing the torus fixed point corresponding to $I_{\K*\PS(\W)}$. 
			Then  $\phi$ is an immersion with image $W$. In particular, the family \[\spec \A^\univ \to \spec \KK[\bt]\] is canonically isomorphic to the restriction of 
	$\U\to\Hilb_{\A}^H$ to  $W$.
	\end{enumerate}
\end{thm}

Before starting the proof proper, we state and prove a number of lemmas. Under the assumption of a positive grading on $\A$ (to ensure the existence of $\Hilb_\A^H$), we show in Lemmas \ref{lemma:morphism}, \ref{lemma:phi_injective}, and \ref{lemma:differential} that there is a natural injective morphism $\phi:\spec \KK[\bt]\to \Hilb_\A^H$ with injective differential such that $\spec \A^\univ$ is the pullback for the universal family $\U$.
In Lemma \ref{lemma:equivariant} we show that if $\A$ has full $\ZZ$-rank, we may additionally assume that this morphism is torus equivariant. The actual proof of the theorem begins after Lemma \ref{lemma:equivariant}. We first consider the positively graded full $\ZZ$-rank case, and then use it to deduce the other claims.

\begin{lemma}\label{lemma:morphism}
	Assume that $\A$ is positively graded. There is a natural morphism $\phi:\spec \KK[\bt]\to \Hilb_\A^H$ such that $\spec \A^\univ$ is the pullback for the universal family $\U$.
\end{lemma}
\begin{proof}

	By Theorem \ref{thm:basis}, $\A^\univ$ is a free module over $\KK[\bt]$ with basis given by certain monomials in the cluster and frozen variables. It follows that for any $P\in\spec \KK[\bt]$, the fiber of $\spec \A^\univ$ at $P$ has the same $M$-graded Hilbert function as $M$. By the universal property of $\Hilb_\A^H$ (cf.~\cite[Theorem 1.1]{multigraded}), there is thus a unique morphism $\phi:\spec \KK[\bt]\to \Hilb_\A^H$ such that $\spec \A^\univ$ is the pullback for the universal family $\U$.
\end{proof}

\begin{lemma}\label{lemma:phi_injective}
Assume that $\A$ is positively graded. Then the morphism $\phi$ is injective.
\end{lemma}
\begin{proof}
Consider closed points $P\neq P'\in\Aff^p$, and suppose $\phi(P)=\phi(P')$. Let $1\leq i\ \leq p$ be such that $t_i(P)\neq t_i(P')$.
By \cite[Lemma 12.7 and Lemma 12.8]{FZ_clustersIV}, there is a pair of exchangeable cluster variables $x,x'$ such that in the universal cluster algebra,
	\begin{equation}\label{eqn:primitive}
x\cdot x'=t_i \cdot \alpha_1  + g\cdot \alpha_2
	\end{equation}
	where $g\in \langle t_j\ |\ j\neq i\rangle$, $\alpha_1$ is an extended cluster monomial, and $\alpha_2$ is a different monomial in the frozen variables. Indeed, the only way to have $\alpha_1=\alpha_2$ is if $\alpha_1=\alpha_2=1$. This is impossible since $\A$ is positively graded.

	Equation \eqref{eqn:primitive} translates into an element
	\begin{equation}\label{eqn:prim2}
	f_i:=z_x\cdot z_{x'}-t_i\cdot z_{\alpha_1}-g\cdot z_{\alpha_2}\in I_{\A^{\univ}}.
\end{equation}
Here $z_{\alpha_1}$ and $z_{\alpha_2}$ are the monomials in the $z_v$ mapping to $\alpha_1$ and $\alpha_2$.
The fact that $\phi(P)=\phi(P')$ means that $f_i(P')$ must be in the specialization $I(P)$ of $I_{\A^{\univ}}$ to $P$. Thus, we obtain that $f_i(P)-f_i(P')$ is in $I(P)$. This implies that in the specialization of $\A^{\univ}$ to $P$, a non-zero multiple of $\alpha_1$ can be expressed in terms of $\alpha_2$. This contradicts the fact that a basis of $\A^\univ$ is given by extended cluster monomials (Theorem \ref{thm:basis}).

We conclude that instead $\phi(P)\neq \phi(P')$. 
\end{proof}

\begin{lemma}\label{lemma:differential}
Assume that $\A$ is positively graded. Then the differential of $\phi$ is injective everywhere. 
\end{lemma}
\begin{proof}
	Consider a closed point $P\in\spec \KK[\bt]$ with coordinates $(s_1,\ldots,s_p)$. The  corresponding maximal ideal $\mfm\subset \KK[\bt]$ is $\mfm=\langle t_i-s_i\rangle$.
	The point $\phi(P)$ corresponds to an ideal $I(P)\subset \KK[\bz]$. This is obtained by specializing the ideal $I_{\A^\univ}$ to $P$.

	We will use the identification of $T_{\phi(P)} \Hilb_\A^H$ with $\Hom(I(P),\KK[\bz]/I(P))_0$, see e.g.~\cite[Proposition 1.6]{multigraded}. Here we are taking degree $0$ elements with respect to the $M$-grading.
	Consider the tangent vector $\eta_i$ at $P$ in direction $t_i$. To describe the element $\psi_i\in \Hom(I(P),\KK[\bz]/I(P))_0$ corresponding to 	$d\phi(\eta_i)$, consider any $f\in I(P)$. 
	Then there exists $\widetilde f\in \widetilde I_{\A^\univ}$ such that $\widetilde f$ reduces modulo $\mfm$ to $f$. We may write 
	\[
\widetilde f=f+(t_i-s_i)f'+g
	\]
	for some $f'\in \KK[\bz]$ and $g\in\langle (t_i-s_i)^2,t_j-s_j\ |\ j\neq i\rangle$. Then $\psi_i(f)$ is the image of $f'$ in $\KK[\bz]/I(P)$.

	The polynomial $f_i$ from \eqref{eqn:prim2} belongs to $I_{\A^\univ}$. Let $f_i(P)$ be its reduction modulo $\mfm$, viewed as an element of $\KK[\bz]$. We see
\[
	f_i=f_i(P) -(t_i-s_i)z_{\alpha_1} - g'\cdot z_{\alpha_2}
\]
for some $g'\in \langle t_j-s_j\ |\ j\neq i\rangle$.
With notation as in the preceding paragraph, we thus have $\psi_i(f_i)\equiv z_{\alpha_1}$. Furthermore, for $j\neq i$ we have that 
$\psi_j(f_i)$ is a multiple of the image of $z_{\alpha_2}$ modulo $I(P)$ (and possibly zero). Since the extended cluster monomials form a basis
(Theorem \ref{thm:basis}), there is no way to express $\psi_i$ as a linear combination of the $\{\psi_j\}_{j\neq i}$. The claim of the lemma follows. 
\end{proof}

\begin{lemma}\label{lemma:equivariant}
	Assume that $\A$ is positively graded and has full $\ZZ$-rank.
	There is a morphism of tori $\rho:\spec \KK[\bt^{\pm 1}]\to \TT$ making the morphism $\phi:\spec \KK[\bt]\to \Hilb_\A^H$ into a torus equivariant morphism.
\end{lemma}
\begin{proof}
Let $(\tx,\tB)$ be a seed of $\A$. There is a corresponding seed of $\A^\univ$ with extended exchange matrix $\tB^\univ$; the first $m$ rows of $\tB^\univ$ agree with those of $\tB$, but there are additional rows corresponding to the universal coefficients $\bt$. In fact, there are exactly $p$ additional rows, since the number of universal coefficients agrees with the number of almost positive roots of the corresponding root system \cite[Lemma 12.8]{FZ_clustersIV}, and this in turn agrees with the number of cluster variables \cite[Lemma 11.1]{FZ_clustersIV}.

We now consider any $D\in\Mat_{m\times p}$ such that
\[(\tB^\univ)^\tr \cdot \left(\begin{array}{c}
		D\\
		-I_p
	\end{array}\right)=0.
\]
Note that $D$ necessarily exists since we are assuming that the rows of $\tB$ generate $\ZZ^n$.

Following \S\ref{sec:grading}, the matrix
\[D'=\left(\begin{array}{c}
		D\\
		-I_p
\end{array}\right)\]
defines a grading on $\A^\univ$. In particular, every cluster and frozen variable has a degree in $\ZZ^p$. We now assign these same degrees to the variables $z_v$ and $z_w$, forgetting about the universal coefficients. Note that the ideal $I_\A$ is \emph{not} homogeneous with respect to this grading.

Let $\widetilde \TT$ be the torus $\spec \KK[\bt^{\pm 1}]$. Using the above grading, we obtain a map of tori  $\rho:\widetilde \TT \to \TT$. We now consider the action of $\widetilde\TT$ on 
$[\A]\in\Hilb_\A^H$ under $\rho$. 
By construction, the one-parameter subgroup $\lambda$ of $\widetilde\TT$ dual to the $i$th universal coefficient $t_i$ acts under $\rho$ in the exchange relations for $\A$ exactly as $t_i$ does. More specifically, if $s$ is the canonical coordinate for $\lambda$, after acting by $\lambda$ (via $\rho$) on any exchange relation of $\A$ and rescaling the exchange relation uniformly by a power of $s$, the power of $s$ in each monomial agrees with the power of $t_i$ in the corresponding monomial of the exchange relation of $\A^\univ$.

Any fiber of $\spec \A^\univ$ over $P\in \widetilde\TT\subset \spec \KK[\bt]$ is determined by the specialization of the exchange relations to the point $P$. It follows that the restriction of $\phi$ to $\widetilde \TT$ is torus equivariant. It is a straightforward exercise to show that any morphism of varieties with torus action $X\to Y$ that is equivariant on a dense open subset is equivariant on all of $X$; the claim follows.
\end{proof}

\begin{proof}[Proof of Theorem \ref{thm:hilb}]
Assume first that $\A$ is positively graded.	By Lemma \ref{lemma:morphism} we have seen that there is a natural morphism $\phi:\spec \KK[\bt]\to\Hilb_\A^H$.
	Moreover, it is injective with injective differential by Lemma \ref{lemma:phi_injective} and Lemma \ref{lemma:differential}.

	Before proving the claims of the theorem in general, we continue with the assumption that $\A$ is positively graded and additionally assume that is has full $\ZZ$-rank. Then $\phi$ is torus equivariant by Lemma \ref{lemma:equivariant}.
 Since $\phi$ is equivariant and the torus is acting transitively on $\Aff^p$, $\phi(\Aff^p)$ is contained in the closure of the torus orbit $\TT.[\A]$. 

 Since $\phi$ is injective, we obtain that $\dim \overline{\phi(\Aff^p)} = \dim \spec \KK[\bt]$. As noted in the proof of Lemma \ref{lemma:equivariant}, $\dim \spec \KK[\bt]=p$. On the other hand, the dimension of the torus orbit $\TT.[\A]$ is 
\[
		\dim \TT-\dim \TT_{[\A]}=
(p+q)-\dim H=(p+q)-q=p.
	\]
	Indeed, the dimension of $H$ is equal to the rank of the cokernel of $\tB$ for an extended seed $\tB$ of $\A$; since $\tB$ has full rank this is $m-n=q$. We conclude that $\overline{\phi(\Aff^p)}$ is equal to the closure of the torus orbit $\TT.[\A]$.
	Since $\A$ is positively graded, $\Hilb_{\A}^H$ is projective, so $\overline{\phi(\Aff^p)}$ is a projective (potentially non-normal) toric variety, and $\phi$ is a toric morphism.

	It follows that the image of $\phi$ is an affine $\TT$-invariant chart of the orbit closure of $[\A]$. Since $\phi$ has injective differential at its torus fixed point by Lemma \ref{lemma:differential}, it follows by Lemma \ref{lemma:iso} below that $\phi$ is an isomorphism onto its image.

We may thus identify $\spec \A^\univ\to \Aff^p$ with the restriction of $\U\to \Hilb_\A^H$ to $W$. This latter family is naturally $\TT$-equivariant. Since $W$ is a $\TT$-invariant subvariety of the Hilbert scheme, we obtain a $\TT$-equivariant structure on the family 
\[\spec \A^\univ\to \Aff^p\]
as well. 

Continuing in this special case of full $\ZZ$-rank, we now show that $P=\phi(0)$ is the point of $\Hilb_\A^H$ corresponding to $I_{\K*\PS(\W)}$. Let $I(P)$ be the ideal corresponding to $P$. We notice that $P$ is a torus fixed point of $\Hilb_\A^H$, so $I(P)$ must be a monomial ideal. Again by Theorem \ref{thm:basis}, the monomials in $z_v$ from compatible cluster and frozen variables form a $\KK$-basis for $\KK[\bz]/I(P)$. It follows that $I(P)=I_{\K*\PS(\W)}$ as desired. 
We have thus proven all claims of the theorem  when $\A$ is positively graded of full $\ZZ$-rank.

We now drop the assumption that $\A$ is positively graded of full $\ZZ$-rank. We can certainly add frozen variables to $\A$ to obtain a new cluster algebra $\A'$ of full $\ZZ$-rank. By Lemma \ref{lemm:add_frozens}, we can assume that we have done this in a way so that $\A'$ is also positively graded. We have shown that the statements of Theorem \ref{thm:hilb} thus apply to $\A'$. In particular, the family 
\[\spec (\A')^\univ\to \Aff^p\]
is $\TT'$-equivariant, where $\TT'$ is the diagonal torus on $\Aff^{p+q'}$. Here, $q'-q$ is the number of frozen variables we added. 
Furthermore, the fiber over $0$ in this family is $\spec S_{\K*\PS(\W')}$. Here $\W'$ is the set of frozen variables for $\A'$.

We may obtain $\spec \A^\univ$ by intersecting $\spec (\A')^\univ\subseteq \Aff^{p+q'}$ with the $q'-q$ hyperplanes $z_v=1$ for those $v\in \W'\setminus\W$, that is, setting the additional frozen variables to $1$. The torus $\TT$ embeds as a subtorus of $\TT'$, and since it acts trivially on $z_v$ for $v\in \W'\setminus\W$, we obtain a $\TT$-action on $\spec \A^\univ$. With regards to this action, the family
$\spec \A^\univ\to \Aff^p$ is $\TT$-equivariant by construction. This implies that in the positively graded case, the morphism $\phi$ is $\TT$-equivariant.

Finally, we note that the fiber over $0$ of this family is obtained from the fiber over $0$ of $\spec (\A')^\univ\to \Aff^p$ by setting $z_v=1$ for those
$v\in \W'-\W$. This is clearly $\spec S_{\K*\PS(\W)}$ as desired.

It remains to consider claim \ref{part:h4} in the cases where $\A$ has full rank, but not full $\ZZ$-rank. Similar to the case of full $\ZZ$-rank, we obtain that the dimension of $\TT.[\A]$ is $p$, so again $\overline{\phi(\Aff^p)}$ is the closure of $\TT.\A$. Since by the above $\phi$ is $\TT$-equivariant, we see that it acts on $\Aff^p$ with a dense orbit. We may again appeal to Lemma \ref{lemma:iso} to conclude that $\phi$ is an isomorphism. This shows claim \ref{part:h4}. 
\end{proof}

\begin{lemma}\label{lemma:iso}
Let $W$ be an affine variety. Suppose an algebraic torus $\widetilde \TT$ acts on $\Aff^p$ and $W$, both with dense orbits.
	Let $f:\Aff^p\to W$ be a dominant $\widetilde \TT$-equivariant morphism, with the differential $df$ injective at $0$. Then $f$ is an isomorphism.
\end{lemma}
\begin{proof}
	See \cite{CLS} for details on toric varieties.
	The variety $W$ may be constructed as $\spec \KK[S]$ for some affine semigroup $S$, and the morphism $f$ is induced by a map of semigroups $g:S\to \ZZ_{\geq 0}^p$. Since $f$ is dominant, $g$ must be injective, so we may identify $S$ with a submonoid of $\ZZ_{\geq 0}^p$.

	The differential being injective at $0$ is equivalent to the induced map of cotangent spaces $\mfm_{f(0)}/\mfm_{f(0)}^2\to \mfm_0/\mfm_0^2$ being surjective, where $\mfm_0$ is the maximal ideal of $0\in \Aff^p$, etc. Notice that $\mfm_0/\mfm_0^2$ has a basis given by the images of the characters corresponding to the standard basis of $\ZZ_{\geq 0}^p$. Since the map of cotangent spaces takes monomials to monomials, its surjectivity implies that $S$ must contain the standard basis of $\ZZ_{\geq 0}^p$. But then $S=\ZZ_{\geq 0}^p$ and the claim follows.
\end{proof}

\begin{rem}
	Instead of considering $\A^{\univ}$, we may consider the algebra
	\[\AOP:=\A^{\univ}\otimes_{\KK[\bt]} \KK[\bt^{\pm 1}]\]
	where we have inverted all the universal coefficients.
	Then $\spec \AOP$ fibers over $(\KK^*)^p$ instead of $\Aff^p$. When $\A$ is positively graded and has full rank,
	$\spec \AOP$ is the restriction of the universal family $\U$ to the torus orbit $\TT.[\A]$. In fact, one is able to prove this much simpler statement without making use of the machinery of \cite{GHKK} (in the form of Theorem \ref{thm:basis}). Indeed, using the full rank of $\A$ and an argument similar to that of Lemma \ref{lemma:equivariant}, the family $\spec \AOP\to (\KK^*)^p$ is equivariant with a transitive action on the base. As such, the map must be flat, and all fibers must be irreducible. It follows by dimension reasons that the fiber over $(1,\ldots,1)$ is $\spec \A$, and we thus obtain obtain a corresponding morphism $(\KK^*)^p\to \TT.[\A]$. The rest of the argument is similar to the proof of Theorem \ref{thm:basis}.

	In the proof of Theorem \ref{thm:hilb}, we used the existence of $\A^\univ$, along with some of its specific properties. However, we suspect that one should be able to show directly that the universal family over the orbit $\TT.[\A]$ satisfies the universal properties required of $\AOP$.
 Going a step further, we believe that one should be able to obtain a similar identification in the infinite cluster type case. Here, the torus acting would be infinite dimensional, since there are infinitely many cluster variables. This is a necessary feature of the infinite cluster type case, since we know from \cite{reading} that there are infinitely many universal coefficients in this case.

 In many infinite cluster type cases, $\A$ is nonetheless generated by a finite number of cluster variables. Instead of considering the cluster embedding, we could consider one coming from a finite number of cluster variables. Let $\TT'$ be the torus of the ambient space.	As noted above, we cannot recover $\AOP$ as the universal family over  $\TT'.[\A]$. It would be interesting to identify the subfamily of $\AOP$ obtained in this way.

 Extending from $\spec \AOP$  to $\spec {\A}^{\univ}$ presents significantly more difficulty in the infinite cluster type case, since Theorem \ref{thm:basis} will no longer be true in general. However, in good situations, $\A^{\univ}$ will still have a theta basis in the sense of \cite{GHKK}. We expect that the fiber over zero of the family $\spec \A^\univ \to \spec \KK[\bt]$ will no longer come directly from the cluster complex, but will instead be related to the scattering diagram for $\A$. 
\end{rem}

\begin{rem}\label{rem:canadd}
	Consider any cluster algebra $\A$ of finite cluster type. As noted in the proof of Theorem \ref{thm:hilb}, it is always possible to add additional frozen variables so that the new cluster algebra has full $\ZZ$-rank. Furthermore, by Lemma \ref{lemm:add_frozens}, it is always possible to do this in a way so that the new cluster algebra $\A'$ admits a positive grading.
\end{rem}

\begin{ex}\label{ex:hilb}
	It is straightforward to check that many classical examples of cluster algebras of finite cluster type are positively graded and of full $\ZZ$-rank:
	\begin{enumerate}
		\item The polynomial ring in $2(n+1)$ variables (type $A_{n}$) \cite[Example 6.1.2]{FWZ_chapter6};
		\item The homogeneous coordinate ring of the Grassmannian $\Gr(2,n+3)$ of $2$-planes in $\KK^{n+3}$ in its Pl\"ucker embedding (type $A_{n}$) \cite[Example 6.3.1]{FWZ_chapter6};
		\item The homogeneous coordinate rings of the Grassmannians $\Gr(3,6)$ (type $D_4$), $\Gr(3,7)$ (type $E_6$), and $\Gr(3,8)$ (type $E_8$) in their Pl\"ucker embeddings \cite{scott} \cite[Theorem 6.7.8]{FWZ_chapter6}.
	\end{enumerate}
A nice example of a cluster algebra of type $D_n$ is the homogeneous coordinate ring of the Schubert divisor of $\Gr(2,n+2)$ in its Pl\"ucker embedding (type $D_{n}$) \cite[Example 6.3.5]{FWZ_chapter6}. However, this example is not of full $\ZZ$-rank. Instead, we may consider
	\begin{enumerate}[resume]
		\item The cluster algebra of type $D_n$ described in \cite[Example 5.4.11]{FWZ_chapter4} with $n+2$ frozen variables.
	\end{enumerate}
	This cluster algebra does satisfy the hypotheses of Theorem \ref{thm:hilb}. The homogeneous coordinate ring of the Schubert divisor may be obtained by specializing two of the frozen variables to $1$.
	We may also consider non-simply laced examples:
	\begin{enumerate}[resume]
		\item The homogeneous coordinate ring of $\Gr(2,n+2)$ in its Pl\"ucker embedding with type $B_{n}$ cluster structure \cite[Proposition 12.11]{FZ_clustersII};
		\item The homogeneous coordinate ring of the Segre embedding $\PP^n\times\PP^n\subset \PP^{n^2+2n}$ with type $C_{n}$ cluster structure \cite[Example 6.3.4]{FWZ_chapter6}.
	\end{enumerate}
Both of these examples are also positively graded and of full $\ZZ$-rank.
\end{ex}

\begin{rem}
\label{rem:deficient}
	Assume that $\A$ is a positively graded cluster algebra of finite cluster type, but for some (or equivalently all) seeds, the extended exchange matrix $\tB$ is rank deficient. Then the conclusion of Theorem \ref{thm:hilb} that $\spec \A^\univ\to\spec \KK[\bt]$ may be identified with the universal family over a partial closure of $\TT.[\A]$ cannot be true. Indeed, since $\tB$ is rank deficient, the dimension of $H$ is strictly larger than $m-n=q$. We thus have 
	\[
		 \dim \TT.[\A]=p+q-\dim H<p.
	\]
	Since $\dim \spec \KK[\bt]=p$, the claim follows.
\end{rem}
We record the following algebraic consequence of Theorem \ref{thm:hilb}. 
\begin{cor}\label{cor:gorenstein}
Let $\A$ be a cluster algebra of finite cluster type. Then $\A$ is Gorenstein.
\end{cor}
\noindent 
Under the convention that both frozen variables and their inverses are adjoined (see Remark \ref{rem:invert}), \cite[Corollary 1.21]{BFZ_clustersIII} shows that any acyclic cluster algebra is a complete intersection, in particular Gorenstein. This includes all cluster algebras of finite cluster type.
Following our convention from Remark \ref{rem:invert} on \emph{not} adjoining inverses of frozen variables, \cite[Corollary 1.21]{BFZ_clustersIII} implies Corollary \ref{cor:gorenstein} only in the special case that $\A$ has no frozen variables.
See also \cite{CM} for a related result for lower bound algebras (the convention of loc.~cit. is again to adjoin inverses of frozen variables), and \cite{faber} for a classification of the possible singularities of $\A$ when $\A$ has no frozen variables.

\begin{proof}[Proof of Corollary \ref{cor:gorenstein}]
We may add frozen variables to obtain a positively graded cluster algebra $\A'$ of full $\ZZ$-rank. If we can show that $\A'$ is Gorenstein, then it follows that $\A$ is also Gorenstein, since $\A$ is a complete intersection in $\A'$ (obtained by setting additional frozen variables equal to $1$). Hence, in the following we may assume that $\A$ is positively graded of full $\ZZ$-rank. 

The ring $S_{\K*\PS(\W)}$ is Gorenstein since $\K$ is a sphere, see \cite[\S II.5]{stanley}. By Theorem \ref{thm:hilb}\ref{part:h2} $\spec S_{\K*\PS(\W)}$ is a complete intersection in $\spec \A^\univ$, so it follows from the definition of Gorenstein that $\A^{\univ}$ is also Gorenstein along $\bt=0$. By Theorem \ref{thm:hilb}\ref{part:h4} every $\TT$-orbit closure of $\spec \A^\univ$ intersects the special fiber over $\bt=0$. Since the non-Gorenstein locus of $\spec \A^\univ$ is closed, $\TT$-invariant, and avoids the fiber over $\bt=0$, it follows that the non-Gorenstein locus of $\spec \A^\univ$ is empty. Thus, $\A^{\univ}$ is Gorenstein.

To conclude, we observe that since $\A$ is a complete intersection in $\A^\univ$ (obtained by setting $t_1=\ldots=t_p=1$), $\A$ inherits the Gorenstein property from $\A^\univ$.
\end{proof}

\begin{ex}[Continuation of Running Example \ref{ex:r6}]\label{ex:r7}
	Let $\A$ be the cluster algebra from Example \ref{ex:r1}. Since $\A$ has full $\ZZ$-rank, we obtain that $\A^\univ$ (cf. Example \ref{ex:r4}) is the coordinate ring of the universal family over an open subset of $\overline{\TT.[\A]}$. 
	Specializing to $\bt=0$, we obtain the Stanley-Reisner ring $S_\K$ for the cluster complex $\K$ of Example \ref{ex:r2}. The torus $\TT$ is acting on the parameters $t_i$ with the following weights:
	\begin{align*}
\deg_\TT(t_1)=(-1,0,1,1,0,0,0,-1)\\
\deg_\TT(t_2)=(0,-1,0,1,1,-1,0,0)\\
\deg_\TT(t_3)=(1,0,-1,0,1,0,-1,0)\\
\deg_\TT(t_4)=(1,1,0,-1,0,-1,0,0)\\
\deg_\TT(t_5)=(0,1,1,0,-1,0,-1,0).
	\end{align*}
	Here, we have ordered the cluster and frozen variables as $x_{25}$ ,$x_{13}$, $x_{24}$, $x_{35}$, $x_{14}$, $s_1$, $s_2$, $s_3$.
\end{ex}

\subsection{Connection to Gr\"obner Theory}\label{sec:grob}
Our Theorem \ref{thm:hilb} is connected to Gr\"obner theory for cluster algebras as we now explain. See \cite{gfan,grob} for an introduction to Gr\"obner theory.
Let $J\subset \KK[z_1,\ldots,z_s]$ be an ideal. The ideal $J$ induces a natural polyhedral decomposition of $\RR^s$ into a finite collection of cones: two elements $u,w\in \RR^s$ lie in the relative interior of the same cone if and only if the initial ideal of $J$ with respect to $u$ equals the initial ideal of $J$ of with respect to $w$. The \emph{Gr\"obner fan} of $J$ consists of the subset of these cones whose relative interiors intersect the positive orthant. When $J$ is homogeneous with respect to a positive grading, then the Gr\"obner fan covers all of $\RR^s$.

Let $J\subset \KK[z_1,\ldots,z_s]$ be an $M$-graded ideal with each graded piece a finite dimensional $\KK$-vector space. We may thus consider the Hilbert scheme $\Hilb_J$ parametrizing $M$-graded ideals with the same multigraded Hilbert function as $J$. In particular, $J$ gives a point $[J]$ of this Hilbert scheme. Let $\TT$ be the diagonal torus acting on $\spec \KK[\bz]= \Aff^s$ and let $N_\TT= \ZZ^s$ be its cocharacter lattice. The action of $\TT$ on $\spec \KK[\bz]$  induces an action on $\Hilb_J$. The setting of Theorem \ref{thm:hilb} is exactly the special case where $J=I_\A$ and $M$ is the grading induced by the stabilizer of $\TT$ on $\spec \A$.

Returning to the  general setting, consider the orbit closure $Y=\overline{\TT.[J]}\subset \Hilb_J$. This is a projective, potentially non-normal toric variety.
Let $\widetilde \TT$ be the image of $\TT$ in $\Aut(Y)$; we denote the character lattice of this torus by $N_{\widetilde \TT}$. The normalization of $Y$ comes with a natural $\widetilde \TT$-action as well and hence corresponds to a fan $\widetilde\Sigma$ in the $\RR$-vector space associated to $N_{\widetilde \TT}$ \cite[Corollary 3.1.8]{CLS}.

The map $\TT\to \widetilde{\TT}$ induces a map
\[\eta: N_{\TT} \to N_{\widetilde \TT}.\]
Let $\Sigma$ be the preimage in $\RR^s$ of $\widetilde{\Sigma}$ under the map of real vector spaces associated to $\eta$.

The following is well-known to experts:

\begin{prop}\label{prop:grob}
	The fan $\Sigma$ is the Gr\"obner fan of $J$.
\end{prop}
\begin{proof}
	Since each graded piece of $J$ is finite dimensional, $J$ is homogeneous with respect to a positive grading and the Gr\"obner fan covers all of $\RR^s$.
	By definition two weight vectors in $\ZZ^s$ lie in the relative interior of the same cone of the Gr\"obner fan of $J$ if and only if the initial ideals of $J$ with respect to these weight vectors is the same ideal. This is equivalent to requiring that the limits of $[J]$ in $\Hilb_J$ of the images under $\eta$ of the corresponding one-parameter subgroups are equal. 

	To conclude the proof, we claim that two one-parameter subgroups of $\widetilde\TT$ have the same limits in the normalization of $Y$ if and only if the  limits in $Y$ coincide. The proposition will then follow since two vectors in $N_{\widetilde \TT}$ lie in the relative interior of the same  cone of $\widetilde\Sigma$ if and only if the one-parameter subgroups in the normalization of $Y$ have the same limit. To show the claim, note that since $\Hilb_J$ has an affine $\TT$-invariant open cover, so does $Y$; this cover of $Y$ is also $\widetilde\TT$-invariant. We may thus reduce to the affine case. But by e.g.~\cite[Theorem 3.A.3(c)]{CLS} the normalization map induces a bijection between the limits of one-parameter subgroups.
\end{proof}

We now return to the special situation of Theorem \ref{thm:hilb}. Let $\deg_\TT$ be the map assigning to any cluster or frozen variable $v\in \V\cup \W$ the degree in $\ZZ^{p+q}$ of $z_v$. We extend this in the obvious way to a map on the extended cluster monomials of $\A$.
\begin{cor}\label{cor:grob}
	Let $\A$ be a cluster algebra of  finite cluster type, with ideal $I_\A$ as in \S \ref{sec:clusterembedding}. 
	\begin{enumerate}
		\item 	The Gr\"obner fan of $I_\A$ contains a maximal cone $C$ corresponding to the monomial ideal $I_{\K*\PS(\W)}$. 
\item The cone $C$ is dual to the cone of $\ZZ^{p+q}$ generated by all
	\[
		\deg_\TT(x\cdot x')-\deg_\TT\alpha_1\in \ZZ^{p+q}
	\]
where $x,x',\alpha_1$ appear in an exchange relation
\[
x\cdot x'=\alpha_1+\alpha_2
\]
such that $\alpha_1,\alpha_2$ are extended cluster monomials, and $\alpha_2$ only depends on the frozen variables.
		\item Assume that $\A$ is positively graded of full rank. Modulo the lineality space, the cone $C$ is a simplicial cone.
Moreover, if $\A$ has full $\ZZ$-rank, $C$ is smooth modulo lineality space.
	\end{enumerate}
\end{cor}
\begin{proof}
Our strategy will be to first prove all the claims of the corollary in the special case that $\A$ is positively graded and of full rank. We will then use this to deduce the first and second claims in the general case, that is, dropping the positively graded and full rank assumptions.

To begin, let us thus assume that $\A$ is positively graded and of full rank.
Let $\widetilde \Sigma$ be the fan describing the normalization of the closure of $\TT.[\A]$ in $\Hilb_\A^H$. By item \ref{part:h4}  of Theorem \ref{thm:hilb},
	there is a maximal cone $\widetilde {C}$ of $\widetilde \Sigma$ corresponding to $I_{\K*\PS(\W)}$. 
	Applying Proposition \ref{prop:grob},
	the preimage under $\eta$ of $\widetilde{C}$ in the Gr\"obner fan $\Sigma$ is the desired cone $C$. This shows the first claim (under the assumption of positive grading and full rank).

	Again by item \ref{part:h4} of Theorem \ref{thm:hilb}, we know that the cone $\widetilde {C}$ is smooth, since it corresponds to a smooth chart of $\overline{\TT.[\A]}$. In particular, the generators of the rays of $\widetilde{C}$ are linearly independent. It follows that the cone $C$, modulo lineality space, is simplicial. This shows the third claim in the full rank case.

	Let us now additionally assume that $\A$ has full $\ZZ$-rank. By Corollary \ref{cor:fullZ}, this implies that $H$ is a torus, or equivalently, that $M$ is torsion free. 
The torus $\widetilde \TT$ may be identified with $\TT/H$, so we have an exact sequence of character groups
\[
	0\to M_{\widetilde \TT} \to \ZZ^{p+q}\to M\to 0.
\]
Here $M_{\widetilde \TT}$ is the character lattice of $\widetilde \TT$, dual to $N_{\widetilde \TT}$, and $\ZZ^{p+q}$ is the character lattice of $\TT$. The map $M_{\widetilde \TT} \to \ZZ^{p+q}$ is dual to the map $\eta:N_{\TT}\to N_{\widetilde \TT}$.
 Since $M$ is torsion free, the above sequence splits, and we obtain a cosection of the map $M_{\widetilde \TT} \to \ZZ^{p+q}$. Dually, this yields 
a section of the map $\eta:N_{\TT}\to N_{\widetilde \TT}$. This implies that $C$ is smooth modulo lineality space, completing the proof of the third claim of the corollary.

Remaining under the assumption of positive grading and full rank, we now show the second claim of the corollary. Consider the cone $\widetilde C^\vee$ dual to $\widetilde C$. The generators of $\widetilde C^\vee$ are the degrees (in $M_{\widetilde \TT}$) of the universal coefficients $t_1,\ldots,t_p$. Since $M_{\widetilde \TT}$ is a sublattice of $\ZZ^{p+q}$, we will view these degrees as elements of the latter lattice. By \cite[Lemma 12.7 and Lemma 12.8]{FZ_clustersIV}, these degrees are exactly of the form 
$\deg_\TT(x\cdot x')-\deg_\TT \alpha_1$, subject to the conditions in the statement of the corollary. Since we obtain $C$ by taking the dual of the image of $\widetilde C^\vee$ in $\RR^{p+q}$, the second claim follows.

Having shown all the claims under the assumptions of a positive grading and full rank, we now return to the general setting of $\A$ being an arbitrary cluster algebra of finite cluster type. By Lemma \ref{lemm:add_frozens}, we can add frozen variables to obtain a positively graded full rank cluster algebra $\A'$. Since $\A'$ is positively graded, the Gr\"obner cone $C'$ for $\A'$ contains in its interior a weight $w'\in\ZZ_{\geq 0}^{p+q'}$. The initial term with respect to $w'$ of each element of $I_{\A'}$ is in $I_{\K*\PS(\W')}$, and does not involve any of the extra frozen variables. Projecting $w'$ to $w\in \ZZ_{\geq 0}^{p+q}$, it follows that the initial terms with respect to $w$ of each element of $I_\A$ is in $I_{\K*\PS(\W)}$. It follows that $I_{\K*\PS(\W)}$ is an initial ideal of $I_\A$, and the first claim of the corollary follows in full generality.

For the second claim of the corollary in the general case, it is straightforward to verify that the elements listed in the claim must lie in $C^\vee$, the dual cone of $C$. We now show that they generate $C^\vee$. To that end, let $\sigma$ be the cone generated by these elements. Let $w\in\RR^{p+q}\cap \sigma^\vee$. Extend $w$ to the unique element $w'\in\RR^{p+q'}$ by giving all coordinates corresponding to a new frozen variable value $0$. Then from the description of the Gr\"obner cone $C'$ for $\A'$, we see that $w'\in C'$. It follows by the discussion of the preceding paragraph that $w\in C$. Hence $\sigma^\vee \subseteq C$, and we conclude that $C^\vee\subseteq \sigma$. The second claim of the corollary now follows.
\end{proof}

\begin{ex}[Continuation of Running Example \ref{ex:r7}]\label{ex:r8}
	For the cluster algebra of Example \ref{ex:r1}, the  Gr\"obner cone $C$ corresponding to $I_{K*\PS(\W)}$ is the cone dual to the degrees of the $t_i$ from Example \ref{ex:r8}. 
One computes that $C$ is the cone whose lineality space is spanned by the rows of
\[\begin{pmatrix}
		-1&1&0&0&1&0&0&1\\
		0&2&-1&2&0&0&1&1\\
      1&1&1&1&1&1&1&1\end{pmatrix}\]
and has positive directions generated by the rows of
\[\begin{pmatrix}
      -1&1&0&0&1&0&0&0\\
      1&-1&1&0&0&0&0&0\\
      0&1&-1&1&0&0&0&0\\
      0&0&1&-1&1&0&0&0\\
      1&0&0&1&-1&0&0&0\end{pmatrix}.\]
We are again ordering variables as $x_{25}$ ,$x_{13}$, $x_{24}$, $x_{35}$, $x_{14}$, $s_1$, $s_2$, $s_3$.
We see immediately that this cone is in fact smooth modulo lineality space, as Corollary \ref{cor:grob} predicts.
\end{ex}

In general, if $\A$ has full rank but not full $\ZZ$-rank, the cone $C$ of Corollary \ref{cor:grob} may not be smooth, as the following example shows.
\begin{ex}[Pullback of $\Gr(2,6)$]
	Consider the cluster algebra $\A$ associated to the seed 
	\begin{align*}	
	\begin{array}{c c c c c c c c c c c c c}
		\multirow{3}{*}{$\tB^\tr=\Bigg($}& 0 & 1 & 0 & -2 & 2 & -2 & 0 & 0 & 0& \multirow{3}{*}{$\Bigg)$}\\
			&-1& 0 & 1 & 0 & 0 & 2 & -2 & 0 & 0\\
			&0 & -1 & 0 & 0 & 0 & 0 & 2 & -2 & 2\\
			\\
			\tx=(&x_{13},&x_{14},&x_{15},&y_{12},&y_{23},&y_{34},&y_{45},&y_{56},&y_{16}&)
		\end{array}.
		\end{align*}
Note that $\A$ has full rank, but not full $\ZZ$-rank.
		It is straightforward to check that $\A$ may be obtained from the coordinate ring of $\Gr(2,6)$ by replacing each frozen Pl\"ucker variable $x_{i(i+1)}$ with $y_{i(i+1)}^2$.

	Although the Gr\"obner cone $C$ for $I_\A$ is simplicial modulo lineality space, we claim it is not smooth. Indeed, this equivalent to seeing that the dual cone $C^\vee$ is not smooth. We may describe primitive generators for the rays of $C^\vee$ as in Corollary \ref{cor:grob}. This gives $9$ rays in $\ZZ^{15}$, which we may encode as the columns of a $15\times 9$ matrix. The first $9$ rows correspond to the cluster variables. Since omitting the frozen variables from $\A$ leads to a cluster algebra of non-full rank, it follows that the determinant of the top $9\times 9$ submatrix vanishes (cf. Remark \ref{rem:deficient}).
	On the other hand, since the variables $y_{ij}$ are always appearing in any exchange relation with a power divisible by two, it follows that any other maximal minor of this matrix is divisible by $2$. Hence $C^\vee$ is not smooth.
\end{ex}

We record another algebraic consequence of our results:
\begin{cor}\label{cor:t2}
Let $\A$ be a skew-symmetric cluster algebra of finite cluster type. Then $T^2(\A)=0$.
\end{cor}
\begin{proof}
	By Corollary \ref{cor:grob}, there is a $\KK^*$-equivariant family over $\Aff^1$ with special fiber $\spec(S_{\K*\PS(\W)})$ and all other fibers isomorphic to $\spec(\A)$. The claim now follows from semi-continuity of $T^2$ and Theorem \ref{thm:unobstructed}.
\end{proof}

\begin{rem}\label{rem:BMN}
In \cite{BMN}, Bossinger, Mohammadi, and N\'ajera Ch\'avez show that for the coordinate rings of $\Gr(2,n+3)$ and $\Gr(3,6)$, one may obtain the corresponding cluster algebras with universal coefficients via a construction from Gr\"obner theory. In light of Example \ref{ex:hilb}, our Theorem \ref{thm:hilb} is a generalization of this, as we now explain.

	Given a weighted homogeneous ideal $J$ of a polynomial ring $\KK[\bz]$ and a maximal cone $C$ in the Gr\"obner fan of $J$, the authors of \cite{BMN} construct a flat family over $\Aff^s$ whose generic fiber is $V(J)$ \cite[Theorem 3.14]{BMN}. The  cluster algebras with universal coefficients $\A^\univ$ are then realized for
	$\Gr(2,n+3)$ and $\Gr(3,6)$ by applying this construction to a well-chosen cone of the Gr\"obner fan for the ideal $I_\A$ of the appropriate cluster algebra $\A$ \cite[Theorem 1.3, Theorem 1.4]{BMN}. That this construction agrees with $\A^\univ$ is verified by an explicit combinatorial computation.

	A careful study of the flat family of \cite{BMN} reveals that it may be constructed as follows. With notation as in the beginning of this section,
	the choice of $C$ determines a torus fixed point $P$ of $\Hilb_J$. Consider the $\TT$-orbit closure in $\Hilb_J$ of $[J]$. There is a unique affine torus-invariant chart $W$ of this orbit closure that contains $P$.  Let $d$ be the number of connected components of the stabilizer $\TT_{[J]}$; this is equal to the number of connected components in a fiber of the map of tori $\TT\to \widetilde{\TT}$.

	The flat family of \cite{BMN} may be identified with the pullback of the universal family over $\Hilb_{J}$ to the Cox torsor over a degree-$d$ cover of the normalization of $W$.
	To check this, it suffices to show that the restriction of the family from \cite{BMN} to $(\KK^*)^s\subset \Aff^s$ has the desired form. But on this open set, the family of \cite{BMN} is relatively easy to describe; we leave it to the reader to check the details. 

	In the instances where $\A$ is a positively graded cluster algebra of full $\ZZ$-rank and finite cluster type, we showed in the course of the proof of Theorem \ref{thm:hilb} that the chart $W$ is just a copy of affine space $\Aff^p$. 	Hence, it agrees with the Cox torsor of its normalization.
	Furthermore, by Corollary \ref{cor:fullZ} it follows that $d=1$.
We  thus recover
\cite[Theorem 1.3, Theorem 1.4]{BMN}.

Suppose instead that $\A$ is only of full rank, but not full $\ZZ$-rank. 
In such cases, the construction of \cite{BMN} cannot produce the universal cluster algebra, since their family arises by pulling back along (among other things) a degree $d$ cover for some $d>1$. 
\end{rem}

\section{Universal Coefficients and Deformations}\label{sec:def}
\subsection{The Characteristic Map}
By Theorem \ref{thm:basis}, the family $\spec \A^{\univ}\to \Aff^p$ is flat, and by Theorem \ref{thm:hilb}, the fiber over $0$ is $\spec S_{\K*\PS(\W)}$. Thus, this family is a \emph{deformation} of $Y=\spec S_{\K*\PS(\W)}$.
Our next goal will be to relate $\A^{\univ}$ to the semiuniversal deformation of $Y$. Before doing this, we need to analyze the first order deformations induced by $\A^{\univ}$.

\begin{defn}[Property T1]\label{defn:T1}
	We say a matrix $\tB=(b_{ij})\in\Mat_{m\times n}(\ZZ)$ fulfills property T1 if the following holds:

	For every $1\leq j \leq n$ and every $w$ in the $\ZZ$-span of the columns of $\tB$ satisfying $w_j=0$ and 
	\begin{equation}\label{eqn:t1}
w_i\geq 
\begin{cases}
	1-\max(0,b_{ij})&	i \leq n,\ b_{ij}\neq 0\\
	-\max(0,b_{ij})&	\textrm{else}\\
\end{cases}
\end{equation}
for all $i\neq j$,
it follows that either $w=0$ or $w=-B_j:=(-b_{1j},\ldots,-b_{mj})$.

We say that a cluster algebra $\A$ fulfills property T1 if every extended exchange matrix fulfills property T1.
\end{defn}

\begin{rem}\label{rem:sinksource}
	Suppose that  $w=0$ satisfies \eqref{eqn:t1} for some fixed $j$. It follows that $b_{ij}$  is non-negative for $i\leq n$. Likewise,
	suppose that $w=-B_j$ satisfies \eqref{eqn:t1} for some fixed $j$. It follows that $b_{ij}$  is non-positive for $i\leq n$. In these situations, the 
vertex $j$ of the graph $\Gamma(B)$ is respectively a sink or a source.
\end{rem}
\begin{ex}[Continuation of Running Example \ref{ex:r8}]\label{ex:r9}
	We check the T1 property for the extended exchange matrix $\tB$ from Example \ref{ex:r1}.
Consider first $j=1$.	For $w$ to be in the $\ZZ$-span of the columns of $\tB$ and satisfy $w_1=0$, $w$ must be a multiple of the first column of $\tB$. 
Imposing \eqref{eqn:t1} for $j=1$ imposes the (in)equalities
	\begin{align*}
w_2\geq 1,\ w_3\geq -1, w_4\geq 0, w_5\geq -1.
	\end{align*}
The only way to satisfy this is by taking $w$ as the $-1$ times the first column of $\tB$.

Consider instead $j=2$. For $w$ to be in the $\ZZ$-span of the columns of $\tB$ and satisfy $w_2=0$, $w$ must be a multiple of the second column of $\tB$. 
	Imposing \eqref{eqn:t1} for $j=0$ imposes the (in)equalities
	\begin{align*}
w_1\geq 0,\ w_3\geq -1, w_4\geq 0, w_5\geq 0.
	\end{align*}
The only way to satisfy this is by taking $w=0$.

We thus see that property T1 is satisfied for $\tB$. One may similarly check that the four other extended exchange matrices for $\A$ also satisfy property T1. Hence, the cluster algebra $\A$ satisfies property T1.
\end{ex}

Let $\A$ be a cluster algebra of finite cluster type.
We say $\A$ has an \emph{isolated vertex} if for some (or equivalently every) extended exchange matrix $\tB=(b_{ij})$ of $\A$, there is an index $1\leq k \leq n$ such that $b_{ij}=0$ whenever $i=k$ or $j=k$. 
Note that if $\A$ is positively graded, it cannot have an isolated vertex.

Let $J=I_{\K*\PS(\W)}$. As in the proof of Lemma \ref{lemma:differential}, there is a natural map
$T_0\spec \KK[\bt] \to \Hom(J,\KK[\bz]/J)_0$. Indeed, using the notation in the proof of Lemma \ref{lemma:differential}, the ideal $J$ is just the ideal $I(0)$. Here the degree $0$ part of $\Hom(J,\KK[\bz]/J)$ is taken with respect to the $M$ grading.
Composing this with the surjection $\Hom(J,\KK[\bz]/J)\to T^1(\KK[\bz]/J)$, we obtain the \emph{characteristic map}
\[
T_0\spec \KK[\bt] \to T^1(\K*\PS(\W))^H.
\]
Here $T^1(\K*\PS(\W))^H$ denotes the $H$-invariant elements of $T^1$.

The main proposition of this subsection shows the importance of the T1 property:
\begin{prop}\label{prop:t1p}
Let $\A$ be a  cluster algebra of finite cluster type with no isolated vertices. 
Then the characteristic map $T_0\spec \KK[\bt] \to T^1(\K*\PS(\W))^H$ is an isomorphism if and only if $\A$ satisfies property T1.
\end{prop}

To prove the proposition, we will first consider several lemmas:

\begin{lemma}\label{lemma:injective}
Let $\A$ be a cluster algebra of finite cluster type with no isolated vertex.
\begin{enumerate}
	\item The characteristic map $T_0\spec \KK[\bt] \to T^1(\K*\PS(\W))^H$ is injective.

	\item The image of the characteristic map consists precisely of the graded pieces $T^1(\K*\PS(\W))_\bc$ for $\bc=\ba-\bb$, where $z^{\bb}=z_{x}\cdot z_{x'}$ for exchangeable cluster variables $x,x'$, and the exchange relation for $x,x'$ has the form
\[
x\cdot x'=\alpha_1+\alpha_2
\]
for extended cluster monomials $\alpha_1,\alpha_2$ such that $\ba$ is the $\ZZ^{\V\cup \W}$ degree of $\alpha_1$, and $\alpha_2$ only involves frozen variables.
\end{enumerate}
\end{lemma}
\begin{proof}
	Even though $\A$ is not necessarily positively graded, we will use the  discussion and notation from \eqref{eqn:primitive} and the proof of Lemma \ref{lemma:differential} (specializing to $P=0$). 
	Using Theorem \ref{thm:hilb}, we now know that $I(0)=J=I_{\K*\PS(\W)}$. In particular, $\Hom(J,\KK[\bz]/J)$ has a ``fine'' grading by $\ZZ^{\V\cup \W}$. 
	In the proof of Lemma \ref{lemma:differential}, we saw that the image of
	the tangent vector $\eta_i:\spec \KK[t_i]/t_i^2 \to \spec \KK[\bt]$ is a particular element $\psi_i\in \Hom(J,\KK[\bz]/J)_0$ which sends $z_x\cdot z_{x'}$ to $\alpha_1$ for a certain exchangeable pair $\{x,x'\}$.
	 	 Let $\ba$ be the degree of the monomial in $\bz$ corresponding to $\alpha_1$, and $\bb$ the degree of $z_x\cdot z_{x'}$. Then the graded piece of $\psi_i$ of degree $\ba-\bb$ is the homomorphism sending $z_x\cdot z_{x'}$ to $\alpha_1$, and all other exchange pair generators of $J$ to $0$. Here we are using that $\alpha_1$ and $\alpha_2$  from \eqref{eqn:primitive} are not equal; this follows from the assumption that $\A$ has no isolated vertex.

	We claim that the image of $(\psi_i)_{\ba-\bb}$ in $T^1(\K*\PS(\W))$ is non-zero. Indeed, consider any seed $(\tx,\tB)$ with $x=x_k$ in the cluster $\tx$ and $x'=x_k'$ obtained via mutation of this seed at $k$. Let $\Omega$ be the set of cluster variables appearing in $\alpha_1$.  It follows from the exchange relation \eqref{eqn:primitive} that $k$ is $\Omega$-isolated. By Lemma \ref{lemma:tensor} and the fact that $T^1(\KK[z_w\mid w\in \W])=0$,
$T^1(\K*\PS(\W))\cong T^1(\K)\otimes \KK[z_w\ |\ w\in\W]$. By Theorem \ref{thm:t1} $T^1(\K*\PS(\W))$ thus has a non-zero element of degree $\ba-\bb$, and by Remark \ref{rem:t1el}, that element is exactly $(\psi_i)_{\ba-\bb}$.

We have thus seen that each of the image of each of the $\psi_i$ in $T^1(\K*\PS(\W))^H$ has a non-zero homogeneous component, and the degrees of these components are pairwise disjoint. It follows that the map $T_0\spec \KK[t] \to T^1(\K)^H$ is injective.

Moreover, since the family $\spec \A^\univ\to \spec \KK[\bt]$ is $\TT$-equivariant, each $\psi_i$ is necessarily homogeneous. It follows that $\psi_i=(\psi_i)_{\ba-\bb}$.
The description of the degrees of the image now follows from  \eqref{eqn:primitive} and the above discussion.
\end{proof}

Consider some degree $\bc=\ba-\bb$ for which $T^1(\K*\PS(\W))\neq 0$. The degree $\bc$ is constructed exactly as in Theorem \ref{thm:t1}, except that $\ba$ may also have support involving the frozen variables $\W$. We wish to understand what it means for this degree to be $H$-invariant. This is the same thing as requiring that it be of degree zero with respect to the $M$-grading. We write $\bc\equiv 0$ in the $M$-grading when this is the case.

\begin{lemma}\label{lemma:deg0}
	The degrees $\bc=\ba-\bb$ arising as in Theorem \ref{thm:t1} from the seed $(\tx,\tB)$ with $T^1(\K*\PS(\W))_\bc\neq 0$, $\bc\equiv 0$ in the $M$-grading, and $b=\{x_k,x_k'\}$ are in bijection with those $w\in\ZZ^m$ such that $w$ is in the $\ZZ$-span of the columns  of $\tB$, $w_k=0$, and $w$ satisfies \eqref{eqn:t1} for $j=k$ and all $i\neq j$.  The positive part $\ba$ is given by $\ba_{x_i}=w_i+\max(0,b_{ik})$.
\end{lemma}
	\begin{proof}

	Let $(\tx,\tB)$ be a seed giving rise to the degree $\bc$ as in Theorem \ref{thm:t1}, that is, $\bb$ has support $\{x_k,x_k'\}$ for some cluster variable $x_k$ of $\bx$, $x_k$ is $\Omega$-isolated for some subset $\Omega$ of the elements of $\bx$, and the support of $\ba$ is the union of $\Omega$ with a subset of $\W$. 
	Recall from \S\ref{sec:clusterembedding} that we can identify $M$ with the cokernel of $\tB$. 
	The $M$-degree of $z_{x_k}\cdot z_{x_k'}$ is the same as the degree of 
	\[
		\prod_{i=1}^m z_{x_i}^{\max(0,b_{ik})}.
	\]
On the other hand, a monomial
\[
	\prod_{i=1}^m z_{x_i}^{d_i}
\]
having $M$-degree zero is equivalent to requiring that $(d_1,\ldots,d_m)^\tr$ is in the $\ZZ$-span of the columns of $\tB$.
Putting this together, we conclude that $\bc=\ba-\bb$ has degree zero if and only if the element $w\in \ZZ^m$ with
\[
	w_i=\ba_{x_i}-\max(0,b_{ik})
\]
is in the $\ZZ$-linear column span of $\tB$.

Let $\Omega'$ consist of those $x_i$ such that $i\leq n$ and $b_{ik}\neq 0$. The condition that $x_k$ be $\Omega$-isolated is exactly the condition that $\Omega'\subseteq \Omega$. This translates to the condition that $\ba_{x_i}\geq 1$ for those $i$ such that $i\leq n$ and $b_{ik}\neq 0$.
Combining this with the above discussion, we arrive at the statement of the lemma.\end{proof}

\begin{lemma}\label{lemma:bipartite}
Let $\A$ be a cluster algebra of finite cluster type with extended exchange matrix $\tB$. Fix a vertex $j$ that is a sink or a source of $\Gamma(B)$. Then there exists a sequence of mutations avoiding $j$ and its neighbors leading to an exchange matrix $B'$ such that every vertex of $\Gamma(B')$ is a sink or a source.
\end{lemma}
\begin{proof}
It is enough to treat the case where $\A$ has no frozen variables and is connected. We may assume moreover that $j$ is a source, the sink case being analogous. Let $c$ be the number of neighbors of $j$.

We claim that if we remove $j$ from $\Gamma(B)$ we obtain a graph with $c$ connected components. 
Assume to the contrary that there are fewer connected components. Then there would be at least two neighbors of $j$, say $i$ and $k$, that belong to the same component. Consider a minimal subgraph in this component containing both $i$ and $k$. By adding back the vertex $j$ (and the corresponding edges), we obtain a non-oriented cycle as a full subgraph of $\Gamma(B)$.
However, this contradicts $\A$ being of finite cluster type by \cite[Proposition 9.7]{FZ_clustersII}.
Hence, $\Gamma(B)$ has exactly $c$ connected components after removing $j$.

To prove the lemma, it remains to check that while working on a component of $\Gamma(B)\setminus \{j\}$, we can perform mutations avoiding the unique neighbor $k$ of $j$ contained in that component in order to make this neighbor $k$ a sink.
This follows at once from \cite[Proposition 11.1 (1)]{FZ_clustersIV} and \cite[Theorem 6.2]{denom}.
\end{proof}

\begin{proof}[Proof of Proposition \ref{prop:t1p}]
	First, suppose that there exists a seed $(\tx,\tB)$, an index $1\leq j\leq n$, and $w\in \ZZ^n$ in the $\ZZ$-column span of $\tB$ satisfying $w_j=0$ and \eqref{eqn:t1} for all $i\neq j$ but such that $w\neq 0$ and $w\neq -B_j$. By Lemma \ref{lemma:deg0}, there is an $H$-invariant element of $T^1(\K*\PS(\W))$ which is the image of the homomorphism sending $z_{x_j}\cdot z_{x_j'}$ to 
	\[
		\prod_{i=1}^m z_{x_i}^{w_i+\max(0,b_{ij})}.
	\]
	However, since $w\neq 0$ and $w\neq -B_j$, this monomial is not one of the monomials appearing in the exchange relations for $\{x_k,x_k'\}$, and as such, this element of $T^1(\K*\PS(\W))$ cannot be in the image of the map in question.

	Conversely, let us suppose that for every seed $(\tx,\tB)$, $\tB$ satisfies property T1.
	By Lemma \ref{lemma:injective}, the map $T_0\spec \KK[\bt] \to T^1(\K*\PS(\W))^H$ being an isomorphism is equivalent to showing that $\dim T^1(\K*\PS(\W))^H \leq p$.
We will do this by showing that to every element in a basis of  $T^1(\K*\PS(\W))^H$, we may associate a distinct universal coefficient $t_i$.

To that end, consider any degree $\bc$ with $T^1(\K*\PS(\W))_\bc\neq 0$ and $\bc\equiv 0$ with respect to the $M$ grading. Suppose that this $T^1$ element comes from the seed $(\tx,\tB)$, with $b=\{x_k,x_k'\}$. This degree $\bc$ corresponds to $w\in \ZZ^n$ as in the statement of Lemma \ref{lemma:deg0}, with $\ba_{x_i}=w_i+\max(0,b_{ik})$. Since we are assuming that the T1 property is satisfied, 
either $w=0$ or $w=-B_j$. 

 By Remark \ref{rem:sinksource}, the vertex $k$ of $\Gamma(B)$ is a sink or a source.
By Lemma \ref{lemma:bipartite}, we may find a sequence of mutations in vertices distinct from $k$ and its neighbors that allow us to replace the seed $(\tx,\tB)$ with a new seed $(\tx',\tB')$ where every vertex of $\Gamma(B')$ is a sink or a source. We may thus assume from the start that without loss of generality,
$\Gamma(B)$ has every vertex a sink or a source. In the language of \cite{FZ_clustersIV}, this seed lies on the \emph{bipartite belt} of the cluster complex.

Thus, to the degree $\bc$, we may associate a universal coefficient coming from the pair $\{x_k,x_k'\}$ as described in \cite[Lemma 12.7 and Lemma 12.8]{FZ_clustersIV}. There is a single such coefficient, unless $k$ is both a sink and a source in $\Gamma(B)$, in which case there are two. Similarly, by the above discussion, the T1 property guarantees that for any pair $\{x_k,x_k'\}$, there is only one $H$-invariant element of $T^1$ of with $b=\{x_k,x_k'\}$, unless $k$ is both a sink and source of $\Gamma(B)$, in which case there are two. This proves the proposition.
\end{proof}

\begin{rem}\label{rem:add}
	Let $\A$ be a full rank cluster algebra, and suppose that one of its extended exchange matrices $\tB$ satisfies the T1 property. If we add additional frozen variables to $\A$, then the new extended exchange matrix $\tB'$ corresponding to $\tB$ will still satisfy the T1 property. Indeed, we can view $\tB'$ as having the same first $m$ rows as $\tB$, with $m'-m$ additional rows. Fix $1\leq j \leq n$. Consider any element $w$ of $\ZZ^{m'}$ in the $\ZZ$-linear column span of $\tB'$ satisfying \eqref{eqn:t1} for the matrix $\tB'$ to $\ZZ^m$. Projecting to $\ZZ^m$, we see that the first $m$ coordinates of $w$ are either all $0$, or all $-b_{ij}$.  Since $\tB$ has full rank, there is a unique way to obtain  the projection of $w$ as a linear combination of the columns of $\tB$; $w$ itself must be that same linear combination of the columns of $\tB'$. It follows that either $w=0$ or $w_i=-b_{ij}$ for all $i$.
Hence, $\tB'$ fulfills property T1.
\end{rem}

\begin{prop}\label{prop:addt1}
Let $\A$ be a cluster algebra of finite cluster type. Then we can always add additional frozen variables to $\A$ so that the new cluster algebra $\A'$ is positively graded and satisfies property T1.
\end{prop}
\begin{proof}
	By adding frozen variables, we can without loss of generality assume that $\A$ has full rank and is positively graded, see Lemma \ref{lemm:add_frozens}.

	Let $\tB$ be an extended exchange matrix of $\A$ such that property T1 is not satisfied. Let $j$ be an index for which there exists an element $w$ in the $\ZZ$-linear column span of $\tB$ satisfying \eqref{eqn:t1} but such that $w\neq 0,-B_j$. We note that since $\A$ is positively graded, only finitely many such $w$ exist.

	We may write $w$ uniquely as a $\ZZ$-linear combination $w=\sum_i \lambda_i \cdot B_i$ of the columns of $\tB$, and by assumption, we must have $\lambda_k\neq 0$ for some $k$ with $k\neq j$. We add a new coefficient $x_{m+1}$ by adding a row to the matrix $\tB$ and setting 
	\[
		b_{(m+1)i}=\begin{cases}
			1 & i=k, \lambda_k<0\\
			-1 & i=k, \lambda_k>0\\
			0 & i\neq k
		\end{cases}.
	\]
	Taking $w'=\sum_i \lambda_i \cdot B_i'$ where $B_i'$ are the columns of this new matrix $\tB'$, we see that $w'$ does not satisfy \eqref{eqn:t1}.
It follows that there is one less obstruction to $\tB'$ satisfying the T1 property as there is for $\tB$. Thus, we can continue adding coefficients until we arrive at an exchange matrix $\tB'$ that does satisfy the T1 property. 

We may now do the same thing at all the other (finitely many) extended exchange matrices of $\A$. By Remark \ref{rem:add}, once we have added enough frozen variables so that one exchange matrix satisfies T1, adding further coefficients doesn't destroy this property. Finally, we may add more frozen variables if necessary so that $\A'$ again becomes positively graded (Lemma \ref{lemm:add_frozens}). Again, by Remark \ref{rem:add}, this will not destroy the T1 property.
\end{proof}

\begin{ex}
It is straightforward to show that any positively graded cluster algebra of type $A_1$ or $A_2$ will always satisfy the T1 property for every extended exchange matrix. This is no longer the case for type $A_3$, as the following example demonstrates.

Consider the matrix 
\[
\tB=	\left(\begin{array}{c c c}
		0 & -1 & 0\\
		1&0&-1\\
		0&1&0\\
		-1&0&1\\
		1&0&-2\\
		0&0&1
	\end{array}\right).
\]
The corresponding quiver is as follows (frozen vertices are in gray boxes):
\[
\xymatrix{
	\color{lightgray}{\fbox{4}}\ar[drr] & & \color{lightgray}{\fbox{6}}\ar[d]\\
1\ar[u] & 2\ar[l] & 3\ar[l] \ar@<-0.5ex>[ld] \ar@<0.5ex>[ld]\\
& \color{lightgray}{\fbox{5}}\ar[lu] & 
}.
\]
Taking $j=1$, the element $w=(0,0,0,0,-1,1)^\tr$ is in the column span of $\tB$ and satisfies \eqref{eqn:t1}, showing that the T1 property is not fulfilled for this matrix.
Furthermore, one may check that it is possible to endow the cluster algebra $\A(\tB)$ with a positive grading.
\end{ex}

\begin{rem}\label{rem:ex}
	Instead of looking at exchange matrices directly, one may verify property T1 by looking at all exchange monomials $x_k\cdot x_k'$ for a cluster algebra $\A$, as we now describe. Consider a homogeneous element of $T^1(\K*\PS(\W))^H$. By Theorem \ref{thm:t1}, its negative part corresponds to some exchange monomial $x_k\cdot x_k'$. 
	Its positive part corresponds to an extended cluster monomial $x^\alpha$. Again by Theorem \ref{thm:t1}, each cluster variable appearing in $x^\alpha$ must be compatible with both $x_k$ and $x_k'$. Furthermore, considering the exchange relation for $x_k,x_k'$, any cluster (but not frozen) variable appearing on the right hand side of \eqref{eq:exchange_rel} must be a factor of $x^\alpha$: this corresponds to $k$  being $\Omega$-isolated.
Here we are using the fact that for every exchange pair, there is a unique exchange relation \cite[Theorem 1.11]{FZ_clustersII}.
Finally, the $M$-degrees of $x^\alpha$ and $x_k\cdot x_k'$ must agree in order for this element to be $H$-invariant.

In light of Lemma \ref{lemma:deg0}, verifying the T1 property for every extended exchange matrix of $\A$ is equivalent to showing that for every exchange monomial $x_k\cdot x_k'$, the only extended  cluster monomials $x^\alpha$ satisfying the conditions of the previous paragraph are exactly one of the two monomials on the right hand side of the exchange relation \eqref{eq:exchange_rel}. See Example \ref{ex:t1} for a demonstration of this criterion.
\end{rem}

\begin{ex}\label{ex:t1}
	One may show that for each of the cluster algebras (1)--(6) listed in Example \ref{ex:hilb}, the T1 property holds. 
Furthermore, property T1 also holds for the $D_n$ case of the Schubert divisor of $\Gr(2,n+2)$.

	For the individual cases $\Gr(3,6)$, $\Gr(3,7)$, and $\Gr(3,8)$ we checked this using a computer calculation. 
	This involved generating all seeds via iterated mutation, then checking the T1 property for each extended exchange matrix. This amounts to checking that the lattice points in various polyhedra have a specific form and is straightforward to carry out with software such as \cite{polyhedra}.

	For the other families of examples, one may use the approach of Remark \ref{rem:ex}.
	We show how this works in the $D_n$ case of the Schubert divisor of $\Gr(2,n+2)$ \cite[Example 6.3.5]{FWZ_chapter6}. The other cases are similar.

Following the notation of \cite[\S12.4]{FZ_clustersII}, the cluster and frozen variables for the Schubert divisor are 
\begin{align*}
	x_{ab}=x_{\overline a \overline b} \qquad 1\leq a, b \leq 2n,\ a\neq b;\\
	x_{a \overline a}=x_{\overline a a} \qquad 1\leq a \leq n;\\
	\overline x_{a \overline a}=\overline x_{\overline a a}\qquad 1\leq a \leq n
\end{align*}
where $\overline a$ is the remainder of $a+n$ modulo $2n$. The frozen variables are exactly the $x_{ab}$ where $|a-b|\equiv 1$ modulo $2n$.
Taking $e_0,\ldots,e_n$ as the standard basis for $\ZZ^{n+1}$, we may define a grading via
\begin{align*}
	\deg_H x_{ab}=\deg x_{a\overline b}=\deg_H x _{\overline a b}=\deg_H x_{\overline a \overline b}=e_a+e_b \qquad &1\leq a, b \leq n,\ a\neq b;\\
	\deg_H x_{a \overline a}=e_0+e_a \qquad &1\leq a \leq n;\\
	\deg_H \overline x_{a \overline a}=-e_0+e_a \qquad &1\leq a \leq n.
\end{align*}
There are five types of exchange relations, which we consider in turn. Each of these is homogeneous with respect to the above grading.
We will view the numbers $1,\ldots,2n$ as the labels for the vertices of a $2n$-gon, arranged in counter-clockwise order. We will always work modulo $2n$.
\begin{enumerate}
	\item\label{case:one} For $a,b,c,d,\overline a$ in counterclockwise order,
		\[
			x_{ac}x_{bd}=x_{ab}x_{cd}+x_{ad}x_{bc}.
		\]
		Consider any corresponding $T^1$ element.
		Since any $x_{ad}$ is not frozen, any perturbation of $x_{ac}x_{bd}$ must involve it. The product of the remaining variables in the perturbation must have degree $e_b+e_c$. The monomials of this degree are $x_{bc}$, $x_{b\overline c}$,  $x_{b\overline b}\overline x_{c\overline c}$, and 
		$\overline x_{b\overline b}\overline x_{c\overline c}$. The latter two monomials are incompatible, and $x_{b\overline c}$ is incompatible with $x_{ad}$. Thus, the only possible perturbation is $x_{ad}x_{bc}$ as desired.

	\item For $a,b,c,\overline a$ in counterclockwise order,
		\[
			x_{ac}x_{a\overline b}=x_{ab}x_{a\overline c}+x_{a\overline a }\overline x_{a\overline a}x_{bc}.
		\]
		Consider any corresponding $T^1$ element.
		The perturbation must involve $x_{a\overline a }\overline x_{a\overline a}$; the product of the remaining variables in the perturbation must have degree $e_b+e_c$. The monomials of this degree were listed in case \ref{case:one} above. As noted, the latter two are incompatible, and $x_{b\overline c}$ is incompatible with $x_{a\overline a}$. Thus, the only possible perturbation is $x_{a\overline a }\overline x_{a\overline a}x_{bc}$ as desired.
	\item For $a,b,\overline a$ in counterclockwise order,
		\[
			x_{a\overline a}\overline x_{b\overline b}=x_{ab}+x_{a\overline b}.
		\]
		Consider any corresponding $T^1$ element.
		The perturbation must be $x_{ab}$,  $x_{a\overline b}$, or $\overline x_{a\overline a} x_{b\overline b}$, but the latter is incompatible.

	\item \label{case:four} For $a,b,c,\overline a$ in counterclockwise order,

		\[x_{a\overline a} x_{b\overline c}=x_{ab}x_{c\overline c}+x_{a\overline c}x_{b\overline b}
		\]
		Consider any corresponding $T^1$ element.
		The perturbation must involve both $x_{c\overline c}$ and $x_{b\overline b}$. The product of the remaining variables in the perturbation
		must have degree $e_a-e_0$. The only possibility is $\overline x_{a\overline a}$. But this is incompatible with $x_{c\overline c}$, so there can be no perturbation at all.
	\item For $a,b,c,\overline a$ in counterclockwise order,

		\[\overline x_{a\overline a} x_{b\overline c}=x_{ab}\overline x_{c\overline c}+x_{a\overline c}\overline x_{b\overline b}
		\]
		This is similar to case \ref{case:four}.
\end{enumerate}
We have thus verified for each exchange relation that any corresponding element of $T^1(\K*\PS(\W))^H$ must have the form described in Remark \ref{rem:ex}. It follows that the T1 property holds for this example.
\end{ex}

\subsection{Derivations}
To construct a \emph{canonical} $\TT$-equivariant family with prescribed characteristic map, we will need control over the embedded deformations induced by derivations.
We again consider a cluster algebra $\A$ of finite cluster type. Let $H$ be the quasitorus fixing $\spec \A$ in its cluster embedding, with character group $M$. We will use other notation as in \S\ref{sec:clusterembedding}. In particular, we consider the cluster complex $\K$ of $\A$.

Consider the Stanley-Reisner ideal $J=I_{\K*\PS(\W)}$.
For a cluster variable $v\in \V$ and a monomial $z^\alpha\in \KK[\bz]$, we say that the derivation
\[
\partial(v,\alpha):=z^\alpha\cdot \frac{\partial}{\partial z_v}\in \Der_{\KK}(\KK[\bz],\KK[\bz])
\]
is $J$-non-trivial if its image in $\Hom_{\KK[\bz]}(J,S_{\K*\PS(\W)})$ is non-zero.
More concretely, $\partial(v,\alpha)$
is $J$-non-trivial if
there exists a cluster variable $w$ such that 
\begin{equation}\label{eqn:T0}
	z_v\cdot z_w\in J\qquad\textrm{and}\qquad z_w\cdot z^\alpha\notin J.
\end{equation}

Let $\SG\subset \ZZ^{\V\cup \W}$ be the semigroup generated by the degrees of the non-trivial pieces of the image of the characteristic map \[T_0\spec \KK[\bt]\to T^1(\K*\PS(\W))^H.\]
Here we are taking degrees with respect to the $\TT$-action, and denote degree by $\deg_\TT$.
In an earlier version of this paper, we conjectured the following:
\begin{conj}[{\cite[cf. Conjecture 6.2.9]{oldversion}}]\label{conj:T0}
Let $\A$ be a cluster algebra of finite cluster type. Consider any cluster variable $v$ and monomial $z^\alpha$ such that
$\partial(v,z^\alpha)$ is $J$-non-trivial. Then
\begin{equation}
\deg_\TT z^\alpha-\deg_\TT z_v \notin \SG.
\end{equation}
In other words, the non-zero homogeneous pieces of the image of the map \[\Der_{\KK}(\KK[\bz],\KK[\bz]) \to \Hom_{\KK[\bz]}(J,S_{\K*\PS(\W)})\] all have degree outside of $\SG$.
\end{conj}
\noindent The first author and So prove conjecture \ref{conj:T0} in \cite{karolyn}. We will make use of this result in the following section.

\subsection{Exchange Minimal Deformations}
We continue with the notation of the previous section. We wish to identify a canonical deformation of \[Y=\spec S_{\K*\PS(\W)}\subseteq \Aff^{p+q}\] whose characteristic map has the same image as
the image of the characteristic map for the family $\spec \A^\univ\to \Aff^p$. 

To that end, we will introduce the notion of \emph{exchange minimality}. 
Recall that cluster variables $v$ and $w$ are exchangeable if there are two clusters for $\A$ linked by a mutation that interchanges $v$ and $w$.
It is possible to see which cluster variables $v$ and $w$ are exchangeable by only inspecting the cluster complex $\K$. Indeed,
$v$ and $w$ are exchangeable if and only if $\{v,w\}\notin \K$ and there exists a face $f\in \K$ with $f\cup \{v\}$ and $f\cup \{w\}$ being maximal faces of $\K$.
Let $\mfm$ be the homogeneous maximal ideal of $\KK[\bt]$.
Fix a homogeneous subspace $V$ of $T^1(\K*\PS(\W))^H$.
\begin{defn}[Exchange Minimality]\label{defn:exmin}
	Let $\pi:\cY\to \spec \KK[\bt]/\mfm^{k+1}$ be a $\TT$-equivariant $H$-invariant embedded deformation of $Y$. Let $\widetilde J\subseteq \KK[z_v]\otimes \KK[\bt]/\mfm^{k+1}$ be the ideal of $\cY$ in $\Aff^{p+q}\times \spec \KK[\bt]/\mfm^{k+1}$.
	\begin{itemize}
\item Suppose $k=1$. We say the deformation $\pi$ is $1$-exchange minimal (with respect to $V$) if  the characteristic map of $\pi$ is injective with image $V$.
\item Suppose $k>1$. We say the deformation $\pi$ is $k$-exchange minimal (with respect to $V$) if the reduction of $\pi$ modulo $\mfm^k$ is $(k-1)$-exchange minimal, and for every exchangeable pair $v,w$, $z_v\cdot z_w$ has a lift to $\widetilde J/\mfm^{k+1}$ with the fewest number of non-zero terms among all possible lifts in all $(k-1)$-exchange-minimal families.
\item We say that a $\TT$-equivariant $H$-invariant embedded deformation $\pi:\cY\to\Aff^p$ of $Y$ is exchange minimal (with respect to $V$) if for every $k>0$, its reduction modulo $\mfm^{k+1}$ is $k$-exchange minimal.
\end{itemize}
\end{defn}

We now state our main result of this section:
\begin{thm}\label{thm:versal}
Let $\A$ be a cluster algebra of finite cluster type with no isolated vertex. Let $V\subseteq T^1(\K*\PS(\W))^H$ be the image of the characteristic map of
the family $\spec \A^\univ \to \Aff^p$.
\begin{enumerate}
	\item\label{part:one} There is a unique (up to canonical embedded isomorphism) exchange-minimal deformation $\pi:\cY\to \Aff^p$ of $Y$. There is a canonical $\TT$-equivariant isomorphism  $\phi:\Aff^p\to \Aff^p$ under which the family $\spec \A^\univ \to \Aff^p$ is canonically identified with $\pi$.
	\item If additionally $\A$ satisfies property T1, the above family $\pi$ is semiuniversal. In particular, the family $\spec \A^\univ \to \Aff^p$ may be identified with a canonical $H$-invariant semiuniversal deformation of $Y$. 
	\end{enumerate}
\end{thm}
\begin{proof}
	By Theorem \ref{thm:t1} and Remark \ref{rem:t1el}, there is a canonical $\KK$-vector space decomposition
\[\Hom(J,S_{\K*\PS(\W)})^H=T^1(\K*\PS(\W))^H\oplus D\]
where $D$ is the image of $\Der_\KK(\KK[\bz],S_{\K*\PS(\W)})^H$. Under this decomposition, any homogeneous $T^1$ element only affects a single minimal generator of $J$, and that generator is always coming from an exchangeable pair of cluster variables.
By Conjecture \ref{conj:T0}, any non-trivial homogeneous element of $D$ has degree outside of $\SG$.
Furthermore, each homogeneous piece of $\Hom(J,S_{\K*\PS(\W)})$ (with respect to the $\ZZ^{\V\cup \W}$-grading) is at most one-dimensional.

Let $\cY$ be any $k$-exchange-minimal deformation, and let $\cY'=\spec \A^{\univ}$. Consider their ideals $\widetilde{J}$ and $\widetilde{J}'$.
Since exchange-minimal deformations have injective characteristic maps with image $V$, and $\cY'$ also has injective characteristic map by Lemma \ref{lemma:differential}, $\cY$ and $\cY'$ modulo $\mfm^2$ are isomorphic as abstract deformations, differing by a unique linear automorphism $\phi$ of $\Aff^p$.
By applying the automorphism $\phi$, we may reduce to the case that the characteristic maps of $\cY$ and $\cY'$ are equal.
After this reduction, we claim that modulo $\mfm^{k+1}$, $\cY=\cY'$ as embedded deformations.

Suppose that we have shown that $\cY$ and $\cY'$ agree after reducing by $\mfm^{\ell}$ for some $1\leq \ell < k$. Then there are some $\bt^{\alpha_i}\in\mfm^\ell$ and homogeneous $\psi_i\in \Hom(J,S_{\K*\PS(\W)})^H$ such that the minimal $\TT$-equivariant lifts of any $z_v\cdot z_w\in J$ to $\widetilde J$ and $\widetilde{J}'$ differ modulo $\mfm^{\ell+1}$ by
\[\sum_i \bt^{\alpha_i}\psi_i(z_v\cdot z_w)\]
where $\alpha_i=(\alpha_{i1},\ldots,\alpha_{ip})\in \ZZ^p$ satisfies $\sum_j\alpha_{ij}=\ell$.
The $\TT$-equivariance imposes the condition that each $\psi_i$ has degree contained in $\SG$, and thus  belongs to $T^1(\K*\PS(\W))^H$ (and not $D$).

We claim that all $\psi_i$ must be zero.
In the case $\ell=1$, the $\TT$-equivariance imposes that any non-zero $\psi_i$ must belong to $T^1(\K*\PS(\W))^H$. But these $\psi_i$ must also be zero, since otherwise the characteristic maps would not agree.

We may thus assume that $\ell\geq 2$.
Consider first some $0\neq \psi \in T^1(\K*\PS(\W))$. Then $\psi(z_v\cdot z_w)\neq 0$ for exactly one pair exchangeable pair $\{v,w\}$, and adding $\bt^\alpha\cdot \psi(z_v\cdot z_w)$ to the exchange relation will introduce some non-zero multiple of a monomial $\bt^\alpha\cdot z^\ba$. Moreover, it follows from Theorem \ref{thm:t1} and the assumption $\ell\geq 2$ that $\bt^\alpha \cdot z^\ba$ cannot be one of the monomials appearing in the exchange relation of $\A^{\univ}$ for $v,w$. 
Taking any $\psi_i=\psi$ would thus violate the assumption that $\cY$ is $k$-exchange minimal.
We conclude that $\cY$ and $\cY'$ agree modulo $\mfm^{\ell+1}$, and thus by induction that $\cY=\cY'$ modulo $\mfm^{k+1}$.

It follows that $\cY'$ is exchange minimal, and any other exchange minimal deformation is canonically isomorphic to it.
The claims in the first statement of the theorem follow. The claims in the second statement follow from Proposition \ref{prop:t1p}, and the fact that a deformation over a smooth base whose characteristic map is an isomorphism is semiuniversal.
\end{proof}
\begin{ex}[Continuation of Running Example \ref{ex:r9}]\label{ex:r11}
	We consider the cluster algebra $\A$ from Example \ref{ex:r1} for the final time. We have verified that property T1 (Example \ref{ex:r9}) is satisfied. We may thus apply Theorem \ref{thm:versal} and reconstruct the universal cluster algebra $\A^\univ$ (Example \ref{ex:r4}) as the canonical exchange minimal $H$-invariant semiuniversal deformation of $\spec S_{\K*\PS(\W)}$. See \S\ref{app:run} for an explicit \emph{Macaulay2} computation doing this.
\end{ex}

\begin{rem}
\label{rem:equiv_map}
Assume that $\A$ is a positively graded cluster algebra of full rank and finite cluster type.
By Theorem \ref{thm:hilb}, it follows that the $\TT$-action on the family $\spec \A^{\univ}\to \Aff^p$ has a standard action on the base $\Aff^p$ (cf. Remark \ref{rem:standardaction}).
Assume further that property T1 holds. Then we don't need to concern ourselves with the notion of exchange minimality and
we obtain a slightly stronger result than Theorem \ref{thm:versal}: there is a canonical embedded isomorphism between the family  $\spec \A^{\univ}\to \Aff^p$ and \emph{any} embedded $\TT$-equivariant $H$-invariant semiuniversal deformation $\pi$ of $Y$. 

Indeed, arguing as in the proof of Theorem \ref{thm:versal}, the correction terms $\psi_i$ must have degrees $\bc\in\ZZ^{p+q}$ for which simultaneously the degree $\bc$-piece of $T^1(\K*\PS(\W))^H$ is non-zero, and $\bc$ is a sum of $\ell>1$ non-zero $T^1$ degrees. But since the base $\Aff^p$ has a standard action, degrees of the graded pieces of $T^1(\K*\PS(\W))^H$ are linearly independent, so this is impossible.
We note that it is always possible to arrive in this situation by adding additional frozen variables.
\end{rem}

We obtain the following Corollary of Theorem \ref{thm:versal}:
\begin{cor}\label{cor:t1}
Let $\A$ be a positively graded cluster algebra of full rank and finite cluster type. Assume further that $\A$ satisfies the T1 property. Then $T^1(\A)^H=0$, that is, $\spec \A$ has no $H$-invariant deformations.
\end{cor}
\begin{proof}
	By Theorem \ref{thm:versal}, the family $\spec \A^{\univ}\to\Aff^p$ is an $H$-invariant semiuniversal deformation for $\A$. By Theorem \ref{thm:hilb}\ref{part:h4}, the $\TT$-action on $\Aff^p$ has a dense orbit, so the family $\spec \A^{\univ}\to\Aff^p$ is isomorphic to the trivial family on a dense open set. Any of the corresponding fibers is just isomorphic to $\spec \A$.

	By openness of versality \cite{versalopen}, it follows that the family $\spec \A^{\univ}\to\Aff^p$ contains a semiuniversal family for the fiber $\spec \A$ over $\bt=1$. But since the family is isomorphic to the trivial family over $(\KK^*)^p\subseteq \Aff^p$, it follows that $T^1(\A)^H=0$.
\end{proof}

\bibliographystyle{alpha}
\bibliography{notes.bib}

\newcommand{\etalchar}[1]{$^{#1}$}
\begin{thebibliography}{BMR{\etalchar{+}}06}

\bibitem[ABS08]{slices}
Ibrahim Assem, Thomas Br\"{u}stle, and Ralf Schiffler.
\newblock Cluster-tilted algebras and slices.
\newblock {\em J. Algebra}, 319(8):3464--3479, 2008.

\bibitem[AC04]{ac1}
Klaus Altmann and Jan~Arthur Christophersen.
\newblock Cotangent cohomology of {S}tanley-{R}eisner rings.
\newblock {\em Manuscripta Math.}, 115(3):361--378, 2004.

\bibitem[AC10]{ac2}
Klaus Altmann and Jan~Arthur Christophersen.
\newblock Deforming {S}tanley-{R}eisner schemes.
\newblock {\em Math. Ann.}, 348(3):513--537, 2010.

\bibitem[ADS14]{ADS}
Ibrahim Assem, Gr\'{e}goire Dupont, and Ralf Schiffler.
\newblock On a category of cluster algebras.
\newblock {\em J. Pure Appl. Algebra}, 218(3):553--582, 2014.

\bibitem[And71]{cotangent2}
Michel Andr\'{e}.
\newblock Homologie des alg\`ebres commutatives.
\newblock In {\em Actes du {C}ongr\`es {I}nternational des {M}ath\'{e}maticiens
  ({N}ice, 1970), {T}ome 1}, pages 301--308. 1971.

\bibitem[Art74]{versalopen}
M.~Artin.
\newblock Versal deformations and algebraic stacks.
\newblock {\em Invent. Math.}, 27:165--189, 1974.

\bibitem[BFMS21]{faber}
Angelica Benito, Eleonore Faber, Hussein Mourtada, and Bernd Schober.
\newblock Classification of singularities of cluster algebras of finite type.
\newblock arXiv:2106.10290v1, 2021.

\bibitem[BFZ05]{BFZ_clustersIII}
Arkady Berenstein, Sergey Fomin, and Andrei Zelevinsky.
\newblock Cluster algebras. {III}. {U}pper bounds and double {B}ruhat cells.
\newblock {\em Duke Math. J.}, 126(1):1--52, 2005.

\bibitem[Bir09]{polyhedra}
Ren\'{e} Birkner.
\newblock Polyhedra: a package for computations with convex polyhedral objects.
\newblock {\em J. Softw. Algebra Geom.}, 1:11--15, 2009.

\bibitem[BIRS09]{BIRS09}
A.~B. Buan, O.~Iyama, I.~Reiten, and J.~Scott.
\newblock Cluster structures for 2-{C}alabi-{Y}au categories and unipotent
  groups.
\newblock {\em Compos. Math.}, 145(4):1035--1079, 2009.

\bibitem[BMNC21]{BMN}
Lara Bossinger, Fatemeh Mohammadi, and Alfredo N\'{a}jera~Ch\'{a}vez.
\newblock Families of {G}r\"{o}bner degenerations, {G}rassmannians and
  universal cluster algebras.
\newblock {\em SIGMA Symmetry Integrability Geom. Methods Appl.}, 17:Paper No.
  059, 46, 2021.

\bibitem[BMR{\etalchar{+}}06]{BMRRT}
Aslak~Bakke Buan, Robert Marsh, Markus Reineke, Idun Reiten, and Gordana
  Todorov.
\newblock Tilting theory and cluster combinatorics.
\newblock {\em Adv. Math.}, 204(2):572--618, 2006.

\bibitem[CI14]{srdegen}
Jan~Arthur Christophersen and Nathan~Owen Ilten.
\newblock Degenerations to unobstructed {F}ano {S}tanley-{R}eisner schemes.
\newblock {\em Math. Z.}, 278(1-2):131--148, 2014.

\bibitem[CI16]{crelle}
Jan~Arthur Christophersen and Nathan Ilten.
\newblock Hilbert schemes and toric degenerations for low degree {F}ano
  threefolds.
\newblock {\em J. Reine Angew. Math.}, 717:77--100, 2016.

\bibitem[CK08]{CK}
Philippe Caldero and Bernhard Keller.
\newblock From triangulated categories to cluster algebras.
\newblock {\em Invent. Math.}, 172(1):169--211, 2008.

\bibitem[CL20]{denom}
Peigen Cao and Fang Li.
\newblock The enough {$g$}-pairs property and denominator vectors of cluster
  algebras.
\newblock {\em Math. Ann.}, 377(3-4):1547--1572, 2020.

\bibitem[CLS11]{CLS}
David~A. Cox, John~B. Little, and Henry~K. Schenck.
\newblock {\em Toric varieties}, volume 124 of {\em Graduate Studies in
  Mathematics}.
\newblock American Mathematical Society, Providence, RI, 2011.

\bibitem[Dem11]{Demonet11}
Laurent Demonet.
\newblock Categorification of skew-symmetrizable cluster algebras.
\newblock {\em Algebr. Represent. Theory}, 14(6):1087--1162, 2011.

\bibitem[DWZ10]{DWZ}
Harm Derksen, Jerzy Weyman, and Andrei Zelevinsky.
\newblock Quivers with potentials and their representations {II}: applications
  to cluster algebras.
\newblock {\em J. Amer. Math. Soc.}, 23(3):749--790, 2010.

\bibitem[FK10]{Fu_Kel}
Changjian Fu and Bernhard Keller.
\newblock On cluster algebras with coefficients and 2-{C}alabi-{Y}au
  categories.
\newblock {\em Trans. Amer. Math. Soc.}, 362(2):859--895, 2010.

\bibitem[FWZ16]{FWZ_chapter1}
Sergey Fomin, Lauren Williams, and Andrei Zelevinsky.
\newblock Introduction to cluster algebras chapters 1--3.
\newblock arXiv:1608.05735v3, 2016.

\bibitem[FWZ17]{FWZ_chapter4}
Sergey Fomin, Lauren Williams, and Andrei Zelevinsky.
\newblock Introduction to cluster algebras chapters 4--5.
\newblock arXiv:1707.07190v4, 2017.

\bibitem[FWZ21]{FWZ_chapter6}
Sergey Fomin, Lauren Williams, and Andrei Zelevinsky.
\newblock Introduction to cluster algebras chapter 6.
\newblock arXiv:2008.09189v2, 2021.

\bibitem[FZ02]{FZ_clustersI}
Sergey Fomin and Andrei Zelevinsky.
\newblock Cluster algebras. {I}. {F}oundations.
\newblock {\em J. Amer. Math. Soc.}, 15(2):497--529, 2002.

\bibitem[FZ03a]{FZ_clustersII}
Sergey Fomin and Andrei Zelevinsky.
\newblock Cluster algebras. {II}. {F}inite type classification.
\newblock {\em Invent. Math.}, 154(1):63--121, 2003.

\bibitem[FZ03b]{CDM}
Sergey Fomin and Andrei Zelevinsky.
\newblock Cluster algebras: notes for the {CDM}-03 conference.
\newblock In {\em Current developments in mathematics, 2003}, pages 1--34. Int.
  Press, Somerville, MA, 2003.

\bibitem[FZ03c]{FZ_assoc}
Sergey Fomin and Andrei Zelevinsky.
\newblock {$Y$}-systems and generalized associahedra.
\newblock {\em Ann. of Math. (2)}, 158(3):977--1018, 2003.

\bibitem[FZ07]{FZ_clustersIV}
Sergey Fomin and Andrei Zelevinsky.
\newblock Cluster algebras. {IV}. {C}oefficients.
\newblock {\em Compos. Math.}, 143(1):112--164, 2007.

\bibitem[GHK15]{GHK}
Mark Gross, Paul Hacking, and Sean Keel.
\newblock Birational geometry of cluster algebras.
\newblock {\em Algebr. Geom.}, 2(2):137--175, 2015.

\bibitem[GHKK18]{GHKK}
Mark Gross, Paul Hacking, Sean Keel, and Maxim Kontsevich.
\newblock Canonical bases for cluster algebras.
\newblock {\em J. Amer. Math. Soc.}, 31(2):497--608, 2018.

\bibitem[GLS10]{GLS11}
Christof Geiss, Bernard Leclerc, and Jan Schr\"{o}er.
\newblock Cluster algebra structures and semicanonicalbases for unipotent
  groups.
\newblock arXiv:0703039v4, 2010.

\bibitem[GP18]{GP18}
Jan~E. Grabowski and Matthew Pressland.
\newblock Graded {F}robenius cluster categories.
\newblock {\em Doc. Math.}, 23:49--76, 2018.

\bibitem[Gra15]{Gra}
Jan~E. Grabowski.
\newblock Graded cluster algebras.
\newblock {\em J. Algebraic Combin.}, 42(4):1111--1134, 2015.

\bibitem[Hap88]{Hap}
Dieter Happel.
\newblock {\em Triangulated categories in the representation theory of
  finite-dimensional algebras}, volume 119 of {\em London Mathematical Society
  Lecture Note Series}.
\newblock Cambridge University Press, Cambridge, 1988.

\bibitem[Har10]{deftheory}
Robin Hartshorne.
\newblock {\em Deformation theory}, volume 257 of {\em Graduate Texts in
  Mathematics}.
\newblock Springer, New York, 2010.

\bibitem[HS04]{multigraded}
Mark Haiman and Bernd Sturmfels.
\newblock Multigraded {H}ilbert schemes.
\newblock {\em J. Algebraic Geom.}, 13(4):725--769, 2004.

\bibitem[Ilt12]{versal}
Nathan~Owen Ilten.
\newblock Versal deformations and local {H}ilbert schemes.
\newblock {\em J. Softw. Algebra Geom.}, 4:12--16, 2012.

\bibitem[INCT21]{oldversion}
Nathan Ilten, Alfredo Nájera~Chávez, and Hipolito Treffinger.
\newblock Deformation theory for finite cluster complexes.
\newblock arXiv:2111.02566v1, 2021.

\bibitem[IS25]{karolyn}
Nathan Ilten and Karolyn So.
\newblock Gr\"obner cones for finite type cluster algebras.
\newblock arXiv:2501.07065, 2025.

\bibitem[Kel05]{Kel_triang}
Bernhard Keller.
\newblock On triangulated orbit categories.
\newblock {\em Doc. Math.}, 10:551--581, 2005.

\bibitem[Kel10]{Kel_survay}
Bernhard Keller.
\newblock Cluster algebras, quiver representations and triangulated categories.
\newblock In {\em Triangulated categories}, volume 375 of {\em London Math.
  Soc. Lecture Note Ser.}, pages 76--160. Cambridge Univ. Press, Cambridge,
  2010.

\bibitem[Kel12]{Kel12}
Bernhard Keller.
\newblock Cluster algebras and derived categories.
\newblock In {\em Derived categories in algebraic geometry}, EMS Ser. Congr.
  Rep., pages 123--183. Eur. Math. Soc., Z\"{u}rich, 2012.

\bibitem[MR88]{gfan}
Teo Mora and Lorenzo Robbiano.
\newblock The {G}r\"{o}bner fan of an ideal.
\newblock {\em J. Symbolic Comput.}, 6(2-3):183--208, 1988.
\newblock Computational aspects of commutative algebra.

\bibitem[MRZ18]{CM}
Greg Muller, Jenna Rajchgot, and Bradley Zykoski.
\newblock Lower bound cluster algebras: presentations, {C}ohen-{M}acaulayness,
  and normality.
\newblock {\em Algebr. Comb.}, 1(1):95--114, 2018.

\bibitem[Nak21a]{Nak_scat}
Tomoki Nakanishi.
\newblock Cluster algebras and scattering diagrams, part {II}. {C}luster
  patterns and scattering diagrams.
\newblock arXiv:2103.16309v3, 2021.

\bibitem[Nak21b]{Nak}
Tomoki Nakanishi.
\newblock Synchronicity phenomenon in cluster patterns.
\newblock {\em J. Lond. Math. Soc. (2)}, 103(3):1120--1152, 2021.

\bibitem[NZ12]{NZ}
Tomoki Nakanishi and Andrei Zelevinsky.
\newblock On tropical dualities in cluster algebras.
\newblock In {\em Algebraic groups and quantum groups}, volume 565 of {\em
  Contemp. Math.}, pages 217--226. Amer. Math. Soc., Providence, RI, 2012.

\bibitem[Rea14]{reading}
Nathan Reading.
\newblock Universal geometric cluster algebras.
\newblock {\em Math. Z.}, 277(1-2):499--547, 2014.

\bibitem[Rim80]{rim}
Dock~S. Rim.
\newblock Equivariant {$G$}-structure on versal deformations.
\newblock {\em Trans. Amer. Math. Soc.}, 257(1):217--226, 1980.

\bibitem[Sco06]{scott}
Joshua~S. Scott.
\newblock Grassmannians and cluster algebras.
\newblock {\em Proc. London Math. Soc. (3)}, 92(2):345--380, 2006.

\bibitem[Sta96]{stanley}
Richard~P. Stanley.
\newblock {\em Combinatorics and commutative algebra}, volume~41 of {\em
  Progress in Mathematics}.
\newblock Birkh\"{a}user Boston, Inc., Boston, MA, second edition, 1996.

\bibitem[Ste95]{stevens}
Jan Stevens.
\newblock Computing versal deformations.
\newblock {\em Experiment. Math.}, 4(2):129--144, 1995.

\bibitem[Stu96]{grob}
Bernd Sturmfels.
\newblock {\em Gr\"{o}bner bases and convex polytopes}, volume~8 of {\em
  University Lecture Series}.
\newblock American Mathematical Society, Providence, RI, 1996.

\end{thebibliography}

\appendix
\section{Macaulay2 Code for Examples}\label{app:code}
\subsection{Example from \S\ref{ex:g2}}\label{code:g2}~ ~
\begin{footnotesize}
\begin{verbatim}
i1 : loadPackage "VersalDeformations"

o1 = VersalDeformations

o1 : Package

i2 : R=QQ[x_1,x_2,x_3,x_4,x_5,x_6,x_7,x_8,Degrees=>entries id_(ZZ^8)];

i3 : J=ideal {x_1*x_3,x_3*x_5,x_5*x_7,x_7*x_1,x_2*x_4,x_4*x_6,x_6*x_8,
         x_8*x_2,x_1*x_4,x_2*x_5,x_3*x_6, x_4*x_7,x_5*x_8,x_6*x_1,x_7*x_2,
         x_8*x_3,x_1*x_5,x_2*x_6,x_3*x_7,x_4*x_8}; --SR ideal for cluster complex

o3 : Ideal of R

i4 : L={{-1,1,-1,0,0,0,0,0},{0,0,-1,1,-1,0,0,0},
      {0,0,0,0,-1,1,-1,0},{-1,0,0,0,0,0,-1,1},
     {0,-1,3,-1,0,0,0,0},{0,0,0,-1,3,-1,0,0},
      {0,0,0,0,0,-1,3,-1},{3,-1,0,0,0,0,0,-1}}; -- degrees of characteristic map

i5 : t1=matrix {apply(L,i->CT^1(i,J))}; -- image of characteristic map

             20       8
o5 : Matrix R   <--- R

i6 : t2=CT^2(J); -- obstruction space 

             64       12
o6 : Matrix R   <--- R

i7 : (F,R,G,C)=versalDeformation(gens J,-t1,t2);

i8 : matrix entries transpose sum F -- generators of lifted ideal

o8 = | -t_2t_3^2t_4t_6t_7-x_2t_1+x_1x_3                                    |
     | -t_1t_3t_4^2t_7t_8-x_4t_2+x_3x_5                                    |
     | -t_1^2t_2t_4t_5t_8-x_6t_3+x_5x_7                                    |
     | -t_1t_2^2t_3t_5t_6-x_8t_4+x_1x_7                                    |
     | -t_3^3t_4^3t_6t_7^2t_8-x_3^3t_5+x_2x_4                              |
     | -t_1^3t_4^3t_5t_7t_8^2-x_5^3t_6+x_4x_6                              |
     | -t_1^3t_2^3t_5^2t_6t_8-x_7^3t_7+x_6x_8                              |
     | -t_2^3t_3^3t_5t_6^2t_7-x_1^3t_8+x_2x_8                              |
     | -x_5t_3^2t_4t_6t_7-x_3^2t_1t_5+x_1x_4                               |
     | -x_1t_3t_4^2t_7t_8-x_3^2t_2t_5+x_2x_5                               |
     | -x_7t_1t_4^2t_7t_8-x_5^2t_2t_6+x_3x_6                               |
     | -x_3t_1^2t_4t_5t_8-x_5^2t_3t_6+x_4x_7                               |
     | -x_1t_1^2t_2t_5t_8-x_7^2t_3t_7+x_5x_8                               |
     | -x_5t_1t_2^2t_5t_6-x_7^2t_4t_7+x_1x_6                               |
     | -x_3t_2^2t_3t_5t_6-x_1^2t_4t_8+x_2x_7                               |
     | -x_7t_2t_3^2t_6t_7-x_1^2t_1t_8+x_3x_8                               |
     | -x_3t_1t_2t_5-x_7t_3t_4t_7+x_1x_5                                   |
     | -3t_1t_2^2t_3t_4^2t_5t_6t_7t_8-x_4t_2^3t_5t_6-x_8t_4^3t_7t_8+x_2x_6 |
     | -x_5t_2t_3t_6-x_1t_1t_4t_8+x_3x_7                                   |
     | -3t_1^2t_2t_3^2t_4t_5t_6t_7t_8-x_6t_3^3t_6t_7-x_2t_1^3t_5t_8+x_4x_8 |

\end{verbatim}

\end{footnotesize}

\subsection{Running Example \ref{ex:r11} of type $A_2$ }\label{app:run}~~

\begin{footnotesize}
\begin{verbatim}

i1 : loadPackage "VersalDeformations"

o1 = VersalDeformations

o1 : Package

i2 : R=QQ[x_25,x_13,x_24,x_35,x_14,s_1,s_2,s_3,
         Degrees=>{{1,1,-1},{-1,1,1},{2,-1,0},{-1,2,0},{1,-1,1},{1,0,0},{0,1,0},{0,0,1}}];

i3 : J=ideal{x_24*x_35,x_35*x_14,x_14*x_25,x_25*x_13,x_13*x_24}; -- SR ideal

o3 : Ideal of R

i4 : (F,R,G,C)=versalDeformation(gens J,-CT^1({0,0,0},J),CT^2({0,0,0},J));

i5 : matrix entries transpose sum F -- generators of lifted ideal

o5 = | -s_1s_2t_2t_5-x_25s_3t_1+x_24x_35 |
     | -s_2s_3t_1t_3-x_13s_1t_2+x_35x_14 |
     | -s_1^2t_2t_4-x_24s_2t_3+x_25x_14  |
     | -s_2^2t_3t_5-x_35s_1t_4+x_25x_13  |
     | -s_1s_3t_1t_4-x_14s_2t_5+x_13x_24 |

\end{verbatim}

\end{footnotesize}

\end{document}